\documentclass[reqno,12pt]{amsart}

\usepackage{amsfonts}
\usepackage{amsmath}
\usepackage{amsthm}
\usepackage{amssymb}
\usepackage{enumerate}
\usepackage[inline]{enumitem}
\usepackage{moreenum}
\usepackage[greek, english]{babel}
\usepackage{multicol}
\usepackage{mathtools}
\usepackage{tikz-cd}
\usepackage{multirow} 
\usepackage{geometry}
\usepackage{pdflscape} 
\usepackage{comment} 
\usepackage{makecell} 
\usepackage[colorlinks=true, linkcolor=blue]{hyperref}

\numberwithin{equation}{section}
\theoremstyle{plain}
\newtheorem{theorem}[equation]{Theorem}
\newtheorem*{theorem*}{Theorem}
\newtheorem*{proposition*}{Proposition}

\newtheorem{corollary}[equation]{Corollary}
\newtheorem{lemma}[equation]{Lemma}

\newtheorem*{conjecture*}{Conjecture}
\newtheorem{proposition}[equation]{Proposition}

\theoremstyle{definition}
\newtheorem{definition}[equation]{Definition}
\newtheorem*{definition*}{Definition}

\newtheorem*{remark*}{Remark}

\newtheorem{step}{Step}
\newtheorem{substep}{Step}[step]

\newcommand\numberthis{\addtocounter{equation}{1}\tag{\theequation}} 

\newcommand{\RR}{\mathbb{R}}
\newcommand{\CC}{\mathbb{C}}
\newcommand{\QQ}{\mathbb{Q}}
\newcommand{\ZZ}{\mathbb{Z}}

\newcommand{\PP}{\mathbb{P}}
\newcommand{\BB}{\mathbb{B}}
\newcommand{\OO}{\mathcal{O}}
\newcommand{\SiegelA}{\mathcal{A}}

\newcommand{\disc}{\mathrm{disc}}
\renewcommand{\mod}[1]{\mathrm{\ mod\ } #1}

\newcommand{\Gal}{\mathrm{Gal}}
\newcommand{\PSL}{\mathrm{PSL}}

\newcommand{\SL}{\mathrm{SL}}
\newcommand{\GL}{\mathrm{GL}}
\newcommand{\SU}{\mathrm{SU}}

\DeclareMathOperator{\PU}{PU}

\renewcommand{\Im}[1]{\mathrm{Im\,} #1}
\renewcommand{\Re}[1]{\mathrm{Re\,} #1}

\newcommand{\vol}{\mathrm{vol}}

\DeclareMathOperator{\csch}{csch}

\newcommand{\trd}{\mathrm{trd}}
\newcommand{\nrd}{\mathrm{nrd}}
\newcommand{\discrd}{\mathrm{discrd}}
\newcommand{\mat}{\mathrm{Mat}}
\newcommand{\UM}{\mathrm{UM}}
\newcommand{\HH}{\mathbb{H}}

\newcommand{\myquat}{Q}

\newcommand{\norm}{\mathrm{N}}
\newcommand{\trace}{\mathrm{tr}}

\DeclareMathOperator{\pp}{\mathfrak{p}}
\DeclareMathOperator{\ppp}{\mathfrak{P}}

\newcommand{\mult}{\mathrm{mult}}

\DeclareMathOperator{\FixLocus}{Fix}
\DeclareMathOperator{\axislox}{H}
\DeclareMathOperator{\taOp}{ta}
\newcommand{\ta}[1]{
  \taOp\ifstrempty{#1}{}{(#1)}
}
\DeclareMathOperator{\cont}{cont}
\DeclareMathOperator{\different}{diff}

\begin{document}

\title  {Low Genus Curves on Compact Shimura Surfaces}
\author{Zhehao Li}
\date{}
\maketitle

\setcounter{tocdepth}{3}
\tableofcontents

\section*{Notation and Convention}

For a number field $F$, $\trace=\trace_{F/\QQ}$ and $\norm=\norm_{F/\QQ}$ are the trace and norm maps over $\QQ$ when the subscript is omitted. $\OO_F$ is its ring of integers. A rational integer is called a fundamental discriminant if it is the discriminant of a quadratic number field. We use superscript to mean applying a real or complex embedding $\sigma$ of $F$, i.e. $M^{\sigma}:=\sigma(M)$ for $M$ an element in $F$ or a vector or matrix over $F$. $\zeta_N$ is a primitive $N$-th root of unity, and for $N=3,4$, $\omega:=\zeta_3$ and $i:=\zeta_4$. 

$\mat_n(R)$ are all $n\times n$ matrices with entries in $R$. For a matrix $M$, $M^*$ is its adjugate: for $M=\begin{pmatrix} a & b \\ c & d \end{pmatrix}$, $M^*=\begin{pmatrix} d & -b \\ -c & a \end{pmatrix}$. For a matrix group $G$, $PG=G/Z(G)$ where $Z(G)$ is its center. For a matrix $M$, $\prescript{t}{}{M}$ is its transpose, and when it is over a field with complex conjugation, $M^H:=\prescript{t}{}{\bar{M}}$ is its conjugate transpose. We may use the same notation for an element in a projectivized space and a representative of it.

Some notations to be defined that will be used frequently include: $\ta{\cdot,\cdot}$ (Page~\pageref{page:ta}), $\UM_n(h)$ (Page~\pageref{page:um2}), and ball or tubular neighborhoods $B_{x,r}$ (Page~\pageref{page:ball_neighborhood}), $W_{L,r}$ (Page~\pageref{page:line_neighborhood}), $W_{B,r}$ (Page~\pageref{page:curve_neighborhood}).

\section{Introduction}\label{sec:introduction}

Uniform boundedness conjectures are crucial topics in arithmetic geometry. Many of them are inspired by the celebrated theorem of Mazur \cite{mazur1977modular,mazur1978rational} about uniform boundedness on torsion rational points on elliptic curves:
\begin{theorem*}[Mazur]
	For any elliptic curve $E/\QQ$, $|E(\QQ)_{\mathrm{tor}}|\le 16$. 
\end{theorem*}
Mazur's theorem can be extended in different directions. For curves of genus greater than 1, uniform boundedness conjectures on numbers of $K$-rational points on such curves are refinements of Faltings' finiteness theorem, and a weaker version has been proved in \cite{dimitrov2021uniformity}. For torsion points on elliptic curves over a general number field $K$, some special cases are known, including general quadratic and cubic number fields (see e.g. Mazur's survey \cite{balakrishnan2024oggstorsionconjecturefifty} for detail). However, for the uniform boundedness for torsion points of higher dimensional abelian varieties, little is known. 

In this paper, we focus on the uniform boundedness on the endomorphism types of abelian varieties: it is expected that there are only finitely many possible isomorphism classes of endomorphism rings $\operatorname{End}(A)$ of abelian varieties of a fixed dimension $g$ over a fixed number field $K$. 

Mazur's proof boils down to understanding rational points on the corresponding moduli spaces of elliptic curves, i.e. modular curves with level structures. Shimura varieties are deep generalizations of modular curves, which include moduli spaces of abelian varieties and moduli spaces of other important geometric objects such as K3 surfaces. Uniform boundedness conjectures for certain types of abelian varieties are related to the geometry on the corresponding Shimura varieties. This paper is dedicated to understanding the geometry of certain types of 2-dimensional Shimura varieties, and it can be interpreted as a characteristic 0 function field version of uniform boundedness on the endomorphism types of certain abelian varieties. 

Besides products of modular curves, the next fundamental Shimura surfaces are Hilbert modular surfaces and Picard modular surfaces. They are constructed as ball quotients of $\BB^1\times\BB^1$ and $\BB^2$ by certain unitary groups with entries in the ring of integers of a real quadratic field and an imaginary quadratic field respectively. Hilbert modular surfaces are important in the picture of endomorphism types of abelian surfaces: the countably many Hilbert modular surfaces inside the moduli space of abelian surfaces parametrize each type of abelian surfaces with real multiplication, on each of them there are countably many ``special curves'' parametrizing abelian surfaces with quaternionic multiplication, and finally abelian surfaces with complex multiplication are points at the intersection of special curves.

In fact, the geometry of special subvarieties of Shimura varieties is an interesting subject itself. Vaguely speaking, special subvarieties of Shimura varieties are ``Shimura subvarieties'', i.e. they are algebraic subvarieties which are also compatible with the complex quotient space structure of the Shimura variety. Besides their importance in moduli problems, they also show surprising phenomena in geometry. One example is the Andr\'{e}-Oort conjecture, which was proved recently in \cite{pila2021canonical}, saying that the Zariski closure of any set of special points in a Shimura variety is a finite union of special subvarieties. 

We can interpret the uniform boundedness conjecture on the endomorphism types of abelian varieties using $\SiegelA_g$ the moduli space of abelian varieties of dimension $g$ and its special subvarieties. For a $K$-point $p\in\SiegelA_g(K)$, consider $S_p$ the irreducible special subvariety containing $p$ such that the endomorphism ring corresponding to $p$ is isomorphic to the endomorphism ring corresponding to a generic member of $S_p$. Fix a number field $K$ and $g\in\ZZ_{> 0}$. We consider the special closure $S$ of $\SiegelA_g(K)$, i.e. $S=\{S_p:p\in \SiegelA_g(K)\}$. Then $S$ is expected to be a finite set. 

In this paper, we will focus on the geometry of compact variants of Hilbert modular surfaces and Picard modular surfaces. Working with the compact variants can simplify the issues related to the hyperbolic geometry near the boundary, but the arithmetic will become more complicated. In fact, many of our proofs about the geometry of special subvarieties of the compact variants also fit into the usual noncompact case. Although the moduli problem will become slightly more complicated than the usual noncompact version, our main theorems can still be interpreted as uniform boundedness on the endomorphism types of certain abelian varieties. We will describe the way to construct these compact ball quotients in Section~\ref{subsub:compact_arithmetic_quotients}. 

The proof uses hyperbolic geometry and the method of volume estimates, which has been used in, for example, \cite{hwang2002volumes,hwang2012injectivity,bakker2013frey,bakker2018geometric,bakker2019ax,memarian2023volumes} to prove interesting results in algebraic geometry. In contrast to existing literature, we focus on the case with trivial level structure. This requires a better understanding of the geometry of quotient singularities, which is often omitted in existing literature as quotient singularities will disappear when the level structure grows. 

To apply the method of volume estimates in our situation, an essential input is the repulsion of special subvarieties and quotient singularities, i.e. how they will become far apart when the defining discriminant of the Shimura variety increases. We will see distinct repulsion phenomena for compact Hilbert and Picard modular surfaces. On the one hand, Picard modular surfaces illustrate diverse types of quotient singularities. Compared to Hilbert modular surfaces, notable new phenomena for Picard modular surfaces include the existence of 1-dimensional quotient singularities given by some special curves, and non-cyclic quotient singularities at the intersection of these curves. On the other hand, the repulsion on Hilbert modular surfaces is more sophisticated because of the existence of persisting special curves of low genus. 

Now we introduce the notation and setting for our main objects, and state the main theorems. Our main objects will be defined in detail in Section~\ref{sec:background}. 

For the compact variant of Picard modular surfaces, we have the following setting. 

\begin{equation}\label{notation_twoball}\tag{\textbf{CptPicard}}
  \parbox{\dimexpr\linewidth-8em}{%
    \strut
    Let $K$ be a totally real number field of degree $\ell$ and discriminant $\disc(K/\QQ)=D>0$. $F:=K(\sqrt{\beta})$ is a CM extension. Denote its complex embeddings $\sigma_1,\bar{\sigma}_1\cdots\sigma_\ell,\bar{\sigma}_\ell$. Take a Hermitian form $h(\cdot,\cdot)=\langle \cdot,\cdot\rangle$ represented by a matrix $J_h$ on $F^3$ with coefficients in $F$ such that $h^{\sigma_1}$ and $h^{\bar{\sigma}_1}$ are of signature $(2,1)$, and $h$ is definite under all other complex embeddings. Consider the arithmetic lattice $\Gamma:=\SU(h,\OO_F):=\SU(h)\cap \SL_3(\OO_F) $ and the arithmetic quotient $X:=\Gamma\backslash \BB^2 $. 
    \strut
  }
\end{equation}

We proved the following.

\begin{theorem}\label{thm:2_ball_geometry_ver}
 	Follow the setting in \ref{notation_twoball}. For any genus $g$, there exists $D_0(g) >0$ such that for $D>D_0(g)$, the corresponding $\Gamma\backslash\BB^2$ has no complex curves of genus $g$. 
 \end{theorem} 

For the compact variant of Hilbert modular surfaces, we have the following. 

\begin{equation}\label{notation_hms_smallsubfield}\tag{\textbf{CptHilbert}}
  \parbox{\dimexpr\linewidth-9em}{%
    \strut
    Fix $K_0$ a totally real number field of degree $\ell$ and discriminant $\disc(K_0/\QQ)=d>0$. Let $K=K_0(\sqrt{\alpha})$ be a real quadratic extension of $K_0$ with absolute discriminant $\disc(K/\QQ)=D>0$. Let $F_0:=K_0(\sqrt{\beta})$ be a CM extension for a fixed totally negative $\beta\in K_0$, so $F=K(\sqrt{\beta})$ is a bi-quadratic extension of $K_0$. Let $\sigma_j,\bar{\sigma}_j, \sigma'_j,\bar{\sigma}'_j$ be the two pairs of complex-conjugate embeddings with the same restriction to $K_0$ for $j=1,\cdots, \ell$. Take a Hermitian form $h(\cdot,\cdot)=\langle \cdot,\cdot\rangle$ represented by a matrix $J_h$ on $F^3$ with coefficients in $F_0$ such that $h^{\sigma_1}$, $h^{\bar{\sigma}_1}$, $h^{\sigma'_1}$, and $h^{\bar{\sigma}'_1}$ are of signature $(1,1)$, and under all other embeddings it is definite. Consider the arithmetic lattice $\Gamma:=\SU(h)\cap \SL_2(\OO_F) $, and the arithmetic quotient $X:=\Gamma\backslash (\BB^1\times \BB^1) $. 
    \strut
  }
\end{equation}

Note that $\beta\in K_0$ guarantees that the involution $\iota:K\to K, \sqrt{\alpha}\mapsto -\sqrt{\alpha}$ extends to $F$.

We proved the following.

\begin{theorem}\label{thm:hms_geometry_ver}
 	Follow the setting in \ref{notation_hms_smallsubfield}. For any genus $g$, there exist $D_0(g) >0$ such that for $D>D_0(g)$, there are at most finitely many genus $g$ curves on the corresponding $\Gamma\backslash(\BB^1\times\BB^1)$, and they must be special. 
 \end{theorem} 

\section{Background}\label{sec:background}

The classification of hermitian symmetric domains shows that Shimura surfaces must be arithmetic quotients of a complex hyperbolic 2-ball or two 1-balls. Section~\ref{sub:complex_hyperbolic_n_ball} is an introduction to complex hyperbolic $n$-balls $\BB^n$ and their isometries, and it ends with introducing a construction of arithmetic quotients of $n$-balls, i.e. the main objects in this paper. In Section~\ref{sub:geometry_of_bb_1_and_bb_2_}, we will collect useful facts and formulas about the differential geometry of $\BB^1$ and $\BB^2$. Then Section~\ref{sub:special_curves} discusses certain embeddings of $\BB^1$ into $\BB^2$ and $\BB^1\times \BB^1$.

\subsection{Complex hyperbolic balls and ball quotients}\label{sub:complex_hyperbolic_n_ball}

The target of this section is to introduce a way to construct compact quotients of balls and define the main objects in this paper. We will start from a review of the projective model of complex hyperbolic balls and its isometries (see \cite{goldman1999complex} or \cite{parker2003notes} for detail). 

\subsubsection{The projective model}

The unit ball in $\CC^n$ and the Bergman metric, which is a natural metric with constant negative holomorphic sectional curvature, give a model for complex hyperbolic $n$-space $\BB^n$. When $n=1$, this is the usual Poincar{\'e} unit disc model of the hyperbolic plane.

Let $V$ be $\CC^{n+1}$ equipped with a Hermitian form $\langle \cdot,\cdot\rangle$ of signature $(n,1)$. One example is the standard Hermitian form $h_{std}(u,v)=u_1\bar{v}_1+\cdots +u_n\bar{v}_n-u_{n+1}\bar{v}_{n+1}$. Since the norm $\langle u,u\rangle$ of $u\in V$ is real, we can define $V_-$ (resp. $V_0$, $V_+$) to be the subsets of $V\backslash \{0\}$ consisting of elements with negative (resp. zero, positive) norm, and write $u < 0$ (resp. $u>0$) if $u\in V_-$ (resp. $V_+$). Since $\langle \lambda u,\lambda u\rangle=|\lambda|^2 \langle u,u\rangle$ for any scalar $\lambda\in\CC^*$ and any $u\in V$, the projectivization map $\PP: \CC^{n+1}\to \PP_\CC^{n}$ preserves the sign of $\langle u,u\rangle$ and thus $\PP V_-$, $\PP V_0$, and $\PP V_+$ are well-defined. The projective model of complex hyperbolic $n$-space $\BB^n$ is defined to be the negative lines in $V$, i.e. $\BB^n := \PP V_-$. Thus, its boundary $\partial\BB^n = \PP V_0$. 

We can define a hyperbolic metric on $\BB^n$, which agrees with the Bergman metric for a general bounded symmetric domain. First we define the ``tance'' function for any $u,v\in V_- \cup V_+$: \label{page:ta}
\[
	\ta{u,v} := \frac{\langle u,v\rangle \langle v,u\rangle}{\langle u,u\rangle \langle v,v\rangle}\in \RR 
\]
It has the following useful properties:
\begin{enumerate}[label=(\roman*)]
	\item $\ta{u,v}=\ta{v,u}$ for any $u,v\in V_- \cup V_+$;
	\item $\ta{\lambda u,v}=\ta{u,v}$ for any $\lambda\in \CC^*$ and $u,v\in V_- \cup V_+$, so $\ta{\cdot,\cdot}$ is well-defined on $\PP V_- \cup \PP V_+$;
	\item $\ta{u,v} \ge 1$ if $u,v\in V_-$;  $\ta{u,v} \ge 0$ if $u,v\in V_+$; $\ta{u,v} \le 0$ if $u\in V_+$ and $v\in V_-$.
\end{enumerate}
Now the distance function $\rho_{\BB^n}(\cdot,\cdot)$ on $\BB^n$ is defined to be
\[\cosh^2\left(\frac{\rho_{\BB^n}(u,v)}{2} \right):= \ta{u,v} :=\frac{\langle u,v\rangle \langle v,u\rangle}{\langle u,u\rangle \langle v,v\rangle} \]
for $u,v\in \BB^n$. The holomorphic sectional curvature is normalized to be $-1$. 

This metric is invariant to the choice of the Hermitian form: for different choices of Hermitian forms, the complex hyperbolic $n$-balls are isometric. More specifically, since any Hermitian matrix $J$ can be unitarily diagonalized with real eigenvalues, we can find a matrix $T\in \GL_{n+1}(\CC)$ such that $J=T^HJ_{std}T$, and obtain an isometry called Cayley transform $(\CC^{n+1}, J)\to(\CC^{n+1}, J_{std}), \ \ u\mapsto Tu$, where $J_{std}=diag(1,\cdots,1,-1)$. 

When we consider $h_{std}$, any point $[u_1,\cdots,u_{n+1}]\in \PP V_-$ corresponds to a point $(u_1/u_{n+1},\cdots, u_n/u_{n+1})$ inside the open unit ball in $\CC^n$, so we also call $\BB^n$ the complex hyperbolic $n$-ball. Another example is the upper half plane $\HH$. It is a model of $\BB^1$ if we identify $z\in\HH$ as $(z,1)\in \CC^2$ and consider $J_h=\begin{pmatrix} 0 & i \\ -i & 0 \end{pmatrix}$. 

\subsubsection{Isometries}

Consider $V\cong\CC^{n+1}$ equipped with a Hermitian form $h(\cdot,\cdot)=\langle\cdot,\cdot\rangle$ of signature $(n,1)$. If we fix a basis of $V$, the Hermitian form is represented by a Hermitian matrix $J_h$, and linear transformations $M$ on $V$ preserving the Hermitian form satisfy:
\[ \langle Mu,Mv\rangle = v^HM^HJ_hMu =\langle u,v\rangle = v^H J_h u. \]
for any $u,v\in V$, so it is equivalent to requiring $M^HJ_hM=J_h$. Thus they define the unitary group:
\[U(h):=\{M\in \mat_{n+1}(\CC): \langle Mu,Mv\rangle =\langle u,v\rangle \ \ \forall u,v\in V \}\cong U(n,1). \]
The last isomorphism is induced by Cayley transform $T$: $U(h)\to U(h_{std})=:U(n,1), \ \ M\mapsto TMT^{-1}$. Since $\det(M^HJ_hM)=\det(J_h)$, we also have $|\det(M)|=1$. Since scalar multiplication $\lambda I_{n+1}$ for any $\lambda\in\CC^*$ commutes with any $M\in U(h)$, $U(h)$ is well-defined on $\PP V$, and the kernel of this action on $\PP V$ is $\{\lambda I: |\lambda|=1 \}\cong U(1)\subset U(h)$. Then, we define the projective unitary group
\[PU(h):=U(h)/\{\lambda\cdot id: |\lambda|=1 \}=\SU(h)/\{\lambda\cdot id\in \SU(h): |\lambda|=1 \} \]
whose element always has a representative with determinant 1 in $\SU(h)$. Indeed, $\BB^n$ is homogeneous since $PU(h)$ acts transitively on $\BB^n$ \cite[Lemma 3.1.3]{goldman1999complex}.

By Brouwer fixed point theorem, isometries of $\BB^n$ always have a fixed point in $\BB^n \cup \partial\BB^n$. We have the following trichotomy classifies the isometries of $\BB^n$:
\begin{definition}
A holomorphic complex hyperbolic isometry of $\BB^n$ is:
	\begin{enumerate}[label=(\roman*)]
		\item \emph{elliptic} if it fixes at least one point of $\BB^n$;
		\item \emph{loxodromic} or \emph{hyperbolic} if it is not elliptic and fixes two distinct points of $\partial\BB^n$;
		\item \emph{parabolic} if it has a unique fixed point on $\partial\BB^n$.
	\end{enumerate}
\end{definition}

\subsubsection{Compact arithmetic quotients}\label{subsub:compact_arithmetic_quotients}

The following construction generalizes the classical Picard modular surfaces and Hilbert modular surfaces.

Let $K$ be a totally real number field of degree $\ell$. Consider a CM extension $F:=K(\sqrt{\beta})$ for some totally negative $\beta\in K$. The embeddings come in complex conjugate pairs: $\sigma_1,\bar{\sigma}_1,\cdots,\sigma_\ell,\bar{\sigma}_\ell$. Then, take a Hermitian form $h(\cdot,\cdot)=\langle \cdot,\cdot\rangle$ in $n+1$ variables with coefficients in $F$, i.e.
\[\langle z,w\rangle := w^H J_h z, \ \ z,w\in F^{n+1} \]
for some Hermitian matrix $J_h=J_h^H\in \mat_{n+1}(F)$, such that for the first $m$ pairs of embeddings $(\sigma, \bar{\sigma})=(\sigma_1,\bar{\sigma}_1),\cdots,(\sigma_m,\bar{\sigma}_m)$, $h^{\sigma_1}$ and $h^{\bar{\sigma}_1}$ are of signature $(n,1)$, and for all other pairs of embeddings $(\sigma, \bar{\sigma})=(\sigma_{m+1},\bar{\sigma}_{m+1}),\cdots,(\sigma_\ell,\bar{\sigma}_\ell)$, $h^\sigma$ and $h^{\bar{\sigma}}$ are definite. Then under these first $m$ pairs of embeddings, they define an arithmetic lattice $\Gamma:= \PU(h,\OO_F):= \PU(h)\cap \PSL_{n+1}(\OO_F) =(\SU(h)\cap \SL_{n+1}(\OO_F))/\mu_{n+1}$, where $\mu_{n+1}$ is the group of $n$-th roots of unity. Let $\Gamma$ act on $(\BB^n)^m\times (0)^{\ell-m}$ under the $\ell$ pairs of embeddings and we can obtain a Shimura variety $X:=(\BB^n)^m / \Gamma$.

To see when the lattices constructed in this way are cocompact, we can apply the compactness criterion proved by Borel and Harish-Chandra:

\begin{theorem}[Godement's compactness criterion, Theorem 3 in \cite{borel1962arithmetic}]
  Let $G\subset \GL_n(\CC)$ be a reductive $\QQ$-group and $\Gamma$ be an arithmetic subgroup of $G$. Then $\Gamma\backslash G(\RR)$ is compact if and only if the identity component of $G$ has no nontrivial rational character and every element of $G_\QQ$ is semisimple. 
\end{theorem}

If $\ell \ne m$, since $\SU(n+1)$ is compact, by Godement's compactness criterion, $\sigma_{m+1}(\Gamma)$ has no parabolic elements, and thus $\Gamma$ has no parabolic elements. By \cite[Lemma 8.1.3]{maclachlan2013arithmetic}, its projection to the first $m$ components is also cocompact. Therefore, the Shimura variety we constructed is compact if $\ell \ne m$.

This construction can also describe the classical noncompact examples: the case $n=2,m=\ell=1$ corresponds to Picard modular surfaces, and the case $n=1,m=\ell=2$ corresponds to Hilbert modular surfaces. We will focus on two of the dimension 2 compact examples: the case $n=2,m=1, \ell>m$ is the compact variant of Picard modular surfaces, while the compact variant of Hilbert modular surfaces belongs to the case $n=1,m=2, \ell>m$.

\subsection{Geometry of $\BB^1$ and $\BB^2$}\label{sub:geometry_of_bb_1_and_bb_2_}

Here we will collect facts and formulas about the differential geometry of $\BB^1$ and $\BB^2$ to be used later in the proof. Let $(V, h)$ be the Hermitian vector space defining $\BB^n$ for $n=1,2$. 

For $M\in \SU(h)$, we have the following restriction on its eigenvalues, which is a consequence of $M^HJ_hM=J_h$. 
\begin{lemma}[Lemma 6.2.5 of \cite{goldman1999complex}]\label{lem:eigenval_conj_inverse}
	If $M\in \SU(h)$ and $\lambda_0$ is an eigenvalue of $M$, then $\bar{\lambda}_0^{-1}$ is also an eigenvalue of $M$. Moreover, if we denote $t:=\trace(M)$, then
	\begin{enumerate}[label=(\roman*)]
		\item When $n=1$, the characteristic polynomial of $M$ is $\lambda^2-t\lambda+1 $ and $t\in\RR$;
		\item When $n=2$, the characteristic polynomial of $M$ is $\lambda^3-t\lambda^2+\bar{t}\lambda-1 $. 
	\end{enumerate} 
\end{lemma}

We have the following classification of isometries on $\BB^2$. In the literature they are mostly proved for the case where $J_h=J_{std}$, but the statements also apply to the general case since they do not depend on the choice of basis. 

\begin{theorem}[Theorem 6.2.4 of \cite{goldman1999complex}]\label{thm:classification_of_isometries}
  Consider $\BB^2$ given by $(V,h)$, $M\in \SU(h)$.  
  \begin{enumerate}[label=(\roman*)]
    \item $M$ is loxodromic if and only if one of the eigenvalues $\lambda$ has $|\lambda|\ne 1$. In this case, if we denote $\lambda=re^{i\theta}$ for $r,\theta\in\RR$, then the three eigenvalues are $re^{i\theta}$, $r^{-1}e^{i\theta}$, $e^{-2i\theta}$. If they have eigenvectors $u,v,w$ respectively, then
    \[\langle u,u\rangle=\langle v,v\rangle=\langle u,w\rangle=\langle v,w\rangle=0, \ \ \langle w,w\rangle >0, \ \ \langle u,v\rangle\ne 0. \]
    \item $M$ is parabolic if and only if it has a repeated eigenvalue whose eigenspace is spanned by one null vector. 
    \item $M$ is elliptic if and only if it can be diagonalized and all the eigenvalues have norm 1. We have the following cases:
    \begin{enumerate}[label=(\alph*)]
      \item $M$ only has one fixed point in $\BB^2$. 
      \item $M$ is called boundary elliptic when the fixed points in $\BB^2$ form a complex line. In this case, this line corresponds to a 2-dimensional eigenspace spanned by both positive and negative vectors. In other words, we can find eigenvectors $u,v,w$ corresponding to eigenvalues $\lambda_1=\lambda_2\ne \lambda_3$ with $u<0$ and $v,w> 0$. 
      \item The last case is where $M$ corresponds to scalar multiplication by a third root of unity. 
    \end{enumerate}
  \end{enumerate}
\end{theorem}

When $M\in \SU(h)$ is diagonalizable, we can pick some eigenvectors as a basis of $V$. Using the orthogonal relations above we can obtain many useful identities. Before we state them, we summarize the notation we will use, which will also be used in the distance formulas.

\begin{definition}\label{def:ratio_notation}
  For $u,v,p,q\in V$,
  \begin{enumerate}[label=(\alph*)]
    \item Complex cross-ratio: $\displaystyle [u,v,p,q]:= \frac{\langle p,u\rangle\langle q,v\rangle}{\langle p,v\rangle\langle q,u\rangle}$;
    \item Goldman's $\eta$: $\displaystyle\eta(p,q,u):=[u,q,p,u]=\frac{\langle p,u\rangle\langle u,q\rangle}{\langle p,q\rangle\langle u,u\rangle}$;
    \item $\displaystyle \ta{u,v}=[v,u,u,v]=\frac{\langle u,v\rangle\langle v,u\rangle}{\langle u,u\rangle\langle v,v\rangle}$.
  \end{enumerate}
\end{definition}

Note that these ratios are invariant if we multiply any of the input by a scalar in $\CC^\times$. To motivate the definition of the complex cross-ratio, note that for $x:=[v_1,v_2,v_3,v_4]$, by definition $[v_2,v_1,v_4,v_3]=x$, $[v_3,v_4,v_1,v_2]=\bar{x}$, and $[v_2,v_1,v_3,v_4]=1/x$, so in general among the $4!=24$ ways to order the four inputs, the degree of freedom of these complex cross-ratios is just $24/2^3=3$. In the $\BB^1$ case, by Equation~\ref{eq:lox_sum_one_oneball_cross_ratio_ver}, this degree of freedom becomes 1. 

Now one can derive the following identities using the orthogonal relations in classification of isometries:

\begin{corollary}\label{cor:classification_of_isometries}
Consider $\BB^2$ given by $(V,h)$, $M\in \SU(h)$.
  \begin{enumerate}[label=(\roman*)]
    \item When $M$ is loxodromic, let the three eigenvalues be $re^{i\theta}$, $r^{-1}e^{i\theta}$, $e^{-2i\theta}$ and their respective eigenvectors $v_1,v_2,v_3$. Then for any $u\in (\CC^3,h)$, 
    \begin{equation}\label{eq:lox_expand}
      u=\frac{\langle u, v_2\rangle}{\langle v_1, v_2\rangle}v_1 + \frac{\langle u, v_1\rangle}{\langle v_2, v_1\rangle}v_2 + \frac{\langle u, v_3\rangle}{\langle v_3, v_3\rangle}v_3.
    \end{equation}
    and thus
    \begin{equation}\label{eq:lox_sum_one}
      1=\frac{\langle u, v_2\rangle\langle v_1, u\rangle}{\langle v_1, v_2\rangle\langle u, u\rangle} + \frac{\langle u, v_1\rangle\langle v_2, u\rangle}{\langle v_2, v_1\rangle\langle u, u\rangle} + \frac{\langle u, v_3\rangle\langle v_3, u\rangle}{\langle v_3, v_3\rangle\langle u, u\rangle}
      =:2\Re(\eta(v_1,v_2,u))+\ta{u,v_3}.
    \end{equation} 
    \item When $M$ is elliptic with only one fixed point in $\BB^2$ (resp. boundary elliptic), we can find eigenvectors $v_1,v_2,v_3$ corresponding to eigenvalues $\lambda_1,\lambda_2,\lambda_3$ respectively such that $\lambda_1\ne\lambda_2$ and $\lambda_1\ne\lambda_3$ (resp. $\lambda_1=\lambda_2\ne\lambda_3$), and
    \[
      \langle v_1,v_1\rangle <0, \ \ \langle v_2,v_2\rangle, \langle v_3,v_3\rangle >0, \ \ \langle v_1,v_2\rangle=\langle v_1,v_3\rangle=\langle v_2,v_3\rangle=0
    \]
    and consequently, for any $u\in (\CC^3,h)$, 
    \begin{equation}\label{eq:elliptic_expand}
      u=\frac{\langle u, v_1\rangle}{\langle v_1, v_1\rangle}v_1 + \frac{\langle u, v_2\rangle}{\langle v_2, v_2\rangle}v_2 + \frac{\langle u, v_3\rangle}{\langle v_3, v_3\rangle}v_3.
    \end{equation}
    and thus
    \begin{equation}\label{eq:elliptic_sum_one}
      1=\frac{\langle u, v_1\rangle\langle v_1, u\rangle}{\langle v_1, v_1\rangle\langle u, u\rangle} + \frac{\langle u, v_2\rangle\langle v_2, u\rangle}{\langle v_2, v_2\rangle\langle u, u\rangle} + \frac{\langle u, v_3\rangle\langle v_3, u\rangle}{\langle v_3, v_3\rangle\langle u, u\rangle}
      =:\sum^3_{j=1} \ta{u,v_j}
    \end{equation}
  \end{enumerate}
\end{corollary}

For $\BB^1$, we can obtain the following analogue of Corollary~\ref{cor:classification_of_isometries} by a similar proof or using results from $\HH$. An additional property in this case is that when $M$ is loxodromic, its eigenvalues $\lambda,\lambda^{-1}$ are real: by $M^HJ_hM=J_h$, $M^{-1}$ and $M^H$ have the same set of eigenvalues, which means $M$ and $\bar{M}$ have the same set of eigenvalues, and then $|\lambda|\ne 1$ implies that $\lambda\in\RR$. 

\begin{corollary}\label{cor:classification_of_isometries_oneball}
Consider $\BB^1$ given by $(V,h)$, $M\in \SU(h)$.
  \begin{enumerate}[label=(\roman*)]
    \item $M$ is loxodromic if and only if one of the eigenvalues $\lambda$ has $|\lambda|\ne 1$. Let the eigenvalues $\lambda, \lambda^{-1}$ correspond to eigenvectors $v_1,v_2$ respectively. Then 
    \[\langle v_1,v_1\rangle=\langle v_2,v_2\rangle=0, \ \ \ \langle v_1,v_2\rangle\ne 0, \ \ \ \lambda\in\RR, \]
    and consequently, for any $u\in (\CC^2,h)$, 
    \begin{equation}\label{eq:lox_expand_oneball}
      u=\frac{\langle u, v_2\rangle}{\langle v_1, v_2\rangle}v_1 + \frac{\langle u, v_1\rangle}{\langle v_2, v_1\rangle}v_2,
    \end{equation}
    and thus for any $u_1,u_2$ with $\langle u_1,u_2\rangle\ne 0$
    \begin{equation}\label{eq:lox_sum_one_oneball_cross_ratio_ver}
      1=
      \frac{\langle u_1, v_2\rangle\langle v_1, u_2\rangle}{\langle u_1, u_2\rangle\langle v_1, v_2\rangle} 
      + \frac{\langle u_1, v_1\rangle\langle v_2, u_2\rangle}{\langle u_1, u_2\rangle\langle v_2, v_1\rangle}=[v_2,u_2,u_1,v_1]+[v_1,u_2,u_1,v_2],
    \end{equation} 
    and when $u_1=u_2=u$, 
    \begin{equation}\label{eq:lox_sum_one_oneball}
      1 = 2\Re(\eta(v_1,v_2,u))
    \end{equation}
    \item $M$ is elliptic if and only if it is diagonalizable and the eigenvalues have norm 1. Let the eigenvalues $\lambda, \lambda^{-1}$ correspond to eigenvectors $v_1 <0$ and $v_2 > 0$ respectively. Then $\langle v_1,v_2\rangle=0$ and consequently, for any $u\in (\CC^2,h)$, 
    \begin{equation}\label{eq:elliptic_expand_oneball}
      u=\frac{\langle u, v_1\rangle}{\langle v_1, v_1\rangle}v_1 + \frac{\langle u, v_2\rangle}{\langle v_2, v_2\rangle}v_2,
    \end{equation}
    and thus
    \begin{equation}\label{eq:elliptic_sum_one_oneball}
      1=\frac{\langle u, v_1\rangle\langle v_1, u\rangle}{\langle v_1, v_1\rangle\langle u, u\rangle} + \frac{\langle u, v_2\rangle\langle v_2, u\rangle}{\langle v_2, v_2\rangle\langle u, u\rangle} 
      =: \ta{u,v_1} + \ta{u,v_2}.
    \end{equation}
  \end{enumerate}\qed
\end{corollary}

\subsection{Totally geodesic subspaces and special curves}\label{sub:special_curves}

\subsubsection{Complex lines and distance formulas in $\BB^2$}
One can classify the totally geodesic subspaces of $\BB^n$ using the general theory of symmetric spaces:
\begin{theorem}[Section 3.1.11 of \cite{goldman1999complex}]\label{thm:classification_of_totally_geodesic_sub}
	The only totally geodesic subspaces of $\BB^n$ are either complex linear subspaces or totally real totally geodesic submanifolds. 
\end{theorem}
When $n=2$, they are complex lines, (real) geodesics, and totally real Lagrangian planes. Complex lines are the usual complex lines in $\CC^2$ that intersect the open unit 2-ball when we consider $\BB^2$ as the open unit 2-ball. Equivalently, a complex line $L\subset \BB^2$ is the projectivization of the intersection between a 2-dimensional linear subspace of $V=(\CC^3, h)$ and the negative cone $V_-$. Therefore, it can be determined by its polar vector: $L=\{z\in\BB^2:\langle z,n\rangle=0 \}$ for some $n\in V_+$. Totally real Lagrangian planes are embedded copies of the real hyperbolic space $\HH^2_\RR$. When we consider the standard Hermitian form, they are the images under a $\PU(h_{std})$-action of the space given by those points in $\BB^2$ with real coordinates. 

Then we list some distance formulas related to complex lines. 
\begin{proposition}[Corollary 6.7 of \cite{parker2003notes}]\label{prop:dist_point_line}
	Let $L$ be a complex line with polar vector $n$. For any $z\in \BB^2$, the distance $\rho_{\BB^2}(z,L)$ from $z$ to $L$ is given by
	\[
		\cosh^2\left(\frac{\rho_{\BB^2}(z,L)}{2}\right)=1-\ta{z,n}\ge 1. 
	\]
\end{proposition}
\begin{proposition}[Section 3.3.2 of \cite{goldman1999complex} or Proposition 6.8 of \cite{parker2003notes}]\label{prop:dist_between_lines}
	Let $L_1,L_2$ be complex lines with polar vector $n_1,n_2$. Then
	\begin{enumerate}[label=(\roman*)]
		\item If $\ta{n_1,n_2} > 1$, then $L_1$ and $L_2$ do not intersect in $\BB^2 \cup \partial\BB^2$ (which is called ultraparallel) and
		\[
			\cosh^2\left(\frac{\rho_{\BB^2}(L_1,L_2)}{2}\right)=\ta{n_1,n_2} >1.
		\]
		\item If $\ta{n_1,n_2} = 1$, then $L_1$ and $L_2$ intersect on $\partial\BB^2$ (which is called asymptotic) or coincide. 
		\item If $\ta{n_1,n_2} < 1$, then $L_1$ and $L_2$ intersect in $\BB^2$. 
	\end{enumerate}
\end{proposition}

\subsubsection{Totally geodesic submanifolds and distance formulas in $\BB^1\times\BB^1$}

Here we focus on two types of totally geodesic submanifolds: twisted diagonals, and products of geodesics. For $B\in \UM_2(h)^{\times}$, let 
\[
  F_B :=\{(u,Bu: u\in \BB^1\}.
\]
For $u_1^{(1)},u_2^{(1)},u_1^{(2)},u_2^{(2)}\in \partial\BB^1$, consider the geodesics $\widehat{u_1^{(j)}u_2^{(j)}}$ connecting $u_1^{(j)},u_2^{(j)}$ for $j=1,2$, and their product
\[
  \widehat{u_1^{(1)}u_2^{(1)}} \times \widehat{u_1^{(2)}u_2^{(2)}}=\{(u^{(1)},u^{(2)})\in \BB^1\times\BB^1: u^{(1)}\in \widehat{u_1^{(1)}u_2^{(1)}}, u^{(2)}\in \widehat{u_1^{(2)}u_2^{(2)}}\}.
\]

\begin{lemma} \label{lem:distance_formulas_oneball}
  We have the following distance formulas:
  \begin{enumerate}[label=(\alph*)]
    \item For $u=(u_1,u_2)\in \BB^1\times\BB^1$, $B\in SU(h)$,
    \[
      \cosh^2\left(\rho_{\BB^1\times\BB^1}(u, F_B)\right)=\cosh^2\left(\frac{\rho_{\BB^1}(Bu_1, u_2)}{2}\right)=\ta{Bu_1, u_2};
    \]
    \item For $v=(v_1,v_2)\in \BB^1\times\BB^1$, $u_1^{(1)},u_2^{(1)},u_1^{(2)},u_2^{(2)}\in \partial\BB^1$
    \[
      \cosh^2\left(\rho_{\BB^1\times\BB^1}(v, \widehat{u_1^{(1)}u_2^{(1)}} \times \widehat{u_1^{(2)}u_2^{(2)}})\right) = \max_{j=1,2}\lvert 2\eta(u_1^{(j)},u_2^{(j)},v_j)\rvert^2
    \]
    \item For $u^{(k)}_j,v^{(k)}_j\in \partial\BB^1$ for $j,k=1,2$,
    \begin{align*}
      &\cosh^2\left(\rho_{\BB^1\times\BB^1}(\widehat{u_1^{(1)}u_2^{(1)}} \times \widehat{u_1^{(2)}u_2^{(2)}}, \widehat{v_1^{(1)}v_2^{(1)}} \times \widehat{v_1^{(2)}v_2^{(2)}})\right) \\
       = &\max_{k=1,2} (\lvert [v_1^{(k)},u_2^{(k)},u_1^{(k)},v_2^{(k)}]\rvert + \lvert [v_2^{(k)},u_2^{(k)},u_1^{(k)},v_1^{(k)}]\rvert)^2
    \end{align*}
  \end{enumerate}
\end{lemma}
\begin{proof}
For (a), since $B\in SU(h)$:
    \[\rho_{\BB^1\times\BB^1}(u, F_B)=\min_{v\in\BB^1}\max(\rho_{\BB^1}(Bu_1,Bv),\rho_{\BB^1}(u_2,Bv)).\]
By the triangle inequality, $\max(\rho_{\BB^1}(Bu_1,Bv),\rho_{\BB^1}(u_2,Bv))\ge \rho_{\BB^1}(Bu_1,u_2)/2$, and the equality can be obtained when $Bv$ lies at the midpoint of the geodesic segment from $Bu_1$ to $u_2$. For (b), we combine \cite[Theorem 1]{sandler1996distance} and Equation \ref{eq:lox_sum_one_oneball}. For (c), we do the same calculation for $\BB^1$ as in \cite[Lemma 6.3]{parker2003notes} for $\BB^2$, and get
\[\cosh^2\left(\frac{\rho_{\BB^1}(\widehat{u_1u_2}, \widehat{v_1v_2})}{2}\right) = \frac{1}{2}\left(
\left\lvert \frac{\langle u_1,v_1\rangle\langle v_2,u_2\rangle}{\langle u_1,u_2\rangle\langle v_2,v_1\rangle} \right\rvert +
\left\lvert \frac{\langle u_1,v_2\rangle\langle v_1,u_2\rangle}{\langle u_1,u_2\rangle\langle v_1,v_2\rangle} \right\rvert + 1
 \right). \]

\end{proof}

\subsubsection{Special curves in \ref{notation_hms_smallsubfield} and quaternion algebra}

To make sure that the twisted diagonals in $\BB^1\times\BB^1$ descend to some well-behaved algebraic curves in the arithmetic quotient, we impose further restrictions on the defining matrix $B$. Here we describe a family of special curves in this context of \ref{notation_hms_smallsubfield}. 

For the usual Hilbert modular surfaces (See \cite[Chapter V]{van2012hilbert} for detail), a special curve is expected to look like $\{(Mz,M'z): z\in \HH\}$ for some $M\in GL_2(\QQ(\sqrt{D}))$ with totally positive determinant, where $M'$ is the Galois conjugate of $M$. Considering this as a graph of a Mobius transformation $B=M^*M'$, one can instead ask for $B'=B^*$. This makes sure that the stabilizer of the curve is a rank 4 $\ZZ$-sublattice of the Hilbert modular group for $\QQ(\sqrt{D})$. We will try to follow the notation in \cite{van2012hilbert} except that our $B$ is the ``skew-hermitian'' $B$ in \cite{van2012hilbert} right multiplied by $\begin{pmatrix} 0 & -1 \\ 1 & 0 \end{pmatrix}$. 

Consider the $K$-algebra generated by $\SU(h,F)$:\label{page:um2}
\[
  \UM_2(h,F):=\{M\in \mat_2(F): M^H J_h = J_h M^* \}.
\]
Since the hermitian form is defined over $F_0$, it descends to a $K_0$-algebra:
\[
  \UM_2(h,F_0):=\{M\in \mat_2(F_0): M^H J_h = J_h M^* \}, \ \ \UM_2(h,F) = \UM_2(h,F_0) \otimes_{K_0} K.
\]
We also consider those with integral entries $\UM_2(h,\OO_F):=\UM_2(h,F) \cap \mat_2(\OO_F)$. For an extension $\iota_F$ of the Galois conjugation $\iota_K$ in $\Gal(K/K_0)$ to $F/F_0$, applying $\iota_F$ entry-wise gives a $K$-semilinear involution on $\UM_2(h,F)$, which is equivalent to applying $\iota_K$ to the scalar term of $\UM_2(h,F_0) \otimes_{K_0} K$. Its restriction on $\UM_2(h,\OO_F)$ is also a $K$-semilinear involution. We also call this involution $\iota_F$. 

\begin{definition}\label{def:special_curve}
  Let $B\in \UM_2(h,\OO_F)$ be an invertible matrix with $\iota_F(B)=B^*$. The special curve $F_B$ in $\BB^1\times\BB^1$ or its arithmetic quotient $\Gamma\backslash(\BB^1\times\BB^1)$ is given by the graph of the Mobius transformation $B^{\sigma_1}$. 
\end{definition}

We want to understand the stabilizer of $F_B$. Consider
\[
	Q_B:= \{M\in \UM_2(h,F): \iota_F(M)B=BM  \}, \ \ \mathfrak{o}_B:=Q_B \cap \mat_2(\OO_F), \ \ E_B:=Q_B \cap \Gamma.
\]
Then the stabilizer $\{M\in \Gamma: \iota_F(M)B=\pm BM  \}$ of $F_B$ is an extension of degree 1 or 2 of $E_B$. 

The target of this section is to generalize \cite[Proposition.V.1.5]{van2012hilbert}, i.e. to compute the discriminants of the corresponding quaternion algebras and orders. Compared to the case of Hilbert modular groups, which says that the discriminant of the quaternion order equals to $\det(B)$, the first difference is that we cannot ask $B$ to be ``primitive'' as in \cite[Definition.V.1.2]{van2012hilbert} since $\OO_{K_0}$ may not be a PID. To fix this, we will work with the ideal $\det(B)/\cont^2(B)$. Note that $\iota_F(B)=B^*$ guarantees that it descends to an ideal in $\OO_{K_0}$. Another difference is that the ambient algebra $\UM_2(h,F)$ may not split. 

\begin{lemma} \label{lem:UM2_is_quaternion}
  $\UM_2(h,F_0)$ is a quaternion algebra over $K_0$ isomorphic to $(\beta,-\det J_h)_{K_0}$.  
\end{lemma}
\begin{proof}
  By hermitian Gram-Schmidt we can diagonalize $J_h$: $J_h=P^H\cdot (\begin{smallmatrix} e & \\ & f  \end{smallmatrix})\cdot P$ for some $P\in \GL_2(F_0)$ and $e,f\in K_0^\times$, so it reduces to the case of $J_h= (\begin{smallmatrix} e & \\ & f  \end{smallmatrix})$. Then the defining equation $M^H J_h = J_h M^*$ becomes
  \[
    \begin{pmatrix} \bar{a} & \bar{c} \\ \bar{b} & \bar{d} \end{pmatrix}
    \begin{pmatrix} e &  \\  & f \end{pmatrix}
    =
    \begin{pmatrix} e &  \\  & f \end{pmatrix}
    \begin{pmatrix} d & -b \\ -c & a \end{pmatrix}
    \implies 
    d=\bar{a}, b=-\bar{c}f/e.
  \]
  This is a standard matrix representation of $(\beta,-f/e)_{K_0}$ \cite[Proposition 2.3.1]{voight2021quaternion}, which is isomorphic to $(\beta,-\det J_h)_{K_0}$. 
\end{proof}

\begin{proposition}
	$Q_B$ is a quaternion algebra over $K_0$.
\end{proposition}
\begin{proof}
  We first count the dimension of $Q_B$ over $K_0$. It is a $K_0$-subalgebra of the 8-dimensional $K_0$-algebra $\UM_2(h,F)$. Since $B\in \UM_2(h,F)$ and $\iota_F$ is a $K_0$-linear automorphism, $\iota_F(M)B=BM$ defines four $K_0$-linear relations in $\UM_2(h,F)$. Therefore, $Q_B$ is at least of dimension $8-4=4$ over $K_0$. On the other hand, since $\trace(\iota(M))=\trace(BMB^{-1}))=\trace(M)$ for $M\in Q_B$, we have $\trace(M)\in K_0$. Thus, $Q_B\otimes_{K_0}K$ is at least of dimension 8 over $K_0$, so it is exactly all of $\UM_2(h,F)$. Note that for a finite-dimensional simple algebra over a field, being central simple is stable under base change (see e.g. \cite[Section 2.2]{gille2017central}). Therefore, $Q_B$ is a quaternion algebra over $K_0$. 
\end{proof}

Then we will use local-global principle to compute the discriminant of the quaternion $\OO_{K_0}$-order $\mathfrak{o}_B$. We first recall a corollary of the classification of hermitian lattices of rank 2 over local fields. 

\begin{proposition}\label{prop:classification_of_hermitian_over_local_fields}
  Let $(F,\pp)\subset (E,\ppp)$ be a quadratic extension of local fields. Let $\iota\in \Gal(E/F)$ be the Galois conjugation $x\mapsto x'$, which extends to an $E$-semilinear matrix map by applying $\iota$ entry-wise. 
  When $B\in \GL_2(E)$ satisfies $B'=B^*$, there exists $P\in \GL_2(\OO_E)$ such that $B={{P'}^*} B_0 P$ with $B_0$ given as follows.
    \begin{enumerate}
      \item When $E/F$ is unramified, for $i=v_F(\cont(B_0)),j=v_F(\det(B_0))-i$, and some uniformizer $p=\pi$ for both $E,F$,
      \[
        B_0=\begin{pmatrix} 0 & \pi^i \\ \pi^j & 0 \end{pmatrix}.
      \]
      \item When $E/F$ is ramified non-dyadic, for $i=v_E(\cont(B_0))$ and some uniformizers $\pi\in F$ and $p\in E$ satisfying $\pi=pp'$ and $p'=-p$,
      \[
        B_0=p\pi^{\frac{i-1}{2}}\begin{pmatrix} 0 & \pi^{-i}\det(B) \\  1 & 0 \end{pmatrix} \text{for odd $i$, or  }(-\pi)^{i/2}\begin{pmatrix} 1 & 0\\ 0 & 1  \end{pmatrix} \text{for even $i$}.
      \]
      \item When $E/F$ is ramified dyadic, for some uniformizers $\pi\in F$ and $p\in E$ satisfying $\pi=pp'$, let $\tau:=p-p'$. Then $B_0$ is one of the following types:
      \[
        \tau\begin{pmatrix} 0 & u_1\pi^r \\ u_2\pi^s & 0 \end{pmatrix}, \ \ 
        \tau\begin{pmatrix}  -{p'}^i & 0 \\ \pi^n & p^i  \end{pmatrix}, \ \ 
        \tau\begin{pmatrix}  -{p'}^i & \pi^{i-n}(u-1) \\ \pi^n & p^i  \end{pmatrix}, \ \ 
      \]
      where $i\le 2n \le i+\delta$ in the second type, and $i\le 2n < i+\delta$ in the third type. Here $\delta:=v_E(\different_{E/F})$, and $u$ is a unit in the nontrivial class of $\OO_F^{\times}/\norm_{E/F}(\OO_E^{\times})$. 
    \end{enumerate}
\end{proposition}
\begin{proof}
  This is a corollary of the general classification of hermitian lattices over local fields in \cite{jacobowitz1960hermitian} in the commutative rank 2 case, since $B'=B^*$ is equivalent to $\begin{pmatrix} 0 & 1 \\ -1 & 0 \end{pmatrix}\tau B$ being hermitian. Here $\tau\in \OO_E$ is an element satisfying $\tau'=-\tau$, which can be obtained by $\tau:=x-x'$ for any $x\ne x'$. We can pick $\tau$ to be a unit in the unramified case, and $p-p'\in\different_{E/F}$ in the ramified case for a uniformizer $p$, which is still a uniformizer when $E/F$ is not dyadic. 
\end{proof}

Now we compute the discriminant of the quaternion $\OO_{K_0}$-order $\mathfrak{o}_B$ at a prime $\pp\subset \OO_{K_0}$. Note that we can compute it over $K$ by:
\[
  \disc_{K_0}((\mathfrak{o}_B)_{\pp}) = \disc_K((\mathfrak{o}_B)\otimes_{\OO_{K_0}} \prod_{\ppp\mid\pp\OO_K} \OO_{K_{\ppp}}) .
\]
By \cite[Proposition 13.4.4]{voight2021quaternion}, when the quaternion algebra $Q_B$ does not split at $\pp$, $Q_B\otimes_{K_0}K$ splits at $\ppp$ if and only if $[K_{\ppp}:(K_0)_{\pp}]=2$. In other words, we only have the following cases: when $\pp$ is inert or ramified in $K/K_0$, $\UM_2(h,F)$ splits at $\ppp\mid \pp\OO_K$; when $\pp$ splits in $K/K_0$, $\UM_2(h,F)$ splits (resp. ramifies) at $\ppp\mid \pp\OO_K$ if $Q_B$ does at $\pp$. First, we compute the local discriminant for primes outside of a fixed set of bad primes
\[
  S:=\{\text{primes in $K_0$ dividing $\disc(F/K)\det J_h$ }\},
\]
and then we will discuss how to modify the computation for primes in $S$. 

\begin{lemma}\label{lem:quaternion_local_disc_nice_case}
  For $\pp\notin S$, if $\disc(K_{\ppp}/K_{0,\pp})=(\pi^{\delta})$ for prime $\ppp\mid\pp\OO_K$, then
  \[
    \discrd_{K_{0,\pp}}((\mathfrak{o}_B)_{\pp})=(\pi^{\delta-\epsilon}\cdot \det(B)/\cont^2_{\Lambda}(B) )
  \] 
  where $0\le \epsilon \le \delta$. 
\end{lemma}
\begin{proof}
  $\pp\notin S$ means that $\UM_2(h,F_0)$ splits at $\pp\subset\OO_K$ by Lemma~\ref{lem:UM2_is_quaternion}, and moreover that $J_h$ is unimodular at $\pp$, so the defining equation $M^H J_h = J_h M^*$ shows that $\UM_2(h,\OO_{F_0})_{\pp}$ is a maximal order in the split quaternion algebra $\UM_2(h,F_0)$. 

  Since all the maximal orders in a split quaternion algebra are conjugate, we can identify $\UM_2(h,F_0)_{\pp}\cong \mat_2(K_{0,\pp})$ with $\Lambda:=\UM_2(h,\OO_{F_0})_{\pp}\cong \mat_2(\OO_{K_{0,\pp}})$. Note that under this identification, these are two matrix representations of the same quaternion algebra and maximal order, so quaternion invariants are unchanged. Moreover, the valuations of the contents are also unchanged, since we have fixed a maximal order. When $\pp$ is inert or ramified in $K/K_0$, for $\ppp\mid\pp\OO_K$, $K_{\ppp}/K_{0,\pp}$ is of degree 2, and the semilinear automorphism $\iota_F$ induced by entry-wise Galois conjugation of $F/F_0$ on $\UM_2(h,F)$ can be identified as the entry-wise Galois conjugation $\iota_{K_{\ppp}}$ of $K_{\ppp}/K_{0,\pp}$ acting on $\mat_2(K_{\ppp})$. When $\pp$ splits in $K/K_0$, i.e. $\pp\OO_K=\ppp_1\ppp_2$, we have $(\OO_K)_{\pp}=\OO_{K_{\ppp_1}}\times \OO_{K_{\ppp_2}}$, and $\iota_F$ acts by swapping these two factors. In conclusion, the defining equation of $Q_B$ is still in the same form $\iota_{K_{\ppp}}(M)B=BM$. Then we can compute case by case. 

  When $\pp$ is inert or ramified in $K/K_0$, it suffices to compute the local discriminants for $B_0$ in the classification in Proposition~\ref{prop:classification_of_hermitian_over_local_fields}. It is enough to check the following types:
  \[
    J_0 = \begin{pmatrix} 0 & u_1\pi^r \\ u_2\pi^s & 0 \end{pmatrix},
    \begin{pmatrix} 1 & 0 \\ 0 & 1 \end{pmatrix}, 
    \begin{pmatrix}  -{p'}^i & \pi^{i-n}(u-1) \\ \pi^n & p^i  \end{pmatrix},
  \]
  where $u_1,u_2$ are units, $r\ge s$, $u=1$ or $p'/p$ and $i\le 2n < i+\delta$ for $\delta:=V_{K_{\ppp}}(\different_{K_{\ppp}/K_{0,\pp}})$. In the first case, $\iota_{K_{\ppp}}(M)B=BM$ is
  \[
    \begin{pmatrix} a' & b' \\ c' & d' \end{pmatrix}
    \begin{pmatrix} 0 & u_1\pi^r \\ u_2\pi^s & 0 \end{pmatrix}
    =
    \begin{pmatrix} 0 & u_1\pi^r \\ u_2\pi^s & 0 \end{pmatrix}
    \begin{pmatrix} a & b \\ c & d \end{pmatrix}
    \implies 
    d=a', c=b'u_2u_1^{-1}\pi^{r-s},
  \]
  so the local discriminant ideal is $(\pi^\delta\cdot\pi^{r-s})=\disc(K_{\ppp}/K_{0,\pp})\det(B)/\cont^2_{\Lambda}(B)$. In the second case, 
  \[
    \begin{pmatrix} a' & b' \\ c' & d' \end{pmatrix}
    \begin{pmatrix} 1 & 0 \\ 0 & 1 \end{pmatrix}
    =
    \begin{pmatrix} 1 & 0 \\ 0 & 1 \end{pmatrix}
    \begin{pmatrix} a & b \\ c & d \end{pmatrix}
    \implies 
    a=a', b=b', c=c', d=d',
  \]
  so the local discriminant ideal is $(1)$. In the third case, 
  \begin{align*}
    &\begin{pmatrix} a' & b' \\ c' & d' \end{pmatrix}
    \begin{pmatrix} -{p'}^i & \pi^{i-n}(u-1) \\ \pi^n & p^i \end{pmatrix}
    =
    \begin{pmatrix} -{p'}^i & \pi^{i-n}(u-1) \\ \pi^n & p^i \end{pmatrix}
    \begin{pmatrix} a & b \\ c & d \end{pmatrix} \implies \\
    & 
    \begin{cases}
      b\pi^n-c'(u-1)\pi^{i-n}=(d'-d)p^i=(a-a')p^i \\
      (d-a')(u-1)\pi^{i-n}=b'p^i+b{p'}^i \\
      (d'-a)\pi^n = c'{p'}^i+cp^i
    \end{cases}
    \implies 
    \begin{cases}
      a= a \\
      b= (c_0'(u-1)+a-a')p^i\pi^{-n} \\
      c= c_0\pi^n p^{-i} \\
      d= a'+c_0 + c'_0
    \end{cases}
  \end{align*}
  i.e. $M$ can be determined by $a$ and $c$. Then to make sure $a,b,c,d\in\OO_{K_{\ppp}}$, we have $a\in I_a:= \OO_{K_{\ppp}}$ which implies $(a-a')p^i\pi^{-n}\in \OO_{K_{\ppp}}$ by $i+\delta > 2n$, and $c_0$ should be in:
  \[
    (\pi^{-n}p^i) \cap \different_{K_{\ppp}/K_{0,\pp}}^{-1}, \text{ and if $u=p'/p$, also in}\  (\pi^np^{-i})\cdot p^{-(\delta-1)}.
  \]
  Since $i+\delta > 2n$, $(\pi^{-n}p^i) \cap \different_{K_{\ppp}/K_{0,\pp}}^{-1} = (\pi^{-n}p^i)$. Therefore, the domain $I_{c_0}$ of $c_0$ is:
  \[
    v_{K_{\ppp}}(I_c)=\begin{cases}
      -(2n-i) & \text{for }u=1, \\
      -\min(2n-i, i+\delta-1-2n) & \text{for }u=p'/p.
    \end{cases}
  \]
  Now we have found an integral basis for the lattice in terms of $a\in \OO_{K_{\ppp}}$ and $c_0\in I_c$. The determinant of the Gram matrix with respect to $(a,a',c_0,c'_0)$ is $-1$, so the reduced local discriminant is
  \[
    |-1|\cdot \norm_{K_{\ppp}/K_{0,\pp}}(I_c)\cdot \disc(K_{\ppp}/K_{0,\pp})=\pi^{v_{K_{\ppp}}(I_c)}\disc(K_{\ppp}/K_{0,\pp}).
  \]

  When $\pp$ splits in $K/K_0$, $(\OO_K)_{\pp}=\OO_{K_{\ppp_1}}\times \OO_{K_{\ppp_2}}$ splits into two isomorphic copies. Since $\iota_{K_{\ppp}}(B)=B^*$ becomes $(B_2,B_1)=(B_1^*,B_2^*)$, $\iota_{K_{\ppp}}(M)B=BM$ becomes componentwise
  \[
    (M_2,M_1)\cdot(B_1,B_2)=(B_1,B_2)\cdot(M_1,M_2) \Leftrightarrow M_2B_1=B_1M_1,
  \]
  i.e. it suffices to check one of the components. Consider a Smith normal form of $B_1$: for some $r\ge s$, $B_1=U\begin{pmatrix}
    p_1^r & 0 \\ 0 & p_1^s
  \end{pmatrix}V$ for $U,V\in \GL_2(\OO_{K_{\ppp_1}})$. Then:
  \begin{align*}
    & U^{-1} M_2U\begin{pmatrix} p_1^r & 0 \\ 0 & p_1^s \end{pmatrix}
    =:\begin{pmatrix} a_2 & b_2 \\ c_2 & d_2 \end{pmatrix}
    \begin{pmatrix} p_1^r & 0 \\ 0 & p_1^s \end{pmatrix}
    =
    \begin{pmatrix} p_1^r & 0 \\ 0 & p_1^s \end{pmatrix}
    \begin{pmatrix} a_1 & b_1 \\ c_1 & d_1 \end{pmatrix}
    :=\begin{pmatrix} p_1^r & 0 \\ 0 & p_1^s \end{pmatrix}VM_1V^{-1} \\
    & \implies 
    a_1=a_2, b_1=p_1^{s-r}b_2, c_1=p_1^{r-s}c_2, d_1=d_2.
  \end{align*}
  Since $M_1,M_2$ have integral entries, we can conclude that the local reduced discriminant is $(\pi^{r-s})$. 
\end{proof}

For $\pp\in S$, the proof has to be modified because of the following. First, when $\pp\mid \disc(\UM_2(h,F_0))$, $\UM_2(h,F_0)_{\pp}$ is a division algebra, so even if $\UM_2(h,F)_{\pp}$ splits, it is not contained in the $\iota_{K_{\ppp}}$-invariant part of $\UM_2(h,F)_{\pp}\cong \mat_2(K_{\ppp})$ i.e. we cannot identify $\iota_F=\iota_{K_{\ppp}}$. Also, when $\pp\mid\det(J_h)$, $\UM_2(h,\OO_{F_0})_{\pp}\otimes_{\OO_{K_{0,\pp}}} \OO_{K_{\ppp}}$ may not be a maximal order, so the content ideal of an element may change since the lattice to define the content changes. 

\begin{lemma}\label{lem:quaternion_local_disc_bad_case}
  For $\pp\in S$ inert or ramified in $K/K_0$, if $\disc(K_{\ppp}/K_{0,\pp})=(\pi^{\delta})$ for $\ppp\mid\pp\OO_K$, then
  \[
    \discrd_{K_{0,\pp}}((\mathfrak{o}_B)_{\pp})=(\pi^{\delta-\epsilon+ \delta_U -2\delta_{nm}}\cdot \det(B)/\cont^2_{\Lambda_F}(B) )
  \] 
  where $0\le \epsilon \le \delta$, $\delta_U\in \{\pm 1\}$ when $\pp\mid \disc(\UM_2(h,F_0))$ is inert in $K/K_0$ and $0$ otherwise, and $\delta_{nm}\in\ZZ$ satisfies $0\le \delta_{nm} \le v_{K_{\ppp}}(\discrd(\UM_2(h,\OO_{F_0})))$.
\end{lemma}
\begin{proof}
  When $\pp\mid \disc(\UM_2(h,F_0))$ is inert or ramified, $\UM_2(h,F_0)_{\pp}$ is a division algebra while its base-change to $K_{\ppp}$ splits. We have two $K_{\ppp}$-semilinear automorphisms $\iota_F$ and $\iota_{K_{\ppp}}$ of $\UM_2(h,F)_{\pp}\cong \mat_2(K_{\ppp})$, and by Skolem-Noether, $\iota_F^{-1}\circ \iota_{K_{\ppp}}$ is an inner automorphism. In other words, there is some $U\in \GL_2(K_{\ppp})$ such that $\iota_F(M)=U^{-1}\iota_{K_{\ppp}}(M) U$ for any $M\in \mat_2(K_{\ppp})$. In fact, this can be made explicit: by \cite[Theorem 6.4.11, 13.3.11]{voight2021quaternion}, $\UM_2(h,F_0)_{\pp}\cong (K_{\ppp}/K_{0,\pp}, a)_{K_{0,\pp}}$ for $a\in K_{0,\pp}$ not a norm. When $\pp$ is inert, we let $a$ be a uniformizer $\pi$, while when $\pp$ is ramified, we let $a$ be a nonsquare unit. In other words, there is some integral $U\in \UM_2(h,F_0)_{\pp}$ such that $U^2=a$ and $U^{-1}MU=\iota_{K_{\ppp}}(M)$ for $M\in \UM_2(h,F_0)_{\pp} \subset \mat_2(K_{\ppp})$, so we can identify $\iota_F$ as sending $M\otimes c \in \UM_2(h,F_0)_{\pp} \otimes_{K_{0,\pp}} K_{\ppp} $ to $U^{-1}\iota_{K_{\ppp}}(M) U \cdot c'$. 

  Consequently, the condition $\iota_F(M)B=BM$ can be rewritten as $\iota_{K_{\ppp}}(M)(UB)=(UB)M$, and $\iota_F(B)=B^*$ becomes $U^*\iota_{K_{\ppp}}(B)=(UB)^*$. Therefore, we still have $\iota_{K_{\ppp}}(UB)=(UB)^*$ if $\iota_{K_{\ppp}}(U)=U^*$, and this can be done using Hilbert 90: since $M=\iota_F^2(M)=U^{-1}\iota_{K_{\ppp}}(U)^{-1}M\iota_{K_{\ppp}}(U)U$ for any $M\in \mat_2(K_{\ppp})$, $\iota_{K_{\ppp}}(U)U$ must be a scalar, and then by Hilbert 90, we can pick $c\in K_{\ppp}^{\times}$ up to $K_{0,\pp}^{\times}$ such that $\iota_{K_{\ppp}}(U)=c/c' U^*$. 

  In summary, when $\pp\mid \disc(\UM_2(h,F_0))$ is inert or ramified, we can pick an integral $U\in \UM_2(h,F_0)_{\pp}$ with $U^2=a$ for $a$ given above, and $c\in \OO_{K_{\ppp}}$ with $0\le v_{K_{\ppp}}(c) < \delta:=v_{K_{\ppp}}(\different_{{K_{\ppp}}/{K_{0,\pp}}})$, such that $\tilde{B}:=cUB$ satisfies $\iota_{K_{\ppp}}\tilde{B}=\tilde{B}^*$ and $\iota_{K_{\ppp}}(\tilde{B})=\tilde{B}^*$. 

  To express the result in terms of the original $B$, another issue is that we change the quaternion order that defines the content of $B$. However, since the index of a quaternion suborder is given by the ratio of discriminants of the orders, and it bounds the change of contents from above.

  Since all maximal orders are conjugate in $\mat_2(K_{\ppp})$, we pick an embedding such that $\Lambda_F:=\UM_2(h,\OO_{F_0})_{\pp}\otimes_{\OO_{K_{0,\pp}}} \OO_{K_{\ppp}} \subset \mat_2(\OO_{K_{\ppp}})=:\Lambda$. Now the same calculation in Lemma~\ref{lem:quaternion_local_disc_nice_case} applies to $\tilde{B}$ with respect to the maximal order $\Lambda$, i.e. we compute the discriminant of the quaternion order
  \[
    \{M\in \Lambda: \iota_{K_{\ppp}}(M)\tilde{B}=\tilde{B}M  \}=\{M\in \Lambda: \iota_{F}(M)B=BM  \},
  \]
  and then combine the results with the following. For determinants, 
  \[
    v_{K_{0,\pp}}(\det(\tilde{B})) = \begin{cases}
      v_{K_{0,\pp}}(\det(B)) + 1 & \text{for $\pp$ inert}; \\
      v_{K_{0,\pp}}(\det(B)) + v_{K_{0,\pp}}(c^2)  & \text{for $\pp$ ramified}. \\
    \end{cases} 
  \]
  For contents, note that since $U^2=a$,
  \[
    v_{K_{\ppp}}(\cont_{\Lambda}(aB)) \ge v_{K_{\ppp}}(\cont_{\Lambda}(U)) + v_{K_{\ppp}}(\cont_{\Lambda}(UB)) \ge 2v_{K_{\ppp}}(\cont_{\Lambda}(U)) + v_{K_{\ppp}}(\cont_{\Lambda}(B)),
  \]
  and thus,
  \[
    v_{K_{\ppp}}(\cont_{\Lambda}(\tilde{B})) - (v_{K_{\ppp}}(\cont_{\Lambda}(B))+ v_{K_{\ppp}}(c)) \in \begin{cases}
      \{0,1\} & \text{for $\pp$ inert}; \\
      \{0\} & \text{for $\pp$ ramified}. \\
    \end{cases}
  \]
  Moreover, when we change the quaternion order that defines the content,
  \[
    v_{K_{\ppp}}(\cont_{\Lambda}(\tilde{B})) - v_{K_{\ppp}}(\cont_{\Lambda_F}(\tilde{B})) \in \{0, 1,\cdots, v_{K_{\ppp}}(\discrd(\UM_2(h,\OO_{F_0})_{\pp})).\}
  \] 
  In summary, 
  \[
    v_{K_{\ppp}}(\det(\tilde{B})/\cont_{\Gamma}^2(\tilde{B}))=
    v_{K_{\ppp}}(\det(B)/\cont_{\Gamma_F}^2(B)) + \delta_U -2\delta_{nm}
  \]
  where $\delta_U\in \{\pm 1\}$ when $\pp\mid \disc(\UM_2(h,F_0))$ is inert in $K/K_0$ and $\delta_{U}=0$ otherwise, and $\delta_{nm}\in\ZZ$ satisfies $0\le \delta_{nm} \le v_{K_{\ppp}}(\discrd(\UM_2(h,\OO_{F_0})))$. Combining this with Lemma~\ref{lem:quaternion_local_disc_nice_case}, we finish the proof. 
\end{proof}

The last case left is when $\UM_2(h,F_0)_{\pp}$ is a division algebra and $\pp$ splits in $K/K_0$. As in the split case in the proof of Lemma~\ref{lem:quaternion_local_disc_nice_case}, we can identify $\UM_2(h,F)_{\pp}$ as $\UM_2(h,F_0)_{\pp} \times \UM_2(h,F_0)_{\pp}$, and $\iota_F$ becomes swapping the two components. Therefore, to require $\iota_F(B)=B$ and $\iota_F(M)B=BM$ for $B=(B_1,B_2)$ and $M=(M_1,M_2)$, it suffices to have $B_2=B_1^*$ and $M_2B_1=B_1M_1$. Now we can bound $\discrd_{K_{0,\pp}}((\mathfrak{o}_B)_{\pp})$ by the fact that $[\Gamma:\Gamma \cap B_1\Gamma B_1^{-1}]_{\OO_{K_{0,\pp}}} \le [\Gamma_{\max} :  B_1\Gamma B_1^{-1}]_{\OO_{K_{0,\pp}}} $ where $\Gamma:= \UM_2(h,\OO_{F_0})_{\pp}$ and $\Gamma_{\max}$ is the unique maximal order in $\UM_2(h,F_0)_{\pp}$. However, for later use, what we want is to bound $\det(B)/\cont_{\Gamma_F}^2(B)$, which is $\det(B_1)/\cont_{\Gamma}^2(B_1)$ under this identification. In fact, for a quaternion uniformizer $j$ (i.e. $j\in \Gamma_{\max}$ with $j^2=\pi$ a uniformizer of $K_0$), if $B_1\in j^i\Gamma_{\max}\backslash j^{i+1}\Gamma_{\max}$, then $v_{K_{0,\pp}}(\cont_{\Gamma_{\max}}(B_1))=\lfloor i/2 \rfloor$ and $v_{K_{0,\pp}}(\det(B_1))=i$, so the valuation of $\det(B_1)/\cont_{\Gamma_{\max}}^2(B_1)$ can only be $0$ or $1$. Thus,
\begin{align*}
  1 &\ge v_{K_{0,\pp}}(\det(B_1)/\cont^2_{\Gamma_{\max}}(B_1)) \\
  &\ge v_{K_{0,\pp}}(\det(B_1)/\cont^2_{\Gamma}(B_1)) - 2v_{K_{0,\pp}}(\discrd(\UM_2(h,F_0)_{\pp})). \numberthis \label{eq:quaternion_local_disc_bad_split_case}
\end{align*}

If we combine Lemma~\ref{lem:quaternion_local_disc_nice_case}, \ref{lem:quaternion_local_disc_bad_case} and \eqref{eq:quaternion_local_disc_bad_split_case} and simplify the statement, we have:
\begin{proposition}\label{prop:quaternion_global_disc_compute}
  The quaternion $\OO_{K_0}$-order $\mathfrak{o}_B$ satisfies the inequality of ideals:
  \[
    \frac{\det(B)}{\cont^2(B)} \le \discrd_{K_0}(\mathfrak{o}_B) \cdot \discrd_{K_0}(\UM_2(h,\OO_{F_0}))^2 \cdot \disc_{K_0}(\UM_2(h,F_0)).
  \]
\end{proposition}

\section{Geometry of elliptic points}\label{sec:ell_pts}

\subsection{Repulsion of elliptic points}\label{subsec:repulsion_ell_pts}

This section is dedicated to show repulsion of elliptic points, i.e. how they become far apart when certain defining discriminants of the discrete subgroup $\Gamma$ grow. Such repulsion phenomena of special subvarieties for Shimura varieties have been shown in certain cases in terms of level structures. For instance, \cite{bakker2016p} showed the repulsion of cusps and special subvarieties for products of modular curves in terms of the level structure, and \cite{bakker2018geometric} proved similar results for cusps and points of Hilbert modular varieties.

Note that in this section we are showing the repulsion of the preimages of elliptic points in the universal cover, i.e. the points in $\BB^2$ or $\BB^1\times\BB^1$ stabilized by some elements in the corresponding $\Gamma$. The repulsion can descend to the quotient since the action of $\Gamma$ is properly discontinuous, and it is useful when we derive the multiplicity bounds in Section~\ref{sec:multiplicity_bounds}. 

We call points (resp. complex lines) in $\BB^2$ or $\BB^1\times\BB^1$ (or their images in the quotient) stabilized by an elliptic element inside the arithmetic lattice \emph{elliptic points} (resp. \emph{elliptic complex lines}). An elliptic point $u$ is an \emph{isolated elliptic point} if the locus stabilized by all elements in its stabilizer subgroup is exactly $u$ itself. In other words, isolated elliptic points include elliptic points not inside an elliptic complex line, and points inside an elliptic complex line with extra stabilizers. 

We will discuss three types of repulsion of elliptic points in this section. By the classification of isometries above, elliptic points in $\BB^2$ can be classified as an isolated elliptic point, or a general point in an elliptic complex line, while $\BB^1$ only has isolated elliptic points. In the cases we consider, isolated elliptic points in $\BB^2$ will simply get far apart, but this is not true for the other two types of elliptic points. For elliptic complex lines in $\BB^2$, there is certainly no pointwise repulsion, but even if we consider repulsion of complex lines, there are still elliptic complex lines intersect when the discriminant is large. Fortunately, the multiplicity of such intersection can be bounded. For (isolated) elliptic points in quotients of $\BB^1\times\BB^1$ in our case, they also do not necessarily get far apart, but we can prove that close elliptic points can be connected by special curves with small discriminants. 

Before we look into each case, we can first prove some general facts. First, the following easy observation will be implicitly used repeatedly. 

\begin{lemma}\label{lem:max_embedding_bounded_by_disc}
  For a number field $K/\QQ$ of degree $n$, let $\sigma_1,\cdots,\sigma_n$ be its complex embeddings. If $\QQ(x)=K$ for some $x\in \OO_K$, then
  \[
    \max_j \lvert \sigma_j(x)\rvert \ge \frac{1}{2}\lvert \disc(K/\QQ) \rvert^{\frac{1}{n(n-1)}}.
  \]
\end{lemma}
\begin{proof}
  By definition of discriminants and the triangle inequality, 
  \[
    \disc(K/\QQ) \le \disc(1,x,\cdots,x^{n-1}) = \prod_{1\le i<j\le n}(\sigma_i(x)-\sigma_j(x))^2 \le (2\max_j \lvert \sigma_j(x)\rvert)^{n(n-1)}.
  \]
\end{proof}

Next, note that in our setting there are not so many types of elliptic points in terms of eigenvalues.

\begin{lemma}\label{lem:ellclass}
Consider the cocompact arithmetic lattice $\Gamma$ and the arithmetic quotient $X:=(\BB^n)^m / \Gamma $ defined as in Section~\ref{subsub:compact_arithmetic_quotients}. Then for fixed $m,n,\ell$, there is a uniform bound on the maximal possible order of elliptic elements.

Furthermore, for our cases $(m,n)=(1,2),(2,1)$, the order of the stabilizer subgroup of $\Gamma$ for a given point is bounded.

\end{lemma}
\begin{proof}
Since $\Gamma$ acts properly discontinuously, the order $N$ of an elliptic element $M\in\Gamma$ is finite, and the eigenvalues of $M$ are N-th roots of unity. Therefore, the splitting field of the characteristic polynomial of $M$ over $F$ should include the N-th cyclotomic field, but the possible values of the degree of the splitting field over $\QQ$ are multiples of $2\ell$: $2\ell, \cdots,2\ell(n+1)$. Consequently, $\phi(N)$ can only be a divisor of them and now there are only finitely many cases ($\phi$ denotes Euler's totient function). 

Furthermore, for our case $(m,n)=(1,2),(2,1)$, the stabilizer subgroup of $\Gamma$ for a point corresponds to a finite subgroup of $U(n)$ (\cite[Corollary 3.2]{parker2003notes}). There is a classification of finite subgroups of $PGL_2(F)$ (see \cite{beauville2009finitesubgroupspgl2k}): it has to be one of cyclic groups, dihedral groups, alternating group of degree 4,5, or symmetric group of degree 4. Since there is a bound on the maximal order of elements, this gives a bound on the order of the group. 

\end{proof}

Then, we have the following useful relation to connect the elliptic elements and the distance between the points they stabilize. 

\begin{lemma}\label{lem:trace_as_sum_of_ta}
For a vector space equipped with a hermitian form $(\CC^n, h)$ and $M_1,M_2\in \SU(h)$, assume that $M_1$ and $M_2$ are diagonalizable, where for $M_1$ (resp. $M_2$) we can pick an eigenbasis $u_1,\cdots, u_n$ (resp. $v_1, \cdots, v_n$) corresponding to eigenvalues $\lambda_1,\cdots, \lambda_n$ (resp. $\tilde{\lambda}_1, \cdots, \tilde{\lambda}_n$) such that $u_1,\cdots, u_n$ (resp. $v_1, \cdots, v_n$) are pairwise orthogonal with respect to the hermitian form. Then
	\[
		\trace(M_1M_2) = \sum_{j,k=1}^n \lambda_j\tilde{\lambda}_k \ta{u_j, v_k}.
	\]
\end{lemma}
\begin{proof}
	We can apply (\ref{eq:elliptic_sum_one}) to $v_k$ (resp. $u_k$ ) to get the following useful equation for any fixed $k\in\{1,2,3\}$:
	\begin{equation}\label{eq:expand}
		\sum_{j=1}^n \ta{v_k, u_j} = 1  \ \ \  \left(\text{resp. }\sum_{j=1}^n \ta{u_k, v_j}=1\right).
	\end{equation}
	To compute $\trace(M_1M_2)$, first note that the orthogonal relations give the following:
	\begin{align*}
		\left(u_1\cdots u_n\right)^H J_h \left(u_1\cdots u_n\right) &= \begin{pmatrix}\langle u_1,u_1\rangle & & \\ & \cdots & \\ & & \langle u_n,u_n\rangle \\ \end{pmatrix} \\
		\implies \left(u_1\cdots u_n\right)^{-1} &= \begin{pmatrix}1/\langle u_1,u_1\rangle & & \\ & \cdots & \\ & & 1/\langle u_n,u_n\rangle \\ \end{pmatrix} \left(u_1\cdots u_n\right)^H J_h \\
	\end{align*}
	Using this we can rewrite the diagonalized $M_1$ and $M_2$ and get:
	\begin{align*}
		&\trace(M_1M_2)\\ =&\trace\left( \left(u_1\cdots u_n\right)\begin{pmatrix}\lambda_1 & & \\ & \ddots & \\ & & \lambda_n \\ \end{pmatrix} \left(u_1\cdots u_n\right)^{-1}\left(v_1\cdots v_n\right)\begin{pmatrix}\tilde{\lambda}_1 & & \\ & \ddots & \\ & & \tilde{\lambda}_n \\ \end{pmatrix} \left(v_1\cdots v_n\right)^{-1}  \right) \\
		=& \trace\Bigg( \begin{pmatrix}\frac{\lambda_1}{\langle u_1,u_1\rangle} & & \\ & \ddots & \\ & & \frac{\lambda_n}{\langle u_n,u_n\rangle} \\ \end{pmatrix} \left(u_1\cdots u_n\right)^H J_h \left(v_1\cdots v_n\right) \\
		& \qquad  \begin{pmatrix}\frac{\tilde{\lambda}_1}{\langle v_1,v_1\rangle} & & \\ & \ddots & \\ & & \frac{\tilde{\lambda}_n}{\langle v_n,v_n\rangle} \\ \end{pmatrix} \left(v_1\cdots v_n\right)^H J_h \left(u_1\cdots u_n\right) \Bigg)\\
		=&\sum_{j,k=1}^n \lambda_j\tilde{\lambda}_k \ta{u_j, v_k} \numberthis \label{eq:trace_sum_of_ta}   \\
	\end{align*}
\end{proof}

This can be stated more canonical in the sense that we can replace eigenvectors by eigenspaces as follows. 

\begin{lemma}\label{lem:ta_space_version}
For a vector space equipped with a hermitian form $(\CC^n, h)$, consider a vector subspace $V_1$ (resp. $V_2$) of dimension $m_1$ (resp. $m_2$) with a basis $u_1,\cdots,u_{m_1}$ (resp. $v_1,\cdots,v_{m_2}$) orthogonal with respect to the hermitian form. Then
\[
  \ta{V_1,V_2}:=\sum_{j=1}^{m_1}\sum_{k=1}^{m_2}\ta{u_j,v_k}
\]
is well-defined regardless of the choice of orthogonal basis. 
\end{lemma}
\begin{proof}
  Note that
  \begin{align*}
    &\sum_{j=1}^{m_1}\sum_{k=1}^{m_2}\ta{u_j,v_k} \\
    =&\trace\Bigg( \begin{pmatrix}\langle u_1,u_1\rangle & & \\ & \cdots & \\ & & \langle u_{m_1},u_{m_1}\rangle \\ \end{pmatrix}^{-1} \left(u_1\cdots u_n\right)^H J_h \left(v_1\cdots v_n\right) \\
    & \qquad  \begin{pmatrix}\langle v_1,v_1\rangle & & \\ & \cdots & \\ & & \langle v_{m_2},v_{m_2}\rangle \\ \end{pmatrix}^{-1} \left(v_1\cdots v_n\right)^H J_h \left(u_1\cdots u_n\right) \Bigg).\\
  \end{align*}
  If we have another orthogonal basis $(u'_1,\cdots,u'_{m_1})=(u_1,\cdots,u_{m_1})M$ for $M=(m_{jk})_{j,k}\in \GL_{m_1}(\CC)$, then 
  \[
    \begin{pmatrix}\langle u'_1,u'_1\rangle & & \\ & \cdots & \\ & & \langle u'_{m_1},u'_{m_1}\rangle \\ \end{pmatrix} 
    = M^H \begin{pmatrix}\langle u_1,u_1\rangle & & \\ & \cdots & \\ & & \langle u_{m_1},u_{m_1}\rangle \\ \end{pmatrix} M.
  \]
  Thus, the term
  \[
    \left(u_1\cdots u_n\right)\begin{pmatrix}\langle u_1,u_1\rangle & & \\ & \cdots & \\ & & \langle u_{m_1},u_{m_1}\rangle \\ \end{pmatrix}^{-1} \left(u_1\cdots u_n\right)^H
  \]
  is invariant and the trace stays the same. The same argument is true for $v_1,\cdots,v_{m_2}$.
\end{proof}

Now we can restate Lemma~\ref{lem:trace_as_sum_of_ta} as follows.
\begin{lemma}\label{lem:trace_as_sum_of_ta_space_version}
  For a vector space equipped with a hermitian form $(\CC^n, h)$ and $M_1,M_2\in \SU(h)$, assume that $M_1$ (resp. $M_2$) is diagonalizable with the eigenvalue set $\{\lambda_1,\cdots,\lambda_{m_1}\}$ (resp. $\{\tilde{\lambda}_1,\cdots,\tilde{\lambda}_{m_2}\}$), and its eigenvectors are orthogonal if they do not correspond to the same eigenvalue. Then
  \[
    \trace(M_1M_2) = \sum_{j=1}^{m_1}\sum_{k=1}^{m_2} \lambda_j\tilde{\lambda}_k \ta{V_{\lambda_j}, V_{\tilde{\lambda}_k}}.
  \]
  where $V_\lambda$ is the corresponding $\lambda$-eigenspace. 
\end{lemma}

\subsubsection{Repulsion of elliptic points in \ref{notation_twoball}}\label{ssub:repulsion_of_elliptic_points_in_bb_2_}

We have three types of elliptic points in \ref{notation_twoball}: elliptic points not on any elliptic complex line, elliptic points on an elliptic complex line with extra stabilizers, and elliptic points on an elliptic complex line whose stabilizer subgroup stabilizes this elliptic complex line. We call the first two types isolated elliptic points. We will show the following kinds of repulsion: when the discriminant $D$ is large enough, the first type of elliptic points is far from any other elliptic points, the second type is far from any elliptic points except those lying on the same elliptic complex line as it, and lastly the elliptic complex lines are far apart unless they intersect. In fact, the first two kinds of repulsion are covered by the following proposition.

\begin{proposition}\label{prop:repulsion}
  Follow the notation in (\ref{notation_twoball}). For any given $R>0$, there exists $D_0(R)>0$ such that when $D>D_0(R)$, if $M_1$ and $M_2$ are elliptic elements such that $\FixLocus(M_1)=\{u\}$ and the distance $\rho_{\BB^2}(u,\FixLocus(M_2)) \le R$ where $\FixLocus(M)$ is the fix locus of $M$. Then, $\rho_{\BB^2}(u,\FixLocus(M_2)) =0$
\end{proposition}

Before we prove this proposition, we first see why it also proves the second kind of repulsion above. 

\begin{lemma}\label{lem:intersection_of_elliptic_lines_can_have_one_dim_eigenspace}
  Follow the notation in (\ref{notation_twoball}). Let $L_1$ and $L_2$ be elliptic complex lines stabilized by $M_1$ and $M_2$ respectively, and they intersect at a point $u\in\BB^2$. Then either $L_1$ and $L_2$ overlap, or $\FixLocus(M_1M_2)=\{u\}$.
\end{lemma}
\begin{proof}
  Let $u_1,u_2,u_3$ (resp. $v_1,v_2,v_3$) be eigenvectors of $M_1$ (resp. $M_2$) in an orthogonal eigendecomposition of $M_1$ (resp. $M_2$) with respect to the hermitian form, and let $\lambda_j$ (resp. $\tilde{\lambda}_j$) be the eigenvalue corresponding to $u_j$ (resp. $v_j$). The assumption says that we can pick $u=u_1=v_1$ and $\lambda_1=\lambda_2$, $\tilde{\lambda}_1=\tilde{\lambda}_2$. Now $u$ is in the $\lambda_1\tilde{\lambda}_1$-eigenspace of $M_1M_2$, and we directly compute when this eigenspace has dimension greater than 1 as follows. For any $w=\sum_{k=1}^3 a_kv_k$,
  \begin{align*}
    M_2 w =& \begin{pmatrix} v_1 & v_2 & v_3 \\ \end{pmatrix}
    \begin{pmatrix} \tilde{\lambda}_1 & & \\ & \tilde{\lambda}_1 & \\ & &  \tilde{\lambda}_3 \end{pmatrix}
    \begin{pmatrix} a_1 \\ b_1 \\ c_1 \end{pmatrix}
    \\
    =& \begin{pmatrix} u_1 & u_2 & u_3 \\ \end{pmatrix}
    \begin{pmatrix} 
      \frac{\langle v_1,u_1 \rangle}{\langle u_1,u_1 \rangle} & \frac{\langle v_2,u_1 \rangle}{\langle u_1,u_1 \rangle} & \frac{\langle v_3,u_1 \rangle}{\langle u_1,u_1 \rangle} \\
      \frac{\langle v_1,u_2 \rangle}{\langle u_2,u_2 \rangle} & \frac{\langle v_2,u_2 \rangle}{\langle u_2,u_2 \rangle} & \frac{\langle v_3,u_2 \rangle}{\langle u_2,u_2 \rangle} \\
      \frac{\langle v_1,u_3 \rangle}{\langle u_3,u_3 \rangle} & \frac{\langle v_2,u_3 \rangle}{\langle u_3,u_3 \rangle} & \frac{\langle v_3,u_3 \rangle}{\langle u_3,u_3 \rangle} \\
      \end{pmatrix}
    \begin{pmatrix} \tilde{\lambda}_1 & & \\ & \tilde{\lambda}_1 & \\ & &  \tilde{\lambda}_3 \end{pmatrix}
    \begin{pmatrix} a_1 \\ b_1 \\ c_1 \end{pmatrix}
    \\
    M_1M_2w=& \begin{pmatrix} u_1 & u_2 & u_3 \\ \end{pmatrix}
    \begin{pmatrix} \lambda_1 & & \\ & \lambda_1 & \\ & &  \lambda_3 \end{pmatrix}
    \begin{pmatrix} 
      \frac{\langle v_1,u_1 \rangle}{\langle u_1,u_1 \rangle} & \frac{\langle v_2,u_1 \rangle}{\langle u_1,u_1 \rangle} & \frac{\langle v_3,u_1 \rangle}{\langle u_1,u_1 \rangle} \\
      \frac{\langle v_1,u_2 \rangle}{\langle u_2,u_2 \rangle} & \frac{\langle v_2,u_2 \rangle}{\langle u_2,u_2 \rangle} & \frac{\langle v_3,u_2 \rangle}{\langle u_2,u_2 \rangle} \\
      \frac{\langle v_1,u_3 \rangle}{\langle u_3,u_3 \rangle} & \frac{\langle v_2,u_3 \rangle}{\langle u_3,u_3 \rangle} & \frac{\langle v_3,u_3 \rangle}{\langle u_3,u_3 \rangle} \\
      \end{pmatrix}
    \begin{pmatrix} \tilde{\lambda}_1 & & \\ & \tilde{\lambda}_1 & \\ & &  \tilde{\lambda}_3 \end{pmatrix}
    \begin{pmatrix} a_1 \\ b_1 \\ c_1 \end{pmatrix}
    \\
    \lambda_1\tilde{\lambda}_1 w=& \lambda_1\tilde{\lambda}_1 \cdot \begin{pmatrix} u_1 & u_2 & u_3 \\ \end{pmatrix}
    \begin{pmatrix} 
      \frac{\langle v_1,u_1 \rangle}{\langle u_1,u_1 \rangle} & \frac{\langle v_2,u_1 \rangle}{\langle u_1,u_1 \rangle} & \frac{\langle v_3,u_1 \rangle}{\langle u_1,u_1 \rangle} \\
      \frac{\langle v_1,u_2 \rangle}{\langle u_2,u_2 \rangle} & \frac{\langle v_2,u_2 \rangle}{\langle u_2,u_2 \rangle} & \frac{\langle v_3,u_2 \rangle}{\langle u_2,u_2 \rangle} \\
      \frac{\langle v_1,u_3 \rangle}{\langle u_3,u_3 \rangle} & \frac{\langle v_2,u_3 \rangle}{\langle u_3,u_3 \rangle} & \frac{\langle v_3,u_3 \rangle}{\langle u_3,u_3 \rangle} \\
      \end{pmatrix}
    \begin{pmatrix} a_1 \\ b_1 \\ c_1 \end{pmatrix}
  \end{align*}
  Since $u=u_1=v_1$, we have $\langle v_1,u_2 \rangle=\langle v_1,u_3 \rangle=\langle v_2,u_1 \rangle=\langle v_3,u_1 \rangle=0$, and 
  \[
    (M_1M_2-\lambda_1\tilde{\lambda}_1)w = 
    \begin{pmatrix} u_1 & u_2 & u_3 \\ \end{pmatrix}
    \begin{pmatrix} 
      0 & 0 & 0 \\
      0 & 0 & \lambda_1(\tilde{\lambda}_3-\tilde{\lambda}_1)\frac{\langle v_3,u_2 \rangle}{\langle u_2,u_2 \rangle} \\
      0 & \tilde{\lambda}_1(\lambda_3-\lambda_1)\frac{\langle v_2,u_3 \rangle}{\langle u_3,u_3 \rangle} & (\lambda_3\tilde{\lambda}_3-\lambda_1\tilde{\lambda}_1)\frac{\langle v_3,u_3 \rangle}{\langle u_3,u_3 \rangle} \\
      \end{pmatrix}
    \begin{pmatrix} a_1 \\ b_1 \\ c_1 \end{pmatrix}.
  \]
  Since $M_1,M_2$ are non-scalar elliptic elements, the eigenvalues are nonzero, and $\lambda_1\ne \lambda_3$, $\tilde{\lambda}_1\ne \tilde{\lambda}_3$. Therefore, the $\lambda_1\tilde{\lambda}_1$-eigenspace has dimension greater than 1 if and only if either $\langle v_3,u_2 \rangle=0$ or $\langle v_2,u_3 \rangle=0$. By Equation (\ref{eq:elliptic_sum_one}), either of them implies that $\ta{u_3,v_3}=1$, which means $L_1$ and $L_2$ overlap by Proposition~\ref{prop:dist_between_lines} (there are no asymptotic complex lines in our compact case). 
\end{proof}

Consequently, inside the stabilizer subgroup for an isolated elliptic point, there is always an elliptic element that only stabilizes this point in $\BB^2$. 

The proofs of the repulsion phenomena in this Section~\ref{sec:ell_pts} depend on the integrality of distances between special subvarieties that stems from the integrality of the arithmetic lattice $\Gamma$, and in this proof we will see how Lemma~\ref{lem:trace_as_sum_of_ta_space_version} can be applied to connect them. 

\begin{proof}[Proof of Proposition \ref{prop:repulsion}]
  Let $u=u_1$ and $v=v_1$ be stabilized by $M_1$ and $M_2$ in $\Gamma$ respectively. Let $u_2,u_3$ (resp. $v_2,v_3$) be other two eigenvectors of $M_1$ (resp. $M_2$) in the orthogonal eigendecomposition of $M_1$ (resp. $M_2$) with respect to the hermitian form. For $j=1,2,3$, let $\lambda_j$ (resp. $\tilde{\lambda}_j$) be the eigenvalue corresponding to $u_j$ (resp. $v_j$). 

  First we have a preliminary bound on $\trace(M_1^{n_1}M_2^{n_2})$. 
  \begin{lemma}\label{lem:trace_bound_3}
  Follow the notation in (\ref{notation_twoball}). For any given $R>0$, there exists $D_0(R)>0$ such that when $D>D_0(R)$, for any $u=u_1$ and $v=v_1$ stabilized by $M_1$ and $M_2$ in $\Gamma$ respectively, $\rho_{\BB^2}(u,v) \le R$ implies that $ |\trace(M_1^{n_1}M_2^{n_2})| \le 3$ for any $n_1,n_2\in\ZZ$. 
  \end{lemma}
  \begin{proof}[Proof of Lemma]
    By Lemma~\ref{lem:trace_as_sum_of_ta}, we have
    \begin{equation}
      \trace(M_1^{n_1}M_2^{n_2})=\sum_{j,k=1}^3 \lambda_j^{n_1}\tilde{\lambda}_k^{n_2} \ta{u_j, v_k}
    \end{equation}

    Since $u_1,v_1<0$ and $u_2,u_3,v_2,v_3 > 0$, we have
    \begin{align*}
    |\trace(M_1M_2)| &= \left|\sum_{j,k=1}^3 \lambda_j^{n_1}\tilde{\lambda}_k^{n_2} \ta{u_j, v_k}\right| \le \sum_{j,k=1}^3 \left| \ta{u_j, v_k}\right| & \\
    &= \ta{u_1, v_1} + \sum_{j,k=2}^3 \ta{u_j, v_k} - \sum_{j=2}^3 (\ta{u_j, v_1}+\ta{u_1, v_j}) & \\
    &= \ta{u_1, v_1} + (\ta{u_1, v_1}+1) - 2(1-\ta{u_1, v_1})  &\mathrm{by\ }(\ref{eq:expand}) \\
    &= 4\ta{u_1, v_1}-1
    \end{align*}

    In conclusion, $|\trace(M_1^{n_1}M_2^{n_2})|^2\in \OO_K$ is bounded by $(4\cosh^2(R/2)-1)^2$ under $\sigma_1,\bar{\sigma}_1$, and is bounded by $3^2$ under any other complex embeddings of $F$ since the corresponding hermitian form $h$ is definite. Therefore, there is a bound on the discriminant of $\QQ(|\trace(M_1^{n_1}M_2^{n_2})|^2)$ only depending on $R$, so it is a proper subfield of $K$ when $D$ is greater than the bound of the discriminant of this field, and $(|\trace(M_1^{n_1}M_2^{n_2})|^2)^{\sigma_1}=(|\trace(M_1^{n_1}M_2^{n_2})|^2)^{\sigma} \le 9$ for some other complex embedding $\{\sigma_1,\bar{\sigma}_1\}\not\ni\sigma:F\to\CC$ that agrees with $\sigma_1$ on $\QQ(|\trace(M_1^{n_1}M_2^{n_2})|^2)$. 
  \end{proof}

  Come back to the proof of the proposition. Let $N$ be the order of $M_1$, so $\{\lambda_j\}_j$ generate the $N$-th cyclotomic field. Consider the $N\times 3$ matrix
  \[
    \tilde{P}= \begin{pmatrix*} 1 & 1 & 1 \\ \lambda_1 & \lambda_2 & \lambda_3 \\ \vdots & \vdots & \vdots \\ \lambda_1^{N-1} & \lambda_2^{N-1} & \lambda_3^{N-1} \end{pmatrix*}.
  \]
  By Artin's theorem on the linear independence of characters (see \cite[Corollary 4.2 in Chapter VI]{lang2012algebra}), the rank of this matrix is the number of distinct eigenvalues, which we call $m_1$. In our case $m_1$ is just 2 or 3. Let $\lambda'_1=\lambda_1,\cdots,\lambda'_{m_1}=\lambda_3$ be the set of distinct eigenvalues. Now we can pick a square matrix $P_1$ of full rank $m_1$:
    \[
    P_1=\begin{pmatrix*} {\lambda'}_1^{e_1} & \cdots & {\lambda'}_{m_1}^{e_1} \\ \vdots & \ddots & \vdots \\ {\lambda'}_1^{e_{m_1}} & \cdots & {\lambda'}_{m_1}^{e_{m_1}} \end{pmatrix*}.
  \]
  We do the same for $M_2$ and $\{\tilde{\lambda}_j\}_j$, and obtain a square matrix $P_2$ of full rank $m_2$ whose entries are powers of $\{\tilde{\lambda}'_j\}_{1\le j\le m_2}$. 
  Now we can apply Lemma~\ref{lem:trace_as_sum_of_ta_space_version} to get:
  \[
    P_1\cdot 
    (\ta{U_{\lambda'_j},V_{\tilde{\lambda}'_k}})_{j,k} \cdot 
    P_2^T
    =
    (\trace(M_1^{e_j} M_2^{e_k}))_{j,k}
  \]
  where $U_{\lambda'}$ (resp. $V_{\tilde{\lambda}'}$) is the $\lambda'$-eigenspace of $M_1$ (resp. $\tilde{\lambda}'$-eigenspace of $M_2$). Since $\det P_1,\det P_2 \ne 0$, we can invert them and obtain:
  \[
    (\ta{U_{\lambda'_j},V_{\tilde{\lambda}'_k}})_{j,k} 
    =
    P_1^{-1}\cdot(\trace(M_1^{e_j} M_2^{e_k}))_{j,k} \cdot (P_2^T)^{-1}.
  \]
  Note that the distance between two points or between a point and a line can be unified as $\ta{U_{\lambda'_1},V_{\tilde{\lambda}'_1}}$ (combining Proposition~\ref{prop:dist_point_line} and Equation~\ref{eq:elliptic_sum_one}). Now in conclusion, $\ta{U_{\lambda'_1},V_{\tilde{\lambda}'_1}}$ is inside the field $F'$ generated by adding those $\trace(M_1^{e_j}M_2^{e_k})$ and then splitting the characteristic polynomial $f_i=\lambda^3-\trace(M_i)\lambda^2+\overline{\trace(M_i)}\lambda-1$ of $i=1,2$. Lemma~\ref{lem:trace_bound_3} implies that the discriminant of $F'$ is bounded. Therefore, when the discriminant $D$ is greater than this bound, $F'$ is a proper subfield of the splitting field of $f_1f_2$ over $F$, and as in the argument at the end of the proof of Lemma~\ref{lem:trace_bound_3}, $\ta{U_{\lambda'_1},V_{\tilde{\lambda}'_1}}\le 1$. Since we already have $\ta{U_{\lambda'_1},V_{\tilde{\lambda}'_1}}\ge 1$, we conclude that $\ta{U_{\lambda'_1},V_{\tilde{\lambda}'_1}}= 1$ and $u_1\in \FixLocus(M_2)$. 
\end{proof}

For the repulsion of elliptic complex lines, we have the following.

\begin{proposition}\label{prop:repulsion_lines}
Follow the notation in (\ref{notation_twoball}). For any given $R>0$, there exists $D_0(R)>0$ such that when $D>D_0(R)$, if $M_1$ and $M_2$ are elliptic elements stabilizing disjoint complex lines, then their distance is greater than $R$.
\end{proposition}
\begin{proof}
The proof is similar to Proposition~\ref{prop:repulsion}. Follow the same notation as in the proof of Proposition~\ref{prop:repulsion}. The major difference is that now the distance is described by $\ta{u_3,v_3}$ where $u_3,v_3$ are the normal vectors of the 2-dimensional eigenspaces corresponding to the two complex lines. Since $L_1,L_2$ are disjoint, $\ta{u_3,v_3}^{\sigma_1}>1$, but $\ta{u_3,v_3}^{\sigma}\le 1$ for $\sigma=\sigma_2,\cdots$. Therefore, $\QQ(\ta{u_3,v_3})\cap F$ can not be a proper subfield of $F$. By the argument in the proof of Proposition~\ref{prop:repulsion}, $\ta{u_3,v_3}\cdot \det P_1 \det P_2$ is integral inside the splitting field of $f_1f_2$ over $F$. Since $\det P_1 \det P_2$ has a bounded norm, when the discriminant is large, the norm of $\ta{u_3,v_3}$ has to be large, which implies that $\ta{u_3,v_3}^{\sigma_1}$ is large. 
\end{proof}

\subsubsection{Repulsion of totally geodesic submanifolds in \ref{notation_hms_smallsubfield}}\label{ssub:repulsion_of_totally_geodesic_submanifolds_in_ref_notation_hms_smallsubfield}

Although we only have isolated elliptic points in this case, the simple pointwise repulsion as in Proposition~\ref{prop:repulsion} is not true. To fully understand the whole picture, we have to investigate the repulsion of all totally geodesic submanifolds compatible with the lattice, including elliptic points, special curves, and axes of loxodromic elements. The latter two types are the diagonal or the product of two real hyperbolic 1-balls twisted by a Mobius transformation, so they are both totally geodesic submanifolds. The list of propositions we want is the following.

\begin{center}
  \begin{tabular}{|c|c|c|c|}\hline 
    Repulsion between... & Elliptic points & Axes of loxodromic elements & Special curves  \\ \hline
    Elliptic points & Proposition~\ref{prop:repulsion_hms} & Proposition~\ref{prop:repulsion_hms_elliptic_and_lox} & Proposition~\ref{cor:repulsion_hms_small_inj_radius_and_special} \\\hline
    Axes of loxodromic elements & & Proposition~\ref{prop:repulsion_hms_axes_of_lox}, \ref{prop:repulsion_hms_axes_of_lox_common_special} & Proposition~\ref{cor:repulsion_hms_small_inj_radius_and_special} \\\hline
    Special curves & & & Proposition~\ref{prop:repulsion_curves_hms} \\ \hline
  \end{tabular}
\end{center}

\begin{proposition}[Repulsion between elliptic points]\label{prop:repulsion_hms}
	Follow the notation in (\ref{notation_hms_smallsubfield}). For any given $R>0$, there exists $D_0(R)>0$ and $N_0(R)>0$ such that when $D>D_0(R)$, if $u,v$ are distinct elliptic points on $\BB^1\times\BB^1$ respectively fixed by $M_1,M_2\in \Gamma$, then either their distance $\rho_{\BB^1\times\BB^1}(u,v) > R$, or they both lie on a special curve $F_B$ with $\norm(\det(B)/\cont^2(B)) < N_0(R)$.
\end{proposition}
\begin{proof}
  It suffices to find $D_0(R)>0$ and $N_0(R)$ for all $u=(u_1,u'_1)$ and $v=(v_1,v'_1)$ with $\rho_{\BB^1\times\BB^1}(u,v) \le R$ and construct a special curve $F_B$ with $\norm(\det B) < N_0(R)$ to connect them assuming $D > D_0(R)$. 

  First, we can connect the trace of $M_1M_2$ with the distance between $u,v$. Let $u_1,u_2$ (resp. $v_1,v_2$) be eigenvectors of $M_1$ (resp. $M_2$) corresponding to eigenvalues $\lambda_1,\lambda_2$ (resp. $\tilde{\lambda}_1,\tilde{\lambda}_2$) with $u_1$ (resp. $v_1$) negative. Then by Lemma~\ref{lem:trace_as_sum_of_ta} and Equation~(\ref{eq:elliptic_sum_one_oneball}), under the embedding $\sigma_1$,
  \[
    \lvert \trace(M_1M_2)\rvert = \lvert \sum_{j=1,2}\sum_{k=1,2} \lambda_j\tilde{\lambda}_k\ta{u_j,v_k}\rvert \le  \sum_{j=1,2}\sum_{k=1,2} \lvert\ta{u_j,v_k}\rvert = 4 \ta{u_1,v_1} -2 
  \]
  Similarly, under $\sigma'_1$ we have $\lvert \trace(M_1M_2)\rvert \le 4 \ta{u'_1,v'_1} -2$. Therefore, $\lvert \trace(M_1M_2)\rvert^2 \le (4\cosh^2(R/2)-2)^2$ under $\sigma_1,\sigma'_1$. Since we always have $\lvert \trace(M_1M_2)\rvert \le 2$ under $\sigma_j,\sigma'_j$ for $j>1$, we conclude that the norm of $\trace(M_1M_2)$ is bounded. Now when the discriminant $D$ is large, $K_0(\trace(M_1M_2))$ is a field with bounded discriminant, so it has to be a proper subfield of $K$, which can only be $K_0$. In other words, $\trace(M_1M_2)\in \OO_{K_0}$ when $D$ is large. 

  Next, consider the $K_0$-algebra $\myquat$ generated by $M_1,M_2$. By the characteristic polynomial of $M_2M_1$, 
  \[
    M_2M_1 = \trace(M_2M_1) - (M_2M_1)^{-1} = \trace(M_2M_1) - (\trace(M_1)-M_1)(\trace(M_2)-M_2).
  \]
  As $\trace(M_1M_2)\in \OO_{K_0}$, $\myquat$ is a quaternion algebra over $K_0$, and the algebra generated by $M_1,M_2$ over $\OO_{K_0}$ is a quaternion order inside $\myquat$ with bounded discriminant (which can be computed by the formula in \cite[Ex. 15.9]{voight2021quaternion} using $\trace(M_1),\trace(M_2),\trace(M_1M_2)$). Therefore, $\norm(\det(B)/\cont^2(B))$ is bounded by Proposition~\ref{prop:quaternion_global_disc_compute}. 

  Now we construct the special curve we want. $\UM_2(h,F):=\{M\in \mat_2(F): M^H J_h = J_h M^* \}$ is an $K$-subalgebra of $\mat_2(F)$ that contains $\SU(h)$. Therefore, we have $\myquat\subset \UM_2(h,F)$. We claim that $\myquat\otimes_{K_0}K \to \UM_2(h,F)$ is an isomorphism of $K$-algebra: note that $\myquat$ only contains elements whose trace is in $K_0$, so this map given by $M\otimes c \mapsto cM$ is an inclusion between two $K$-algebras of dimension 4. Now consider the automorphism of the $K$-algebra $\myquat\otimes_{\QQ(\sqrt{d})}K$ given by applying the involution $\iota:F\to F$ to the first factor, i.e. $\tilde{\iota}:M\otimes c \mapsto \iota(M)\otimes c$. By the Skolem-Noether theorem, we can find $B\in \UM_2(h,F)^\times$ such that $\tilde{\iota}(M) = BMB^{-1}$ for any $M\in \myquat\otimes_{K_0}K$. Note that $\tilde{\iota}$ is different from $\iota$, although they coincide on $\myquat$. 

  As $B$ is only defined up to a scalar, we can assume that it is integral and primitive. To see that $F_B$ is a special curve in the sense of Definition~\ref{def:special_curve}, it is left to find a $B$ with $\iota(B)=B^*$. Since for any $M\in \myquat$ we have $\iota(M)=BMB^{-1}$ and $\iota^2=1$, we have $M=\iota^2(M) = \iota(B)BMB^{-1}\iota(B)^{-1}$. Therefore, $\iota(B)B$ is a scalar in $K$, i.e. there is a $\lambda\in K$ such that $\iota(B)=\lambda B^*$. Then $B = \iota^2(B)=\iota(\lambda) \iota(B)^*=\iota(\lambda)\lambda B $, so $\iota(\lambda)\lambda=1$. Now by Hilbert's Satz 90, we can write $\lambda=\lambda_0/\iota(\lambda_0)$, so by letting $\lambda_0 B$ be the new $B$ we can assume $\iota(B)=B^*$, and then $F_B:=\{(w,Bw)\in \BB^1\times\BB^1: w\in\BB^1\}$ is the special curve we want. 

  To finish the proof of Proposition~\ref{prop:repulsion_hms}, it is left to check that $F_B$ really contains our elliptic points $u$ and $v$. Since $u=(u_1,u'_1)$ is stabilized by $M_1$, $M_1u_1 = \lambda u_1$ for some eigenvalue $\lambda$. Then note that by the definition of $B$:
  \[
    M_1'Bu_1 = BM_1u_1 = \lambda Bu_1.
  \] 
  Thus, $Bu_1$ is an eigenvector of $M_1'$, so $Bu_1=u'_1$ in $\BB^1$ and $u\in F_B$. The same argument applies to $v$.

\end{proof}

A similar strategy applies to the axes of loxodromic elements. If we fix an embedding of a loxodromic element $M$ with eigenvectors $u_1,u_2\in V_0$, the real geodesic $\widehat{u_1u_2}$ is called the axis of $M$. In our setting, the two embeddings $\sigma_1,\sigma'_1$ of $M$ define two geodesics $\widehat{u_1u_2}$ and $\widehat{u_1'u_2'}$, we still call the product of these two real geodesics the axis of $M$: 
\[
  \axislox_M:= \widehat{u_1u_2}\times\widehat{u_1'u_2'}.
\]
We first have the following analogue of Lemma~\ref{lem:trace_as_sum_of_ta} for loxodromic elements.

\begin{lemma}\label{lem:trace_as_sum_of_complex_cross_ratio}
In the $\BB^1$ setting, consider loxodromic $M_1,M_2\in \SU(h)$. For $M_1$ (resp. $M_2$) we can pick eigenvectors $u_1,u_2$ (resp. $v_1, v_2$) corresponding to eigenvalues $\lambda_1,\lambda_2$ (resp. $\tilde{\lambda}_1, \tilde{\lambda}_2$). Then
  \[
    \trace(M_1M_2) = \sum_{j,k=1}^2 \lambda_j\tilde{\lambda}_k\cdot [v_{3-k},u_{3-j},u_j,v_k]=(\lambda_1-\lambda_2)(\tilde{\lambda}_1-\tilde{\lambda}_2) [v_2,u_2,u_1,v_1] +(\lambda_1\tilde{\lambda}_2+\lambda_2\tilde{\lambda}_1).
  \]
\end{lemma}
\begin{proof}
  We can apply (\ref{eq:lox_sum_one_oneball}) to $v_k$ (resp. $u_k$ ) to get the following useful equation for any fixed $k\in\{1,2\}$:
  \begin{equation}\label{eq:expand}
    \sum_{j=1}^2 \ta{v_k, u_j} = 1  \ \ \  \left(\text{resp. }\sum_{j=1}^2 \ta{u_k, v_j}=1\right).
  \end{equation}
  To compute $\trace(M_1M_2)$, we use the orthogonal relations in Corollary~\ref{cor:classification_of_isometries_oneball}:
  \[
    \left(u_1 u_2\right)^{-1} = \begin{pmatrix}0 & 1/\langle u_1,u_2\rangle  \\  1/\langle u_2,u_1\rangle & 0 \\ \end{pmatrix} \left(u_1 u_2\right)^H J_h
  \]
  Using this we have:
  \begin{align*}
    &\trace(M_1M_2)\\ 
    =& \trace\Bigg( \begin{pmatrix}0 & \lambda_1/\langle u_1,u_2\rangle  \\  \lambda_2/\langle u_2,u_1\rangle & 0 \\ \end{pmatrix} \left(u_1 u_2\right)^H J_h \left(v_1 v_2\right) \\
    & \qquad  \begin{pmatrix}0 & \tilde{\lambda}_1/\langle v_1,v_2\rangle  \\  \tilde{\lambda}_2/\langle v_2,v_1\rangle & 0 \\ \end{pmatrix} \left(v_1 v_2\right)^H J_h \left(u_1 u_2\right) \Bigg)\\
    =&\sum_{j,k=1}^n \lambda_j\tilde{\lambda}_k [v_{3-k},u_{3-j},u_j,v_k] \\
  \end{align*}
\end{proof}

Now we can prove the repulsion for axes of loxodromic elements. We first reduce the problem to the case related to special curves. 

\begin{proposition}\label{prop:repulsion_hms_axes_of_lox}
  Follow the notation in (\ref{notation_hms_smallsubfield}). Given $t_0>0$ and $R>0$, there exists $D_0(R,t_0)$ and $N_0(R,t_0)$ such that when $D>D_0(R,t_0)$, then for any two distinct axes of loxodromic elements $\axislox_{M_1},\axislox_{M_2}$ with $\lvert \trace(M_1)\rvert,\lvert\trace(M_2)\rvert\le t_0$ under all embeddings, either their distance is greater than $R$, or $M_1,M_2$ are in the stabilizer subgroup of a special curve $F_B$ for some $B$ with $\norm(\det(B)/\cont^2(B)) < N_0(R,t_0)$. 
\end{proposition}
\begin{proof}
  The idea is similar to the proof of Proposition~\ref{prop:repulsion_hms} of the repulsion of elliptic points. Here we will show that if the distance is bounded, by analyzing the subfield where the distance function lies, we conclude that when the discriminant is large, we can construct a special curve we want by using the two elements to generate a quaternion order. 

  Assume $\rho(\axislox_{M_1},\axislox_{M_2}) \le R$. For $M_1$ (resp. $M_2$) we pick eigenvectors $u_1,u_2$ (resp. $v_1, v_2$) corresponding to eigenvalues $\lambda_1,\lambda_2$ (resp. $\tilde{\lambda}_1, \tilde{\lambda}_2$). 

  We first bound $\trace(M_1M_2)$. Let $x$ be $[v_2,u_2,u_1,v_1]$, so $[v_2,u_1,u_2,v_1]=1-x$ by (\ref{eq:lox_sum_one_oneball_cross_ratio_ver}). Then $\rho(\axislox_{M_1},\axislox_{M_2}) \le 2R$ means 
  \[
    2\min(\lvert x\rvert, \lvert 1-x\rvert) \le \lvert x\rvert + \lvert 1-x\rvert = \cosh(\rho(\axislox_{M_1},\axislox_{M_2})) \le \cosh(2R),
  \]
  and with Lemma~\ref{lem:trace_as_sum_of_complex_cross_ratio} we have
  \[
    \lvert \trace(M_1M_2)\rvert \le \lvert\trace(M_1)\rvert \lvert\trace(M_2)\rvert (\min(\lvert x\rvert, \lvert 1-x\rvert) +1) \le t_0^2(\cosh(2R)/2+1). 
  \]
  This is true for both embeddings $\sigma_1,\sigma'_1$, and we also always have $\lvert \trace(M_1M_2)\rvert \le 2$ under $\sigma_j,\sigma'_j$ for $j>1$. Then there exists $D_0(R,t_0)$ such that when $D > D_0(R,t_0)$, $\trace(M_1M_2)\in \OO_{K_0}$. 

  Now the same as the argument in the proof of Proposition~\ref{prop:repulsion_hms}, the $K_0$-algebra generated by $M_1,M_2$ is a quaternion algebra over $K_0$, whose discriminant is bounded by the discriminant of the quaternion $\OO_{K_0}$-order generated by $M_1,M_2$, which is bounded. Then we can construct a special curve $F_B$ for this quaternion algebra, and $M_1,M_2$ are in the stabilizer of this special curve. 
\end{proof}

Now we will conclude the repulsion for axes of loxodromic elements of small traces outside the neighborhoods of special curves. 

\begin{proposition}\label{prop:repulsion_hms_axes_of_lox_common_special}
  Follow the notation in (\ref{notation_hms_smallsubfield}). Given $t_0>0$ and $R>0$, there exists $D_0(R,t_0)$ and $r_0(R,t_0)$ such that when $D>D_0(R,t_0)$, then for a special curve $F_B$, two distinct axes of loxodromic elements $\axislox_{M_1},\axislox_{M_2}$ with $\lvert \trace(M_1)\rvert,\lvert\trace(M_2)\rvert\le t_0$ under all embeddings inside the stabilizer subgroup of $F_B$, and any point $(z,z')$ on $\axislox_{M_1}$ with $\rho((z,z'), F_B)>r_0(R,t_0)$, we have $\rho((z,z'), \axislox_{M_2}) > R$. 
\end{proposition}

\begin{proof}
  For $M_1$ (resp. $M_2$) we pick eigenvectors $u_1,u_2$ (resp. $v_1, v_2$) corresponding to eigenvalues $\lambda_1,\lambda_2$ (resp. $\tilde{\lambda}_1, \tilde{\lambda}_2$). Since $M_1,M_2$ are in the stabilizer subgroup of $F_B$, the minimal distance between $\axislox_{M_1},\axislox_{M_2}$ is obtained on $F_B$: if on the first component $z_1\in \widehat{u_1u_2}$ and $z_2\in \widehat{v_1v_2}$ give the minimal distance, then since $B$ is an isometry, $(z_1,Bz_1), (z_2,Bz_2)$ give the minimal distance between $\axislox_{M_1},\axislox_{M_2}$. Now consider any point $(z,z')$ on $\axislox_{M_1}$. $\rho((z,z'), (z_1,Bz_1))=\max(\rho(z,z_1), \rho(B^{-1}z',z_1))\ge \rho((z,z'), F_B)/2$. Let $z_m\in\{z,B^{-1}z'\}$ be the one with max $\rho(z_m,z_1)$. Then it reduces to a problem on $\BB^1$. 

  Now consider the distance between $z_m\in \widehat{u_1u_2}$ and $\widehat{v_1v_2}$. Let $w_0\in\widehat{v_1v_2}$ be the point that gives the minimal distance between $z_m$ and $\widehat{v_1v_2}$. Then if $\widehat{u_1u_2}$ and $\widehat{v_1v_2}$ do not intersect, $z_m,w_0,z_2,z_1$ form a Lambert quadrilateral; if $\widehat{u_1u_2}$ and $\widehat{v_1v_2}$ intersect, $z_m,w_0,z_2=z_1$ form a triangle. In both cases, hyperbolic trigonometry says that $\sinh(\overline{z_mw_0})$ is proportional to the length of the complex cross-ratio multiplied by $\cosh(\overline{z_mz_1})$ or $\sinh(\overline{z_mz_1})$. Concretely, for the disjoint case, $\sinh(\overline{z_mw_0})=\sinh(\overline{z_1z_2})\cosh(\overline{z_mz_1})$, so to finish the proof, it suffices to notice that there is a lower bound for the distance $\overline{z_1z_2}$ when it is nonzero. To see this, note that by Lemma~\ref{lem:trace_as_sum_of_complex_cross_ratio}, 
  \begin{align*}
    \lvert\trace(M_1M_2-M_1M_2^{-1})\rvert^2 &= \lvert (\lambda_1-\lambda_2)(\tilde{\lambda}_1-\tilde{\lambda}_2)([v_2,u_2,u_1,v_1]-[v_1,u_2,u_1,v_2]) \rvert^2 \\
    &= (\trace(M_1)^2-4)(\trace(M_2)^2-4) \cosh^2(\overline{z_1z_2}).
  \end{align*}
  Then $\cosh^2(\overline{z_1z_2})$ as a fraction with bounded numerator and denominator in all embeddings has a lower bound under $\sigma_1$ when it is nonzero. The intersection case is similar. 
\end{proof}

We are left to show similar results for an elliptic point and an axis of a loxodromic element. 

\begin{lemma}\label{lem:trace_as_sum_of_etas}
In the $\BB^1$ setting, consider elliptic $M_1$ and loxodromic $M_2\in \SU(h)$. For $M_1$ (resp. $M_2$) we can pick eigenvectors $u_1,u_2$ (resp. $v_1, v_2$) corresponding to eigenvalues $\lambda_1,\lambda_2$ (resp. $\tilde{\lambda}_1, \tilde{\lambda}_2$). Then
  \[
    \trace(M_1M_2) = \sum_{j,k=1}^2 \lambda_j\tilde{\lambda}_k\cdot \eta(v_k,v_{3-k},u_j).
  \]
\end{lemma}
\begin{proof}
  Similar to the proof of Lemma~\ref{lem:trace_as_sum_of_complex_cross_ratio} and \ref{lem:trace_as_sum_of_ta}:
  \begin{align*}
    &\trace(M_1M_2)\\ 
    =& \trace\Bigg( \begin{pmatrix}\lambda_1/\langle u_1,u_1\rangle & 0 \\ 0 & \lambda_2/\langle u_2,u_2\rangle  \\ \end{pmatrix} \left(u_1 u_2\right)^H J_h \left(v_1 v_2\right) \\
    & \qquad  \begin{pmatrix}0 & \tilde{\lambda}_1/\langle v_1,v_2\rangle  \\  \tilde{\lambda}_2/\langle v_2,v_1\rangle & 0 \\ \end{pmatrix} \left(v_1 v_2\right)^H J_h \left(u_1 u_2\right) \Bigg)\\
    =&\sum_{j,k=1}^2 \lambda_j\tilde{\lambda}_k \eta(v_k,v_{3-k},u_j)
  \end{align*}
\end{proof}

\begin{proposition}\label{prop:repulsion_hms_elliptic_and_lox}
  Follow the notation in (\ref{notation_hms_smallsubfield}). Given $t_0>0$ and $R>0$, there exists $D_0(R,t_0)$ and $N_0(R,t_0)$ such that when $D>D_0(R,t_0)$, then for any elliptic point $u$ stabilized by $M_1$ and the axis of a loxodromic element $\axislox_{M_2}$ with $\lvert\trace(M_2)\rvert\le t_0$ under all embeddings, either their distance is greater than $R$, or $M_1,M_2$ are in the stabilizer subgroup of a special curve $F_B$ for some $B$ with $\norm(\det(B)/\cont^2(B)) < N_0(R,t_0)$. 
\end{proposition}
\begin{proof}
  The proof is the same as that of Proposition~\ref{prop:repulsion_hms} and \ref{prop:repulsion_hms_axes_of_lox}. If the distance is bounded, then by Lemma~\ref{lem:trace_as_sum_of_etas}, $\trace(M_1M_2)$ is bounded by a constant determined by $t_0$ and $R$, and since we also always have $\lvert \trace(M_1M_2)\rvert \le 2$ under $\sigma_j,\sigma'_j$ for $j>1$, there exists $D_0(R,t_0)$ such that when $D > D_0(R,t_0)$, $\trace(M_1M_2)\in \OO_{K_0}$. Finally, the $K_0$-algebra generated by $M_1,M_2$ is a quaternion algebra over $K_0$, whose discriminant is bounded by the discriminant of the quaternion $\OO_{K_0}$-order generated by $M_1,M_2$, which is bounded. Then we obtain a special curve $F_B$ for this quaternion algebra, and $M_1,M_2$ are in the stabilizer of this special curve. 
\end{proof}

Then we will show the repulsion related to the special curves. The distance between special curves is given by the following. 

\begin{lemma}\label{lem:distance_between_special_curves}
  Follow the notation in (\ref{notation_hms_smallsubfield}). On $\BB^1\times\BB^1$, the special curves $F_{B_1}$ and $F_{B_2}$ satisfy:
  \begin{enumerate}
    \item When $B_1^*B_2$ is elliptic and non-scalar, $F_{B_1}$ and $F_{B_2}$ intersect at the point stabilized by $B_1^*B_2$;
    \item When $B_1^*B_2$ is loxodromic, 
    \[
      \cosh^2\left(\frac{\rho(F_{B_1},F_{B_2})}{2}\right)=\frac{\trace(B_1^*B_2)^2}{4\det(B_1^*B_2)}.
    \]
  \end{enumerate}
\end{lemma}
\begin{proof}
  Note that since the distance between a point $(u_1,u'_1)$ and a special curve $F_B$ can be described by $\ta{Bu_1,u'_1}$, the distance between two special curves can be described by 
  \[
    \min_{u\in\BB^1} \ta{B_1u,B_2u} = \min_{u\in\BB^1} \ta{u,B_1^*B_2u}.
  \]
  First note that in general, for $M\in \UM_{n+1}(h)$ and $u\in \BB^n$,
  \begin{equation} \label{eq:ta_of_lox_as_absolute_square}
    \ta{u,Mu}:=\frac{\langle u, Mu \rangle \langle Mu, u \rangle}{\langle u, u \rangle \langle Mu, Mu \rangle}
    = \frac{1}{\det M} \left\lvert \frac{\langle Mu, u \rangle}{\langle u, u \rangle} \right\rvert^2.
  \end{equation}
  When $n=1$ and $M$ is loxodromic, let $u_1,u_2$ be eigenvectors of $M$ corresponding to eigenvalues $\lambda_1,\lambda_2$, and then Equation~\ref{eq:lox_expand_oneball} and \ref{eq:lox_sum_one_oneball} gives:
  \begin{align*}
    Mu &=\lambda_1 \frac{\langle u, u_2\rangle}{\langle u_1, u_2\rangle}u_1 + \lambda_2 \frac{\langle u, u_1\rangle}{\langle u_2, u_1\rangle}u_2 \\
    \left\lvert \frac{\langle Mu, u \rangle}{\langle u, u \rangle} \right\rvert
    &= \left\lvert \lambda_1 \frac{\langle u, u_2\rangle\langle u_1, u\rangle}{\langle u_1, u_2\rangle\langle u, u\rangle} + \lambda_2 \frac{\langle u, u_1\rangle\langle u_2, u\rangle}{\langle u_2, u_1\rangle\langle u, u\rangle} \right\rvert \\
    &\ge \left\lvert\Re\left(\lambda_1 \frac{\langle u, u_2\rangle\langle u_1, u\rangle}{\langle u_1, u_2\rangle\langle u, u\rangle} + \lambda_2 \frac{\langle u, u_1\rangle\langle u_2, u\rangle}{\langle u_2, u_1\rangle\langle u, u\rangle}\right) \right\rvert \\
    &= \frac{\lvert \trace(M) \rvert }{2},
  \end{align*}
  and the equality can be obtained when $u$ lies on the real geodesic $\widehat{u_1u_2}$ by Lemma~\ref{lem:distance_formulas_oneball}(b). 

  Now come back to the proof. For distinct $F_{B_1}$ and $F_{B_2}$, if $B_1^*B_2\in \UM_2(h,\OO_F)$ is an elliptic element stabilizing $u\in\BB^1$ under $\sigma_1$, then equivalently it means that $B_1\cdot u=B_2\cdot u$, so $F_{B_1}$ and $F_{B_2}$ both pass through $(u,B_1u)$ and $B_1^*B_2\cdot (u,B_1u)=(B_1^*B_2\cdot u,B_1B_2^*\cdot B_2u)=(u,B_1u)$. Now assume that $B_1^*B_2$ is neither elliptic nor a scalar, i.e. it is loxodromic. Then combine the above we have:
  \[
    \min_{u\in\BB^1} \ta{u,B_1^*B_2u} 
    = \min_{u\in\BB^1} \frac{1}{\det(B_1^*B_2)} \left\lvert \frac{\langle B_1^*B_2u, u \rangle}{\langle u, u \rangle} \right\rvert^2
    = \frac{\trace(B_1^*B_2)^2}{4\det(B_1^*B_2)}.
  \]
\end{proof}

The next lemma includes the essential part of the computation for repulsion related to special curves. 

\begin{lemma}\label{lem:technique_for_special_curves_repulsion}
  For a quadratic extension of number fields $K/K_0$ with Galois conjugation $x\mapsto x'$, consider $B_j\in \mat_2(\OO_K)$ with $B_j^*={B_j}'$, $j=1,2$. Then as ideals
  \[
    16\frac{\trace(B_1^*B_2)^2 - 4\det(B_1^*B_2)}{\cont^2(B_1)\cont^2(B_2)} \subseteq \disc(K/K_0).
  \]
\end{lemma}
\begin{proof}
  By their characteristic polynomials, the scalar matrix:
  \begin{align*}
    &\trace(B_1^*B_2)^2 - 4\det(B_1^*B_2) \\
    =&(B_1^*B_2+B_2^*B_1)(B_1^*B_2+B_2^*B_1) - 4\det(B_1^*B_2) \\ 
    =&(B_1^*B_2-B_2^*B_1)(B_1^*B_2-B_2^*B_1). 
  \end{align*}

For $B = B_1,B_2$, since $B^*=B'$, we have $2B = (B+B^*)+(B-B^*)=\trace(B) + (B-B')$, i.e. $2B \equiv \trace(B)\mod \different_{K/K_0}$, the different ideal of $K/K_0$. Therefore, 
  \[
    \frac{4(B_1^*B_2-B_2^*B_1)}{\cont(B_1)\cont(B_2)} \equiv \frac{\trace(B_1)\trace(B_2)-\trace(B_2)\trace(B_1)}{\cont(B_1)\cont(B_2)} =0 \mod \different_{K/K_0},
  \]
  i.e. entries (as ideals) of $4(B_1^*B_2-B_2^*B_1)/(\cont(B_1)\cont(B_2))$ are all contained in $\different_{K/K_0}$, so $16(\trace(B_1^*B_2)^2 - 4\det(B_1^*B_2))/(\cont^2(B_1)\cont^2(B_2))$ is in $\different_{K/K_0}^2=\disc(K/K_0)$. 
\end{proof}

Now we can show repulsions related to special curves. 

\begin{proposition}[Repulsion of special curves]\label{prop:repulsion_curves_hms}
  Follow the notation in (\ref{notation_hms_smallsubfield}). Given $N_0>0$ and $R>0$, there exists $D_0(R,N_0)$ such that on $\BB^1\times\BB^1$, when $D>D_0(R,N_0)$, the distance between any two distinct special curves $F_{B_1},F_{B_2}$ with $\norm_{K_0/\QQ}(\det(B_1)/\cont^2(B_1)),\norm_{K_0/\QQ}(\det(B_2)/\cont^2(B_2))\le N_0$ is larger than $R$. 
\end{proposition}
\begin{proof}
  When $B_1^*B_2$ is not loxodromic, $\trace(B_1^*B_2)^2-4\det(B_1^*B_2)\le 0$. Since the norm of $\det(B_1^*B_2)/\cont^2(B_1^*B_2)$ is bounded, by Lemma~\ref{lem:technique_for_special_curves_repulsion}, when the discriminant $D$ is large, $\trace(B_1^*B_2)^2-4\det(B_1^*B_2)$ has to be nonnegative and thus $F_{B_1},F_{B_2}$ overlap. 

  When $B_1^*B_2$ is loxodromic, by Lemma~\ref{lem:technique_for_special_curves_repulsion}, the absolute norm of their quotient
  \[
    \norm_{K_0/\QQ}^2(16(\frac{\trace(B_1^*B_2)^2}{4\det(B_1^*B_2)} - 1)) \ge 
    \frac{\disc(K/\QQ)}{N_0^2}.
  \]
   Since $\displaystyle \frac{\trace(B_1^*B_2)^2}{4\det(B_1^*B_2)}$ is bounded between 0 and 1 under other embeddings $\sigma_j$ for $j>1$, we can conclude that $\displaystyle \left(\frac{\trace(B_1^*B_2)^2}{4\det(B_1^*B_2)}\right)^{\sigma_1}$ is larger than any given $R$ when $D$ is large enough. 
\end{proof}

The following repulsion result says that small injectivity radii for points on a special curve can only be obtained from elements in the stabilizer subgroup when the discriminant is large. By Proposition~\ref{prop:inj_radius_ell} and \ref{prop:inj_radius_lox_oneball}, it can be stated as repulsion of the corresponding geometric objects. 

\begin{corollary}\label{cor:repulsion_hms_small_inj_radius_and_special}
  Follow the notation in (\ref{notation_hms_smallsubfield}). Given $N_0>0$ and $R>0$, there exists $D_0(R,N_0)$ such that when $D>D_0(R,N_0)$, for any $F_{B}$ with $\norm_{K_0/\QQ}(\det(B)/\cont^2(B))\le N_0$, any $M\in\Gamma$ with $\rho(z, M\cdot z)<R$ for some $z\in F_B$ must be in the stabilizer subgroup of $F_B$. 
\end{corollary}
\begin{proof}
  Note that $M\cdot F_B=F_{M'BM^{-1}}$ and $\rho(z, M\cdot z)<R$ gives an upper bound on the distance $\rho(F_B,F_{M'BM_{-1}})$. Now by Proposition~\ref{prop:repulsion_curves_hms}, when $D$ is large, they have to overlap. 
\end{proof}

\subsection{Injectivity radius}\label{subsec:inj_radius}
Recall that the injectivity radius $\rho_X(x)$ of a point $x\in X$ is defined as half of the length of the shortest closed geodesic through $x$, or equivalently, half of the smallest distance between preimages of $x$ on the universal cover. In our setting, we do not have a uniform lower bound of injectivity radii. The previous section about repulsion phenomena can be used to describe distances between elliptic points, special curves, or axes of loxodromic elements that are equivalent under the $\Gamma$-action on the universal cover when the discriminant $D$ is large. In this section, we will see how other points with small injectivity radius are related to these objects. 

To compute the injectivity radius, note that by Equation~\ref{eq:ta_of_lox_as_absolute_square}, it suffices to compute $\left|\frac{\langle Mu, u\rangle}{\langle u, u\rangle} \right|$.

To understand points with small injectivity radius, we divide them into cases by what types of elements cause a small injectivity radius. 

If the closed geodesics of a point given by an elliptic element $M$ is small, then this point should be close to the locus $M$ fixes. The geometric intuition of elliptic points says that locally it is a rotation, and given the angle of rotation and the arc length, we may at least bound the radius. This can be made precise as follows. 

\begin{proposition}\label{prop:inj_radius_ell}
	For any given $R>0$, if $\rho_{\BB^n}(u,Mu)<R$ for some $u\in\BB^n$ and some elliptic element $M$ of order $N$, then 
  \[
    \rho_{\BB^2}(u,\FixLocus(M))< \frac{1}{\sqrt{2(1-\cos(2\pi/N))}}\cdot R.
  \]
\end{proposition}
\begin{proof}
	Let $V_1=\FixLocus(M),V_2,\cdots,V_m$ be pairwise orthogonal eigenspaces of $M$ and $\lambda_1,\cdots,\lambda_m$ be their corresponding distinct eigenvalues. Recall that the eigenvalues are $N$-th roots of unity, $\ta{u,V_1}\ge 1$, and $\ta{u,V_j}\le 0$ for $j>1$. Apply (\ref{eq:elliptic_expand}) to $u$ and we have
	\begin{align*}
		\left|\frac{\langle Mu,u\rangle}{\langle u,u\rangle}\right|&=\left| \lambda_1\ta{u,V_1}+\cdots+\lambda_m\ta{u,V_m} \right| \\
		&= |\ta{u,V_1} + \lambda_2/\lambda_1 \ta{u,V_2}+\cdots + \lambda_m/\lambda_1 \ta{u,V_m} | \\
		&\ge \ta{u,V_1} + \Re(\lambda_2/\lambda_1 \ta{u,V_2} +\cdots+ \lambda_m/\lambda_1 \ta{u,V_m} ) \\
		&\ge \ta{u,V_1} + \cos(2\pi/N) (\ta{u,V_2}+\cdots+\ta{u,V_m}) \\
    &= (1 - \cos(2\pi/N)) \ta{u,V_1} + \cos(2\pi/N).
	\end{align*}
	The geometric intuition for this proposition is that given the arc length of a rotation by $\theta=2\pi/N$, we want to find the radius of this rotation. Inspired by hyperbolic trigonometry, we have the following computation:
  \[
    \sinh^2\left(\frac{\rho_{\BB^2}(u,v_1)}{2}\right) = \ta{u,V_1} - 1\le  \frac{\left|\frac{\langle Mu,u\rangle}{\langle u,u\rangle}\right|-1}{1-\cos(2\pi/N)}\le \frac{2\sinh^2(R/4)}{1-\cos(2\pi/N)}   .
  \]
  Therefore, by calculus,
	\[
		\sinh\left(\frac{\rho_{\BB^2}(u,V_1)}{2}\right) \le \sqrt{\frac{2}{1-\cos(2\pi/N)}} \sinh\left(\frac{R}{4}\right) \le \sinh\left( \sqrt{\frac{2}{1-\cos(2\pi/N)}}\frac{R}{4}\right).
	\]
\end{proof}

Then we come to the injectivity radius of loxodromic elements. For \ref{notation_twoball}, loxodromic elements always give large injectivity radii as follows.

\begin{proposition}\label{prop:inj_radius_lox_twoball}
Follow the notation in (\ref{notation_twoball}). For any $R>0$, there exists $D_0(R)>0$ such that when $D>D_0(R)$, for any $x\in\BB^2$ and $M\in\Gamma$ loxodromic, $\rho_{\BB^2}(x,Mx)>R$.
\end{proposition}
\begin{proof}
For a loxodromic element $M$, by Theorem~\ref{thm:classification_of_isometries} we write its eigenvalues as $re^{i\theta}$, $r^{-1}e^{i\theta}$ and $e^{-2i\theta}$ with $r >1$, and if their respective eigenvectors are $v_1,v_2,v_3$, then $1 = \eta(v_1,v_2,u)+\overline{\eta(v_1,v_2,u)}+\ta{u,v_3}$ for any $u\in V\backslash V_0$. Suppose that $x\in X$ is represented by $u\in V_-$. Since $u\in V_-$ and $v_3\in V_+$, $\ta{u,v_3} \le 0$, and thus $\Re(\eta(v_1,v_2,u)) \ge 1/2$. 

First note that when the discriminant $D$ is large, $r$ has to be large as follows. Since $M$ is loxodromic, $\lvert\trace(M)\rvert>3$ under $\sigma_1$, and we always have $\lvert\trace(M)\rvert\le 3$ under $\sigma_j$ for $j>1$. Then we can conclude that $\QQ(\trace(M))=K$, since if $\QQ(\trace(M))$ is a proper subfield of $K$, $\trace(M)^{\sigma_1}$ should coincide with $\trace(M)^{\sigma_j}$ for some $j>1$, which leads to contradiction. Therefore, $\lvert\trace(M)^{\sigma_1}\rvert \ge D^{\frac{1}{\ell(\ell-1)}}/2$, and $r=\lvert re^{i\theta}\rvert \ge \lvert\trace(M)^{\sigma_1}\rvert -2 \ge D^{\frac{1}{\ell(\ell-1)}}/2-2$. 

Now we compute the injectivity radius. When $D$ is large so that $r>2$, under $\sigma_1$,

\begin{align*}
\left|\frac{\langle Mu, u\rangle}{\langle u, u\rangle} \right| 
=& \left| re^{i\theta} \eta(v_1,v_2,u) + r^{-1}e^{i\theta} \overline{\eta(v_1,v_2,u)} + e^{-2i\theta} \ta{u,v_3} \right| \\
\ge & \left| re^{i\theta} \eta(v_1,v_2,u) + r^{-1}e^{i\theta} \overline{\eta(v_1,v_2,u)} \right| - \left| e^{-2i\theta} \ta{u,v_3} \right| \\
\ge & \Re(r \eta(v_1,v_2,u) + r^{-1} \overline{\eta(v_1,v_2,u)})+ \ta{u,v_3} \\
=& \Re(\eta(v_1,v_2,u))(r+r^{-1}) +(1-2\Re(\eta(v_1,v_2,u))) \\
\ge & (|\trace(M)|-3)/2 +1 \\
\ge & \frac{D^{\frac{1}{\ell(\ell-1)}}-2}{4}.
\end{align*}
Therefore, we can have $\rho_{\BB^2}(u,Mu) > R$ for any given $R>0$ as long as $D$ is large enough. 

\end{proof}

For \ref{notation_hms_smallsubfield}, there may not be large loxodromic injectivity radii. We will show that when the discriminant is large, a point with a small loxodromic injectivity radius implies that it is close to the axis of the corresponding loxodromic element. 

\begin{proposition}\label{prop:inj_radius_lox_oneball}
  Follow the notation in (\ref{notation_hms_smallsubfield}). For any given $R>0$, there exists $D_0(R)$ such that when $D>D_0(R)$, if $\rho_{\BB^1\times \BB^1}(z,Mz)<R$ for some $z=(u,u')\in\BB^1\times\BB^1$ and $M\in\Gamma$ loxodromic, then $\rho_{\BB^1\times \BB^1}(z,\axislox_M)<\cosh^{-1}(2^{\ell} \cosh(R/2))$.
\end{proposition}
\begin{proof}
  Under the embedding $\sigma_1$, we write the eigenvalues of $M$ as $r$, $r^{-1}$ with $|r| >1$, and if their respective eigenvectors are $v_1,v_2$, then $1 = \eta(v_1,v_2,u)+\overline{\eta(v_1,v_2,u)}$ for any $u\in V\backslash V_0$. Then
  \[
    \cosh\left(\frac{R}{2}\right) > \cosh\left(\frac{\rho_{\BB^1}(u,Mu)}{2}\right)= \left|\frac{\langle Mu, u\rangle}{\langle u, u\rangle} \right| = \left| r \eta(v_1,v_2,u) + r^{-1}\overline{\eta(v_1,v_2,u)}\right| \ge \frac{|\trace(M)|}{2} . 
  \]
  We have the same inequality under $\sigma'_1$, so when the discriminant $D>D_1(R)$ for some $D_0(R)$, $\trace(M)\in \OO_{K_0}$. Then under $\sigma_1$ we have
  \[
    \cosh\left(\frac{R}{2}\right) > \left| r \eta(v_1,v_2,u) + r^{-1}\overline{\eta(v_1,v_2,u)}\right| \ge |r-r^{-1}|\cdot |\eta(v_1,v_2,u)| . 
  \]
  Since $|\trace(M)| < 2$ under $\sigma_j,\sigma'_j$ for $j>1$, we at least have:
  \[
    |r-r^{-1}|^2 = (\trace(M)^2-4)^{\sigma_1} =\left\lvert \frac{\norm_{K_0/\QQ}(\trace(M)^2-4)}{\prod_{j>1} (\trace(M)^2-4)^{\sigma_j}} \right\rvert > \frac{1}{4^{\ell-1}}.
  \]
  Therefore, combining these two inequalities,
  \[
    |\eta(v_1,v_2,u)| < 2^{\ell-1} \cosh\left(\frac{R}{2}\right).
  \]
  We have the same inequality under $\sigma'_1$: $|\eta(v'_1,v'_2,u')| < 2^{\ell-1} \cosh\left(\frac{R}{2}\right)$, so
  \[
    \cosh(\rho_{\BB^1\times\BB^1}(u,\axislox_M))=\max(|2\eta(v_1,v_2,u)|,|2\eta(v'_1,v'_2,u')|) < 2^{\ell} \cosh\left(\frac{R}{2}\right).
  \]
\end{proof}

\section{Multiplicity bounds}\label{sec:multiplicity_bounds}

In this section, we will conclude two kinds of multiplicity bounds for any curve $C$ inside each type of our compact arithmetic quotients. The first one is the multiplicity bound of $C$ at a point $x\in C$ in terms of the volume of $C$ in a geodesic ball centered at $x$. This is a corollary of a theorem by Hwang and To in \cite{hwang2002volumes}. The second type of multiplicity bound is about the intersection between any curve $C$ and the special curves. The proof uses the idea in \cite[Theorem 1]{hwang2012injectivity}, where a similar bound was proved for the intersection between any curve and the diagonal inside a product of hyperbolic curves. We will show the construction in detail for hyperbolic 2-balls, while similar results for products of hyperbolic 1-balls have been proved in \cite{hwang2012injectivity} and \cite{bakker2016p}. 

We summarize the volume inequalities in this section in the following table.
\begin{center}
  \begin{tabular}{|c|c|c|}\hline 
    Volume inequalities for neighborhood of... & with multiplicities & relative version \\ \hline
    Points in \ref{notation_twoball} & \ref{thm:point_mult_bound} & \ref{prop:point_relative_vol} \\
    Special curves in \ref{notation_twoball} & \ref{thm:tubular_mult_bound} & \ref{prop:tubular_relative_vol} \\ \hline
    Points in \ref{notation_hms_smallsubfield} & \ref{thm:point_mult_bound_hms} & \ref{prop:point_relative_vol_hms} \\
    Special curves in \ref{notation_hms_smallsubfield} & \ref{thm:special_curve_mult_bound_hms}(1) & \ref{thm:special_curve_mult_bound_hms}(2) \\
    Axes of loxodromic elements in \ref{notation_hms_smallsubfield} & & \ref{prop:axis_lox_relative_vol_hms} \\ \hline
  \end{tabular}
\end{center}

\subsection{Multiplicity bounds in \ref{notation_twoball}}\label{sub:multiplicity_bounds_in_bb_2_}

We first give the multiplicity bound of any curve $C$ at an isolated elliptic point $x\in C$, which is a corollary of \cite[Theorem 1]{hwang2002volumes}. 

\begin{theorem}\label{thm:point_mult_bound} \label{page:ball_neighborhood}
Follow the notation in (\ref{notation_twoball}). Then for any fixed $R>0$, any complex analytic curve $C\subset X$, and a point $x\in C$ with stabilizer subgroup $\Gamma_x$ and injectivity radius $>R$, let $B_{x,R}$ be the geodesic ball centered at $x$ with radius $R$. Then there exists $D_0(R)>0$ such that when $D>D_0(R)$, 
	\[
		\vol(C\cap B_{x,R}) \ge \frac{4\pi}{|\Gamma_x|} \sinh^2\left(\frac{R}{2}\right)\mult_{x}(C).
	\]
\end{theorem}
\begin{proof}
	\cite[Theorem 1]{hwang2002volumes} says that, for any complex analytic curve $\tilde{C}\subset \BB^2$ and any point $\tilde{x}\in \tilde{C}$, the geodesic ball $\tilde{B}_{\tilde{x},R}$ centered at $\tilde{x}$ with radius $R$ satisfies
	\[
		\vol(\tilde{C}\cap B_{\tilde{x},R}) \ge  4\pi \sinh^2\left(\frac{R}{2}\right)\mult_{\tilde{x}}(\tilde{C}).
	\]
	Now consider a lift of $x\in C$ in $\BB^2$ called $\tilde{x}\in \tilde{C}$. Then we use Proposition~\ref{prop:repulsion} to pick $D_0(2R)$ such that when $D>D_0(2R)$, distances between $\tilde{x}$ and $M\tilde{x}$ are greater than $2R$ for $M\not\in \Gamma_{\tilde{x}}$ the stabilizer subgroup of $\tilde{x}$. Also note that $M\in \Gamma_{\tilde{x}}$ preserves the distance to $\tilde{x}$. Therefore, when we restrict $\BB^2/\Gamma$ to $B_{\tilde{x},R}$, we obtain $B_{\tilde{x},R}/\Gamma_{\tilde{x}}$, and since $\Gamma_{\tilde{x}}$ are isometries on $B_{\tilde{x},R}$, we have
  \[
    \vol(C\cap B_{x,R}) \ge \frac{1}{|\Gamma_{\tilde{x}}|} \vol(\tilde{C}\cap B_{\tilde{x},R}) \ge \frac{4\pi}{|\Gamma_{\tilde{x}}|} \sinh^2\left(\frac{R}{2}\right)\mult_{\tilde{x}}(\tilde{C}).
  \]
\end{proof}

The other one is a multiplicity bound for the intersection of any curve $C$ with an elliptic complex line $L\subset X$ in terms of a geodesic tubular neighborhood of $L$. Most of the rest of Section~\ref{sub:multiplicity_bounds_in_bb_2_} is dedicated to prove the following theorem, while we will derive a relative version at the end of this section.

\begin{theorem} \label{thm:tubular_mult_bound}
Follow the notation in (\ref{notation_twoball}). For any complex analytic curve $C\subset X$ and a set of elliptic complex line $\{L_j\subset X\}_j$, let $W_{L,r} = \{p\in X: \rho_{\BB^2}(p,L)< r\}$ be the geodesic tubular neighborhood of $L$ with radius $r$. Let $m$ be the maximum of $|\Gamma_x|$ among $x\in \cup_j L_j$. Then for any $r$, there exists $D_0(r)>0$ such that when $D>D_0(r)$, \label{page:line_neighborhood}
	\[
		\vol(C\cap \cup_jW_{L_j,r}) \ge \frac{4\pi}{m} \sinh^2\left(\frac{r}{2}\right)\sum_j(C\cdot L_j) 
	\]
\end{theorem}

The proof will use the idea in \cite[Theorem 1]{hwang2012injectivity}. The idea here is to find a semi-positive (1,1)-form defined in the neighborhood of the complex line $L$, such that its integration against $L$ over $W_{L,r}$ is bounded by the volume of $L$ and its potential function is plurisubharmonic with desired pole orders along $L$. Then this gives a multiplicity bound through computation of the Lelong numbers. 

For an elliptic complex line $\tilde{L}\subset\BB^2$, let $\Gamma_{\tilde{L}}\subset \Gamma$ be the stabilizer group of $\tilde{L}$. Since $\tilde{L}$ is determined by its normal vector $n$, the stabilizer is in fact $\Gamma_{\tilde{L}} = \{M\in \Gamma: M\cdot n = \lambda n \text{ for some }\lambda\in\CC^* \}$ if we consider $n\in V_+$. Note that $\tilde{W}_{\tilde{L},r}$ is invariant under the action of $M\in \Gamma_{\tilde{L}}$ since $\ta{M\cdot p,n}=\ta{p,M^{-1}\cdot n}=\ta{p,n}$. 

Before we derive a result about the quotient space, we will first do the computation on $\tilde{W}_{\tilde{L},r}$ for one elliptic complex line $\tilde{L}\subset \BB^2$. The statement we will prove is the following:

\begin{proposition}\label{prop:tubular_mult_bound}
	Follow the notation in (\ref{notation_twoball}). For any elliptic complex line $\tilde{L}\subset\BB^2$ and any complex analytic curve $\tilde{C}\subset \BB^2/\Gamma_{\tilde{L}}$, let $\tilde{W}_{\tilde{L},r} = \{p\in \BB^2: \rho_{\BB^2}(p,\tilde{L})< r\}$ be the geodesic tubular neighborhood of $\tilde{L}$ with radius $r$. Then for any $r$, there exists $D_0(r)>0$ such that when $D>D_0(r)$, 
	\[\vol(\tilde{C}\cap \tilde{W}_{\tilde{L},r}/\Gamma_{\tilde{L}}) \ge 4\pi \sinh^2\left(\frac{r}{2}\right)(\tilde{C}\cdot \tilde{L}). \]
\end{proposition}

By Proposition~\ref{prop:dist_point_line}, we use
\[
	\tilde{\mu}(p)=\tanh^2\left(\frac{\rho_{\BB^2}(p,\tilde{L})}{2}\right)= \frac{-\ta{p,n}}{1-\ta{p,n}}
\]
to measure the distance to $\tilde{L}$, where $n$ is a normal vector of $L$, and the tubular neighborhood in $\BB^2$ is
\[
 	\tilde{W}_{L,r} = \{p\in \BB^2: \rho_{\BB^2}(p,L)< r\} = \{p\in \BB^2: \tilde{\mu}(p) < 2\tanh^2(r/2) \} .
\]
Since the metric is invariant to different choices of hermitian forms, it suffices to prove the case for $J_h=diag(1,1,-1)$. We can also assume the complex line $\tilde{L}$ has normal vector $n=[1,0,0]\in\PP V^+$ as the $\PP U(2,1)$-action is transitive on $\BB^2$. Consider $p=(z,w)=[z,w,1]\in \PP V^-$. Then we have
\[
	\tilde{\mu}= \frac{|z|^2}{1-|w|^2}.
\]
We first conclude a criterion for plurisubharmonic function when $\log \tilde{\mu}$ is the input. 

\begin{lemma}\label{lem:psh_crit}
For any given $r_* >0$, let $s_*:=\log(\tanh^2(r_*/2))$. Then for any $C^2$ function $f: [-\infty,s_*]\to\RR$, the function $f\circ (\log\tilde{\mu}) $ is plurisubharmonic on $\tilde{W}_{\tilde{L},r}\backslash \tilde{L}$ if $f''(s), f'(s) \ge 0$ for $s<s_*$. 
\end{lemma}
\begin{proof}
It suffices to check that $i\partial\bar{\partial}(f(\log \tilde{\mu}))$ is semipositive. We can check this by writing down the $(1, 1)$-form in coordinates $(z,w)$: we say that a $(1, 1)$-form $\omega$ on $\BB^2$ is represented by a matrix $M$ if
\[M=\left.\begin{pmatrix}\omega\left(\frac{\partial}{\partial z}, \frac{\partial}{\partial \bar{z}}\right) & \omega\left(\frac{\partial}{\partial z}, \frac{\partial}{\partial \bar{w}}\right) \\
\omega\left(\frac{\partial}{\partial w}, \frac{\partial}{\partial \bar{z}}\right) & \omega\left(\frac{\partial}{\partial w}, \frac{\partial}{\partial \bar{w}}\right) \end{pmatrix}\right|_{(z,w)} .\]
Then $i\partial\bar{\partial}\tilde{\mu}$ and $i\partial\tilde{\mu}\wedge\bar{\partial}\tilde{\mu}$ are respectively represented by:
\begin{equation}\label{formula:ddx_dxdx}
	\begin{pmatrix} \frac{1}{1-|w|^2} & \frac{\bar{z}w}{(1-|w|^2)^2} \\ \frac{z\bar{w}}{(1-|w|^2)^2}
 & \frac{|z|^2(1+|w|^2)}{(1-|w|^2)^3} \end{pmatrix}\text{\ \  and\ \  } \tilde{\mu}\cdot\begin{pmatrix} \frac{1}{1-|w|^2}  & \frac{\bar{z}w}{(1-|w|^2)^2} \\ \frac{z\bar{w}}{(1-|w|^2)^2} & \frac{|z|^2|w|^2}{(1-|w|^2)^3} \end{pmatrix}.
\end{equation}
Since 
\[
	i\partial\bar{\partial}(f\circ\log\tilde{\mu}) = \frac{f'(\log\tilde{\mu})}{\tilde{\mu}}\cdot i\partial\bar{\partial}\tilde{\mu} + \frac{f''(\log\tilde{\mu})-f'(\log\tilde{\mu})}{\tilde{\mu}^2}\cdot i\partial\tilde{\mu} \wedge \bar{\partial}\tilde{\mu},
\]
it is represented by 
\begin{align*}
&\frac{f'(\log\tilde{\mu})}{\tilde{\mu}}\cdot\begin{pmatrix} \frac{1}{1-|w|^2} & \frac{\bar{z}w}{(1-|w|^2)^2} \\ \frac{z\bar{w}}{(1-|w|^2)^2} & \frac{|z|^2(1+|w|^2)}{(1-|w|^2)^3} \end{pmatrix} + \frac{f''(\log\tilde{\mu})-f'(\log\tilde{\mu})}{\tilde{\mu}^2}\cdot\tilde{\mu}\cdot \begin{pmatrix} \frac{1}{1-|w|^2}  & \frac{\bar{z}w}{(1-|w|^2)^2} \\ \frac{z\bar{w}}{(1-|w|^2)^2} & \frac{|z|^2|w|^2}{(1-|w|^2)^3} \end{pmatrix} \\
=&f''(\log\tilde{\mu})\cdot \begin{pmatrix} \frac{1}{|z|^2}  & \frac{\bar{z}w}{|z|^2(1-|w|^2)} \\ \frac{z\bar{w}}{|z|^2(1-|w|^2)} & \frac{|w|^2}{(1-|w|^2)^2} \end{pmatrix} + f'(\log\tilde{\mu})\cdot \begin{pmatrix} 0 & 0 \\ 0 & \frac{1}{(1-|w|^2)^2} \end{pmatrix},
\end{align*}
and it is positive semidefinite when $f'(s),f''(s)\ge 0$ since the two matrices in the last row have zero determinants and positive traces. 
\end{proof}

In the following we will consider the function $F(s)=-2\log(1-e^s)$. It can be constructed following the idea by Hwang and To in \cite{hwang2012injectivity}. We first verify that $\tilde{\omega}_F:= i\partial\bar{\partial}(F\circ \log\tilde{\mu})$ is semipositive and bounded by the volume form in our setting. 
\begin{lemma}\label{lem:omega_F_bounded_by_vol}
	$\tilde{\omega}_F:= i\partial\bar{\partial}(F\circ \log\tilde{\mu})$ satisfies $0 \le \tilde{\omega}_F \le \omega_{\BB^2}$, where $\omega_{\BB^2}$ is the K{\"a}hler form given by the Bergman metric. 
\end{lemma}
\begin{proof}
	Since $s\in [-\infty, 0)$,
	\[
		F'(s)=\frac{2e^s}{1-e^s}, \ \ \ F''(s)=\frac{2e^s}{(1-e^s)^2},
	\]
	are nonnegative and by Lemma~\ref{lem:psh_crit}, $\tilde{\omega}_F$ is semipositive. The K\"ahler form $\omega_{\BB^2}$ is represented by
	\[
		\frac{2}{(1-|z|^2-|w|^2)^2} \begin{pmatrix} 1-|w|^2 & \bar{z}w \\ z\bar{w} & 1-|z|^2 \end{pmatrix}. 
	\]
	We plug $F$ into the computation in the proof of Lemma~\ref{lem:psh_crit} and we get that $\tilde{\omega}_F$ is represented by 
	\begin{align*}
		&\frac{2|z|^2(1-|w|^2)}{(1-|z|^2-|w|^2)^2} \cdot \begin{pmatrix} \frac{1}{|z|^2}  & \frac{\bar{z}w}{|z|^2(1-|w|^2)} \\ \frac{z\bar{w}}{|z|^2(1-|w|^2)} & \frac{|w|^2}{(1-|w|^2)^2} \end{pmatrix} + \frac{2|z|^2}{1-|z|^2-|w|^2}\cdot \begin{pmatrix} 0 & 0 \\ 0 & \frac{1}{(1-|w|^2)^2} \end{pmatrix} \\
		=&\frac{2}{(1-|z|^2-|w|^2)^2} \cdot \begin{pmatrix} 1-|w|^2  & \bar{z}w \\ z\bar{w} & \frac{|z|^2(1-|z|^2-|w|^2)}{(1-|w|^2)^2} \end{pmatrix}.
	\end{align*}
	Since $1-|z|^2-|w|^2 > 0$, $\frac{|z|^2(1-|z|^2-|w|^2)}{(1-|w|^2)^2} < \frac{1-|z|^2-|w|^2}{1-|w|^2} \le \frac{1-|z|^2-|w|^2 + |z|^2|w|^2}{1-|w|^2}=1-|z|^2$, and thus $\tilde{\omega}_F \le \omega_{\BB^2}$. 
\end{proof}

Similar to Lemma 21 in \cite{bakker2016p}, we have the following.
\begin{lemma}\label{lem:Ir_increasing}
	$\frac{1}{2\sinh^2(r/2)} \int_{\tilde{C}\cap \tilde{W}_{\tilde{L},r}/\Gamma_{\tilde{L}}} \tilde{\omega}_F$ is an increasing function of $r$. 
\end{lemma}
\begin{proof}
	By Stokes' theorem,  
	\begin{align*}
		\int_{\tilde{C}\cap \tilde{W}_{\tilde{L},r}/\Gamma_{\tilde{L}}} \tilde{\omega}_F &= F'(\log \tanh^2(r/2)) \int_{\tilde{C}\cap \tilde{W}_{\tilde{L},r}/\Gamma_{\tilde{L}}} i\partial\bar{\partial}\log\tilde{\mu} \\
		&= 2\sinh^2(r/2) \int_{\tilde{C}\cap \tilde{W}_{\tilde{L},r}/\Gamma_{\tilde{L}}} i\partial\bar{\partial}\log\tilde{\mu} 
	\end{align*}
	because we can approximate $F$ by a linear function of slope $F'(\log \tanh^2(r/2))$ at the boundary of $\tilde{W}_{\tilde{L},r}$ without changing the integral of the interior. Moreover, applying Lemma~\ref{lem:psh_crit} to $f(s)=s$ shows that $\log\tilde{\mu}$ is plurisubharmonic, and then the claim follows. 
\end{proof}

Now we can conclude the proof as in \cite[Theorem 19(b)]{bakker2016p}. 

\begin{proof}[Proof of Proposition~\ref{prop:tubular_mult_bound}]
	By Lemma~\ref{lem:omega_F_bounded_by_vol} and Lemma~\ref{lem:Ir_increasing}, we have
	\[
		\int_{\tilde{C}\cap \tilde{W}_{\tilde{L},r}/\Gamma_{\tilde{L}}} \omega_{\BB^2} \ge \int_{\tilde{C}\cap \tilde{W}_{\tilde{L},r}/\Gamma_{\tilde{L}}} \tilde{\omega}_F \ge 2\sinh^2(r/2) \cdot \lim_{r\to 0}\int_{\tilde{C}\cap \tilde{W}_{\tilde{L},r}/\Gamma_{\tilde{L}}} i\partial\bar{\partial}\log\tilde{\mu} .
	\]
	We can bound the limit by a local calculation of Lelong numbers the same as in \cite{hwang2012injectivity} with minor modifications: for points in a small polydisc neighborhood $U_j$ of each intersection point $(0,w_j)\in \tilde{C}\cap \tilde{L}$, $\tilde{\mu} = \frac{|z|^2}{1-|w|^2} \le \lambda |z|^2$ for some fixed $\lambda > 1$ as there are only finitely many $U_j$, and this constant term $\lambda$ will have no impact when we take the limit to compute the Lelong number. The same as in \cite{hwang2012injectivity}, we can conclude that 
	\[
		\lim_{r\to 0}\int_{\tilde{C}\cap \tilde{W}_{\tilde{L},r}/\Gamma_{\tilde{L}}} i\partial\bar{\partial}\log\tilde{\mu} \ge 2\pi\cdot (\tilde{C}\cdot \tilde{L}). 
	\]
\end{proof}

Now we continue to prove Theorem~\ref{thm:tubular_mult_bound}, i.e. descend Proposition~\ref{prop:tubular_mult_bound} to the quotient space $X=\Gamma\backslash\BB^2$ and consider a finite set of elliptic complex lines $\{L_j\subset X\}_j$. For each $L_j$, pick a preimage $\tilde{L}_j\subset\BB^2$, so all the preimages of $L_j$ are given by $\Gamma\cdot \tilde{L}_j$.

To derive a result on $\cup_j W_{L_j,r}$, we consider the covering given by the disjoint union of tubular neighborhoods of $\{\tilde{L}_j\}_j$, and we will show that the sizes of the fibers are bounded.  

\begin{lemma}\label{lem:stabilizer_invariant_descends}
	Follow the notation in (\ref{notation_twoball}). For any given $r>0$, there exists $D_0(r)>0$ such that the local degree of the quotient map
	\[
	 	\tilde{\pi}: \sqcup_j \tilde{W}_{\tilde{L}_j, r}/\Gamma_{\tilde{L}_j}\to \cup_j W_{L_j,r}
	 \] 
	is at most $\max_{\tilde{x}\in \cup_j \tilde{L}_j} |\Gamma_{\tilde{x}}|$ for any $D>D_0(r)$. 
\end{lemma}
\begin{proof}
  We first apply Proposition~\ref{prop:repulsion_lines} to pick a $D_1$ such that when $D>D_1$, distances between disjoint elliptic complex lines are greater than $2r$. This guarantees that for any distinct $\tilde{L}_a,\tilde{L}_b\in \{\tilde{L}_j\}_j$, $\tilde{W}_{\tilde{L}_a, r} \cap \tilde{W}_{\tilde{L}_b, r}=\emptyset$ unless $\tilde{L}_a\cap\tilde{L}_b \ne \emptyset$. 

  Next, consider $\tilde{x}\in \tilde{W}_{\tilde{L}_a, r} \cap \tilde{W}_{\tilde{L}_b, r}$ for some distinct $\tilde{L}_a,\tilde{L}_b\in \{\tilde{L}_j\}_j$ with stabilizer $M_a,M_b$ respectively. Let $\tilde{p}:= \tilde{L}_a\cap\tilde{L}_b$. Then Lemma~\ref{lem:intersection_of_elliptic_lines_can_have_one_dim_eigenspace} says $\FixLocus(M_aM_b)=\tilde{p}$. Note that $\rho_{\BB^2}(\tilde{x}, M_aM_b\tilde{x})=\rho_{\BB^2}(M_a^{-1}\tilde{x}, M_b\tilde{x}) \le \rho_{\BB^2}(M_a^{-1}\tilde{x}, \tilde{x}) + \rho_{\BB^2}(\tilde{x}, M_b\tilde{x}) < 2r + 2r = 4r$. Then by Proposition~\ref{prop:inj_radius_ell}, $\rho_{\BB^2}(\tilde{x}, \tilde{p}) < \frac{4r}{\sqrt{2(1-\cos(2\pi/N))}}$ where $N$ is the order of $M_aM_b$. 

  Therefore, if we apply Proposition~\ref{prop:repulsion} to pick a $D_0\ge D_1$ such that the distances between distinct isolated elliptic points are greater than $2\cdot\frac{4r}{\sqrt{2(1-\cos(2\pi/N_0))}}$, where $N_0$ is the maximal order of elliptic elements in $\Gamma$, then $\tilde{x}$ can only lie in tubular neighborhoods of a subset of $\{\tilde{L}_j\}_j$ that intersect at a single point. Consequently, the degree of $\tilde{\pi}$ is bounded by $\max_x |\Gamma_x|$. 
\end{proof}

\begin{proof}[Proof of Theorem~\ref{thm:tubular_mult_bound}]
	Let $\omega_X$ denote the volume form on $X$. By Proposition~\ref{prop:tubular_mult_bound} and Lemma~\ref{lem:stabilizer_invariant_descends}, we have
	\begin{align*}
		\vol(C\cap \cup_jW_{L_j,r}) &= \int_{\cup_jW_{L_j,r}} [C]\wedge \omega_X \\
		 &\ge \frac{1}{m} \sum_j \int_{\tilde{W}_{\tilde{L}_j,r}/\Gamma_{\tilde{L}_j}} [\tilde{C}]\wedge \omega_{\BB^2} \\
		 &\ge \frac{4\pi}{m} \sinh^2\left(\frac{r}{2}\right)\sum_j(C\cdot L_j) .
	\end{align*}
\end{proof}

Similar to \cite[Proposition 23]{bakker2016p}, we can derive the following relative version. 

\begin{proposition}\label{prop:tubular_relative_vol}
	Follow the notation in (\ref{notation_twoball}). Let $m$ be the maximum of $|\Gamma_x|$. Then for any given $0<r<R$, there exists $D_0(R)>0$ such that when $D>D_0(R)$, 
	\[
		\vol(C\cap W_{L,R}) \ge \frac{\cosh^2(R/2)}{m\cdot \cosh^2(r/2)} \vol(C\cap W_{L,r}).
	\]
  More generally, we can replace the tubular neighborhood of a single elliptic complex line by a union of elliptic complex lines. 
\end{proposition}
\begin{proof}
	We will first show that $\displaystyle \frac{1}{\cosh^2(r/2)}\int_{C\cap \tilde{W}_{\tilde{L},r}/\Gamma_{\tilde{L}}} \omega_{\BB^2}$ is a nondecreasing function of $r$, where $L$ is given by $z=0$. Consider 
  \[
    f(t) = \log\left(\frac{e^t}{1-e^t}\right) \ \ \text{ and }\ \  t=\log \mu =\log(\tanh^2(r/2)).
  \]
  By Lelong-Poincar{\'e} we have:
	\begin{align*}
		i\partial\bar{\partial}(f\circ\tilde{\mu}) &=i\partial\bar{\partial}\log|z|^2 - i\partial\bar{\partial}\log(1-|z|^2-|w|^2) \\
		&= \pi [\tilde{L}] +\frac{1}{2} \omega_{\BB^2}.
	\end{align*}
	Then by Stokes' theorem, we have:
	\begin{align*}
		\frac{1}{2} \int_{C\cap \tilde{W}_{\tilde{L},r}/\Gamma_{\tilde{L}}} \omega_{\BB^2} + \pi\cdot (\tilde{L}\cdot \tilde{C}) &= \int_{C\cap \tilde{W}_{\tilde{L},r}/\Gamma_{\tilde{L}}} i\partial\bar{\partial}(f\circ\tilde{\mu})\\
		&= f'(\log(\tanh^2(r/2))) \int_{C\cap \tilde{W}_{\tilde{L},r}/\Gamma_{\tilde{L}}} i\partial\bar{\partial}\log\tilde{\mu}.
	\end{align*}
	By Lemma~\ref{lem:psh_crit}, $\log\tilde{\mu}$ is plurisubharmonic, so $\int_{C\cap W_{L,r}} i\partial\bar{\partial}\log\tilde{\mu}$ is a nondecreasing function of $r$. $f'(\log(\tanh^2(r/2)))=\cosh^2(r/2)$ is an increasing function of $r$, so $\displaystyle \frac{1}{\cosh^2(r/2)}\int_{C\cap W_{L,R}} \omega_{\BB^2}$ is a nondecreasing function of $r$ and we obtain an inequality by plugging in $0<r<R$ and taking the overlaps into account as in the proof of Theorem~\ref{thm:tubular_mult_bound}. The argument stays the same if we replace the tubular neighborhood of a single elliptic complex line by a union of elliptic complex lines as in the proof of Theorem~\ref{thm:tubular_mult_bound}. 
\end{proof}

Similarly, we have the relative version for volume estimates for a point.

\begin{proposition}\label{prop:point_relative_vol}
	Follow the notation in (\ref{notation_twoball}). Let $m$ be the maximum of $|\Gamma_x|$. Then for any given $0<r<R$ and any isolated elliptic point $x\in X$, there exists $D_0(R)>0$ such that when $D>D_0(R)$, 
	\[
		\vol(C\cap B_{x,R}) \ge \frac{\cosh^2(R/2)}{m\cdot \cosh^2(r/2)} \vol(C\cap B_{x,r}).
	\]
\end{proposition}
\begin{proof}
	It suffices to show that $\displaystyle \frac{1}{\cosh^2(r/2)}\int_{C\cap B_{x,R}} \omega_{\BB^2}$ is a nondecreasing function of $r$ for the case where $x$ is $[0,0,1]\in\BB^2$. Now for any $p=[z,w,1]\in\BB^2$, we have $\cosh^2(\rho_{\BB^2}(x,p)/2)=\ta{x,p}=1/(1-|z|^2-|w|^2)$. Now consider
  \[
    g(t)=-\log(1-t), \ \  t=|z|^2+|w|^2\ \ \implies g'(t) = \frac{1}{1-|z|^2-|w|^2}.
  \] 
  Then by Stokes' theorem:
	\begin{align*}
		\int_{C\cap B_{x,r}} \omega_{\BB^2} &= \int_{C\cap B_{x,r}} 2i\partial\bar{\partial}(-\log(1-|z|^2-|w|^2)) \\
    &= g'(|z|^2+|w|^2) \int_{C\cap B_{x,r}} 2i\partial\bar{\partial}(|z|^2+|w|^2) \\
		&= \cosh^2(r/2) \int_{C\cap B_{x,r}} 2i\partial\bar{\partial}(|z|^2+|w|^2)
	\end{align*}
	Note that $p=[z,w,1]\mapsto |z|^2+|w|^2$ is plurisubharmonic, so we finish the proof. 
\end{proof}

\subsection{Multiplicity bounds in \ref{notation_hms_smallsubfield}}\label{sub:multiplicity_bounds_in_bb_1_times_bb_1_}

Similar to Theorem~\ref{thm:point_mult_bound}, we have the following as a corollory of \cite[Theorem 1]{hwang2002volumes}. Note that since we do not have pointwise repulsion of elliptic points in this case, compared to Theorem~\ref{thm:point_mult_bound} we have an assumption that the injectivity radius of the elliptic point is large here.

\begin{theorem}\label{thm:point_mult_bound_hms}
	Follow the notation in (\ref{notation_hms_smallsubfield}). Then for any complex analytic curve $C\subset X$ and an elliptic point $x\in C$ with stabilizer group $\Gamma_x$, let $B_{x,r}$ be the geodesic ball centered at $x$ with radius $r$. Suppose that the injectivity radius of $x$ is greater than $r$. Then there exists $D_0(r)>0$ such that when $D>D_0(r)$, 
	\[
		\vol(C\cap B_{x,r}) \ge \frac{8\pi}{|\Gamma_x|} \sinh^2\left(\frac{r}{2}\right)\mult_{x}(C).
	\]
\end{theorem}
\begin{proof}
	Consider any preimage $\tilde{x}\in \BB^1\times\BB^1$ of $x$ with stabilizer group $\Gamma_{\tilde{x}}$. For any $\tilde{z}\in \tilde{B}_{\tilde{x},r}$ and $M\in \Gamma\backslash \Gamma_{\tilde{x}}$, $M\cdot \tilde{z}\in \tilde{B}_{M\cdot\tilde{x},r}$. Since the injectivity radius of $x$ is greater than $r$, $\rho_{\BB^1\times\BB^1}(\tilde{x},M\cdot \tilde{z}) \ge |\rho_{\BB^1\times\BB^1}(\tilde{x},M\cdot \tilde{x}) - \rho_{\BB^1\times\BB^1}(M\cdot\tilde{x},M\cdot \tilde{z})| > 2r -r =r$. Thus, $M\cdot \tilde{z}\in \tilde{B}_{\tilde{x},r}$ if and only if $M\in \Gamma_{\tilde{x}}$, so $z\in B_{x,r}$ at most has $|\Gamma_x/\{\pm 1\}|=|\Gamma_x|/2$ preimages in $\tilde{B}_{\tilde{x},r}$. 
\end{proof}

Since for different choices of hermitian forms we have isometric hyperbolic 1-balls, the constructions on $\BB^1\times\BB^1$ in \cite{hwang2012injectivity} and \cite{bakker2016p} also work in our case. For a special curve $F_B$, we consider the tubular neighborhood 
\[
	\tilde{W}_{B,r}:= \left\{p\in\BB^1\times\BB^1: \rho_{\BB^1\times\BB^1}(p,F_B)<r  \right\}
  = \left\{(z,w)\in\BB^1\times\BB^1: \rho_{\BB^1}(Bz,w)<2r  \right\},
\]
and $W_{B,r}$ its projection in the quotient. \label{page:curve_neighborhood}

\begin{lemma}\label{lem:stabilizer_invariant_descends_F_B}
	Follow the notation in (\ref{notation_hms_smallsubfield}). For any $r>0$, there is a $D_0>0$ such that when $D>D_0$, the quotient map $p:\tilde{W}_{B,r}/ E_B \to W_{B,r}=\tilde{W}_{B,r}/ \Gamma$ is a biholomorphism, and thus any $E_B$-invariant function (resp. form) on $\tilde{W}_{B,r}$ descends to a function (resp. form) on $W_{B,r}$. 
\end{lemma}
\begin{proof}
	It suffices to check that the quotient map is injective, i.e. for any $z\in\tilde{W}_{B,r} $ and $M\in \Gamma$, if $M\cdot z\in \tilde{W}_{B,r}$, then $M\in E_B$. If not, then by Proposition~\ref{prop:repulsion_curves_hms}, we can find $D_0>0$ such that when $D>D_0$, $\rho_{\BB^1\times\BB^1}(F_B,M\cdot F_B)> 2r$. Then $\rho_{\BB^1\times\BB^1}(M\cdot z, F_B) \ge |\rho_{\BB^1\times\BB^1}(F_B,M\cdot F_B) - \rho_{\BB^1\times\BB^1}(M\cdot z,M\cdot F_B)| > 2r-r=r$, contradictory to $M\cdot z\in \tilde{W}_{B,r}$. Thus, $M\in E_B$. 
\end{proof}

As a consequence, for the diagonal $F_1$, we can apply the constructions on $\BB^1\times\BB^1$ in \cite{hwang2012injectivity} and \cite{bakker2016p} and derive the same multiplicity bounds. We summarize the construction in \cite{hwang2012injectivity,bakker2016p} and compare it with our construction for $\BB^2$ in the table below. 

\begin{center}
\begin{tabular}{|l | c|c|}
\hline
Hyperbolic Space & $\BB^2=\{(z,w): |z|^2+|w|^2<1\}$ & $\BB^1\times\BB^1=\{(z,w): |z|,|w|<1\}$  \\ 
Volume Form $\omega_{hyp}$  &  $-2i\partial\bar{\partial}\log(1-|z|^2-|w|^2)$ & $-2i\partial\bar{\partial}\log((1-|z|^2)(1-|w|^2))$ \\ 
Special Curve & $z=0$ & $z=w$ \\ 
Distance Measure $\mu$ & $\mu = \tanh^2\left(\frac{\rho_{\BB^2}(p,L)}{2}\right)= \frac{|z|^2}{1-|w|^2}$ & $\mu = \tanh^2\left(\frac{\rho_{\BB^1}(z,w)}{4}\right)$ \\ 
$F(\log\mu)$ psh when... & $F', F'' \ge 0$ & $F', F'' \ge 0$ \\ 
Nice psh $F(\log\mu)$ & $F(s)=-2\log(1-e^s)$ & $F(s)=-4\log(1-e^s)$ \\ 
 \hline
\end{tabular}
\end{center}

\begin{theorem}\label{thm:special_curve_mult_bound_hms}
 	Consider a special curve $F_B$. For any $R>r>0$, there exists $D_0(R)$ such that when $D>D_0(R)$, for any complex analytic curve $C\subset \Gamma\backslash (\BB^1\times\BB^1)$:
 	\begin{enumerate}
 		\item $\vol(C\cap W_{B,R}) \ge 8\pi \sinh^2\left(\frac{R}{2}\right) (C\cdot F_B)$;
 		\item $\displaystyle \vol(C\cap W_{B,R}) \ge \frac{\cosh(R)}{\cosh(r)} \vol(C\cap W_{B,r})$.
 	\end{enumerate}
\end{theorem} 
\begin{proof}
	By Lemma~\ref{lem:stabilizer_invariant_descends_F_B}, we can work on constructions in the tubular neighborhood in the universal cover $\BB^1\times\BB^1$ when the discriminant $D$ is large, as long as the constructions are $E_B$-invariant. When $B=1$, i.e. when $F_B$ is the diagonal, the volume inequalities above have been proved in \cite{hwang2012injectivity} and \cite{bakker2016p}. When $B\ne 1$, the same psh function construction will work if we replace $\rho_{\BB_1}(z,w)$ in the diagonal case by $\rho_{\BB_1}(Bz,w)$ for $(z,w)\in F_B$. Another way to think of this is to consider the isometry $\phi: \BB^1\times\BB^1\to\BB^1\times\BB^1, \ \ (z,w)\mapsto (B^*z,w)$. Then $\phi^*(F_B)=F_1$ and $\phi^*(\tilde{W}_{B,r})=\tilde{W}_{1,r}$. Thus, 
	\begin{align*}
		\vol(C\cap W_{B,R}) &= \vol(\phi^*C\cap W_{1,R}) \\ 
		&\ge 8\pi \sinh^2\left(\frac{R}{2}\right) (\phi^*C\cdot F_1) \\ 
		&= 8\pi \sinh^2\left(\frac{R}{2}\right) (C\cdot F_B).
	\end{align*}
	Similarly, $\frac{1}{\cosh(r)} \vol(C\cap W_{B,r})=\frac{1}{\cosh(r)} \vol(\phi^*C\cap W_{1,r})$ is an increasing function of $r$, so we obtain the second inequality. 
\end{proof}

We also have a relative version for volume estimates for a point similar to Proposition~\ref{prop:point_relative_vol}. 

\begin{proposition}\label{prop:point_relative_vol_hms}
  Follow the notation in (\ref{notation_hms_smallsubfield}). Then for any given $0<r<R$ and any point $x\in X$ with injectivity radius at least $2R$, there exists $D_0(R)>0$ such that when $D>D_0(R)$, 
  \[
    \vol(C\cap B_{x,R}) \ge \frac{\cosh^2(R/2)}{2|\Gamma_x|\cosh^2(r/2)} \vol(C\cap B_{x,r}).
  \]
\end{proposition}
\begin{proof}
  After an isometry, we can assume that $x=(0,0)$, so the region in the universal cover becomes
  \[
    \tilde{B}_{0,r} = \{(z_1,z_2)\in \BB^1\times\BB^1: \max(|z_1|,|z_2|) < T \}.
  \] 
  for $T:=\tanh(r/2)$. The idea is similar to \cite[Lemma 22]{bakker2016p}, i.e. we divide $B_{0,r}$ into 
  \[
    A_j = \{(z_1,z_2)\in \BB^1\times\BB^1: |z_{3-j}| < |z_j| < T \}
  \]
  for $j=1,2$. Now consider
  \[
    g(t)=-\log(1-t), \ \  t=|z|^2\ \ \implies g'(t) = \frac{1}{1-|z|^2}, \ \ g'(T^2) = \cosh^2(r/2).
  \]
  By Stokes' theorem,
  \begin{align*}
    \frac{1}{2\pi}\int_{C\cap A_j} dd^c(-\log(1-|z_j|^2)) 
    &= \int_{C\cap \{|z_{3-j}| < |z_j| = T\}} -d^c\log (1-|z_j|^2) + \int_{C\cap \{|z_{3-j}| = |z_j| < T\}} d^c\log (1-|z_j|^2) \\
    &= g'(T^2)\int_{C\cap \{|z_{3-j}| < |z_j| = T\}} d^c |z_j|^2 + \int_{C\cap \{|z_{3-j}| = |z_j| < T\}} g'(|z_j|^2) d^c |z_j|^2 .\\
  \end{align*}
  where the orientation of $\{|z_{3-j}| = |z_j| < T\}$ is from $A_j$ to $A_{3-j}$. Also, 
  \[
    \int_{C\cap A_j} dd^c |z_j|^2  = \int_{C\cap \{|z_{3-j}| < |z_j| = T\}} d^c |z_j|^2  + \int_{C\cap \{|z_{3-j}| = |z_j| < T\}} d^c |z_j|^2.
  \]
  Combining the equations for $j=1,2$, we obtain
  \begin{align*}
    &\frac{1}{2\pi\cosh^2(r/2)}\int_{C\cap \tilde{B}_{0,r}} dd^c\max_{j=1,2}(-\log(1-|z_j|^2)) \\
    =& \int_{C\cap \tilde{B}_{0,r}}dd^c \max_{j=1,2}|z_j|^2 + \int_{C\cap \{|z_2| = |z_1| < T\}} \left[1-\frac{g'(|z_1|^2)}{g'(T^2))}\right] d^c (|z_2|^2-|z_1|^2)
  \end{align*}
  where the orientation of the latter integral on $|z_1| = |z_2|$ is the outward direction for $A_1$. Since its integrand is positive, taking into account its orientation, this integral is positive and a nondecreasing function of $r$. Moreover, the first term is also a nondecreasing function of $r$ since $\max_{j=1,2}|z_j|^2$ is plurisubharmonic. In conclusion, 
  \[
    \frac{1}{2\pi\cosh^2(r/2)}\int_{C\cap \tilde{B}_{0,r}} dd^c\max_{j=1,2}(-\log(1-|z_j|^2)) 
  \]
  is a nondecreasing function of $r$. Then, Note that we have inequalities of plurisubharmonic functions:
  \[
    \max_{k=1,2}(\log (1-|z_k|^2)) \le \sum_{k=1,2}\log (1-|z_k|^2) \le 2\max_{k=1,2}(\log (1-|z_k|^2)).
  \]
  Using standard integration by parts following Chern-Levine-Nirenberg estimates, this gives inequalities of integrals over a curve for the corresponding semi-positive currents, and then for $0<r<R$ we obtain
  \begin{align*}
    \vol(C\cap B_{0,r}) 
    &\le \int_{C\cap \tilde{B}_{0,r}} 2dd^c\max_{j=1,2}(-\log(1-|z_j|^2)) \\
    &\le \frac{2\cosh^2(r/2)}{\cosh^2(R/2)} \int_{C\cap \tilde{B}_{0,R}} dd^c\max_{j=1,2}(-\log(1-|z_j|^2)) \\
    &\le \frac{2|\Gamma_x|\cosh^2(r/2)}{\cosh^2(R/2)}\vol(C\cap B_{0,R}).
  \end{align*}
\end{proof}

We also need a relative volume inequality for axes of loxodromic elements. The strategy is the same as Proposition~\ref{prop:point_relative_vol_hms}, so it is left to find plurisubharmonic functions we need. After an isometry, it suffices to consider the case where $\BB^1$ is the unit ball and the geodesic is the real axis, i.e. 
  \[
    J=\begin{pmatrix} 1 & 0 \\ 0 & -1 \end{pmatrix}, 
    v_1 = \begin{pmatrix} 1 \\ 1 \end{pmatrix},  
    v_2 = \begin{pmatrix} -1 \\ 1 \end{pmatrix},
    u = \begin{pmatrix} z \\ 1 \end{pmatrix}.
  \]
Then, 
\begin{align*}
  \eta = \frac{\langle v_1, u \rangle \langle u, v_2 \rangle}{\langle v_1, v_2 \rangle \langle u, u \rangle} = \frac{(\bar{z}-1)(-z-1)}{(-1-1)(z\bar{z}-1)} = \frac{1}{2} + \frac{z - \bar{z}}{2(1-z\bar{z})},\\
  \sinh^2(\rho) = (2|\eta|)^2 - 1 = \frac{-(z-\bar{z})^2}{(1-z\bar{z})^2}=\left(\frac{\Im(z)}{1-|z|^2}\right)^2,
\end{align*}
and now we use 
\[
  \mu := \tanh^2(\rho) = \frac{\sinh^2(\rho)}{\sinh^2(\rho)+1}=\frac{-(z-\bar{z})^2}{(1-z^2)(1-\bar{z}^2)}=\left(\frac{\Im(z)}{|1-z^2|}\right)^2,
\]
to measure the distance from $u$ to the geodesic $\widehat{v_1v_2}$. 

\begin{lemma}
  Follow the notation above. Then $f(\log \mu)$ is psh if $f\in C^2((-\infty,0))$ satisfies
  \begin{equation}\label{eq:psh_axis_lox}
    -f'+2(1-e^s)f''\ge 0.
  \end{equation}
\end{lemma}
\begin{proof}
  Similar to the computation in the proof of Lemma~\ref{lem:psh_crit}, we compute:
  \begin{align*}
    \left( \frac{\partial \mu}{\partial z}, \frac{\partial \mu}{\partial \bar{z}} \right) &= \left( \frac{-2(z-\bar{z})(1-z\bar{z})}{(1-z^2)^2(1-\bar{z}^2)}, \frac{2(z-\bar{z})(1-z\bar{z})}{(1-z^2)(1-\bar{z}^2)^2} \right), \\ 
      \frac{\partial^2 \mu}{\partial z \partial \bar{z}} &= \frac{2(1+z^2+\bar{z}^2-4z\bar{z}+z^2\bar{z}^2)}{(1-z^2)^2(1-\bar{z}^2)^2} = \frac{-4}{(1-z^2)(1-\bar{z}^2)} \cdot \mu + \frac{2}{(1-z^2)(1-\bar{z}^2)}, \\
      i\partial\bar{\partial}(f \circ \log \mu) &= \frac{f'}{\mu} \cdot i\partial\bar{\partial}\mu + \frac{f''-f'}{\mu^2} \cdot i\partial\mu\wedge \bar{\partial}\mu \\
      &= \frac{1}{(1-z^2)(1-\bar{z}^2)\mu}  \left[ (2-4\mu)f'+(-4\mu+4)(f''-f') \right] \cdot i\partial z \wedge\partial \bar{z} \\
      &= \frac{2}{-(z-\bar{z})^2}  \left[ -f'+2(1-\mu)f'' \right] \cdot i\partial z \wedge\partial \bar{z}.
  \end{align*}
  Therefore, $i\partial \bar{\partial}f(\log \mu)$ is semi-positive if $-f'+2(1-\mu)f''\ge 0$. 
\end{proof}

\begin{proposition}\label{prop:axis_lox_relative_vol_hms}
  Follow the notation in (\ref{notation_hms_smallsubfield}). For a loxodromic element $M\in \Gamma$, let $\Gamma_M$ be the stabilizer subgroup of $\axislox_M$. Then for any given $0<r<R$, consider the neighborhood of $\axislox_M$ of radius $R$:
  \[
    \tilde{W}_{M,R}:= \{x\in \BB^1\times\BB^1: \rho(x,\axislox_M)<R \}
  \]
  and its image $W_{M,R}$ in $\Gamma\backslash\BB^1\times\BB^1$. Then there exists $D_0(R)>0$ such that when $D>D_0(R)$, for any curve $C$, if 
  \[
    m=\max_{x\in C}\#\{T\cdot \axislox_M: \rho(x,T\cdot \axislox_{M}) < R \}
  \]
  is the maximal sheets of neighborhoods of conjugates of $\axislox_M$ covering a point on $C$, then
  \[
    \vol(C\cap W_{M,R}) \ge \frac{\sinh^2(R/2)}{2m\cdot\sinh^2(r/2)} \vol(C\cap W_{M,r}).
  \]
\end{proposition}
\begin{proof}
  Consider
  \[
    g = -\log(1-e^s), \ \ t = \arctan(\sqrt{e^{-s}-1})\implies g(t)=-2\log(\sin t).
  \]
  $g$ and $t$ as functions of $s$ satisfy (\ref{eq:psh_axis_lox}). In fact, $t$ is a solution of the equation $-f'+2(1-e^s)f'' = 0$, and $g$ corresponds to the hyperbolic volume form as:
  \[
    \frac{(1-z\bar{z})^2}{(1-z^2)(1-\bar{z}^2)}=1-\mu \implies \omega_{\BB^1} = dd^c(-2\log(1-|z|^2))=dd^c(-\log(1-\mu)).
  \]
  Note that $dd^c\log(|1-z^2|^2)=0$ as $1-z^2$ is nonvanishing and holomorphic on the unit disc. 

  Then we apply the same argument as Proposition~\ref{prop:point_relative_vol_hms}. For the coordinates $(z_1,z_2)\in\BB^1\times\BB^1$, we have the corresponding $\mu_j=\tanh^2(\rho_{\BB^1}(z_j,\widehat{v_1v_2}))$ and $t_j = \arctan(\sqrt{\mu_j^{-1}-1})$, so the tubular neighborhood of the product of geodesics $\widehat{v_1v_2}\times\widehat{v_1v_2}$ of radius $r$ is
  \[
    \tilde{W}_r:=\{(z_1,z_2)\in\BB^1\times\BB^1: \max_j(\mu_j)<\tanh^2(r)  \}.
  \]
  For $j=1,2$, we divide $\tilde{W}_r$ into 
  \[
    A_j = \{(z_1,z_2)\in \BB^1\times\BB^1: \mu_{3-j} < \mu_j < \tanh^2(r) \}.
  \]
  Consider the $g(t)$ above, and then by Stokes' theorem, 
  \begin{align*}
    &\frac{1}{2\pi}\int_{C\cap A_j/\Gamma_M} dd^c(-2\log(1-|z_j|^2)) 
    = \frac{1}{2\pi}\int_{C\cap A_j/\Gamma_M} dd^c(-\log(1-\mu_j)) \\
    =& g'(\arctan(\csch(r)))\int_{C\cap \{\mu_{3-j} < \mu_j = \tanh^2(r)\}/\Gamma_M} d^c t_j
     + \int_{C\cap \{\mu_{3-j} = \mu_j < \tanh^2(r)\}/\Gamma_M} g'(t_j) d^c t_j .\\
  \end{align*}
  where the orientation of $\{\mu_{3-j} = \mu_j < \tanh^2(r)\}$ is from $A_j$ to $A_{3-j}$. Also, 
  \[
    \int_{C\cap A_j/\Gamma_M} dd^c t_j  = \int_{C\cap \{\mu_{3-j} < \mu_j = \tanh^2(r)\}/\Gamma_M} d^c t_j  + \int_{C\cap \{\mu_{3-j} = \mu_j < \tanh^2(r)\}/\Gamma_M} d^c t_j.
  \]
  Combining the equations for $j=1,2$, we obtain
  \begin{align*}
    &\frac{1}{2\pi\sinh^2(r)}\int_{C\cap \tilde{W}_r/\Gamma_M} dd^c\max_{j=1,2}(-\log(1-|z_j|^2)) \\
    =& \int_{C\cap \tilde{W}_r/\Gamma_M}dd^c \max_{j=1,2}t_j + \int_{C\cap \{\mu_2 = \mu_1 < \tanh^2(r)\}/\Gamma_M} \left[1-\frac{g'(t_1)}{g'(\arctan(\csch(r)))}\right] d^c (t_2-t_1)
  \end{align*}
  where the orientation of the latter integral on $\mu_2 = \mu_1$ is the outward direction for $A_1$. Since its integrand is positive, taking into account its orientation, this integral is positive and a nondecreasing function of $r$. Moreover, the first term is also a nondecreasing function of $r$ since $\max_{j=1,2}t_j$ is plurisubharmonic. In conclusion, using standard integration by parts following Chern-Levine-Nirenberg estimates, for $0<r<R$ we restrict to $W_{M,R}$ and obtain
  \begin{align*}
    \vol(C\cap \tilde{W}_{r}/\Gamma) 
    &\le \int_{C\cap \tilde{W}_r/\Gamma_M} 2dd^c\max_{j=1,2}(-\log(1-|z_j|^2)) \\
    &\le \frac{2\sinh^2(r)}{\sinh^2(R)} \int_{C\cap \tilde{W}_R/\Gamma_M} dd^c\max_{j=1,2}(-\log(1-|z_j|^2)) \\
    &\le \frac{2m\sinh^2(r)}{\sinh^2(R)}\vol(C\cap \tilde{W}_R/\Gamma).
  \end{align*}

\end{proof}

\section{Proof of main theorems}\label{sec:proof_of_main_theorems}

The proofs of Theorem~\ref{thm:2_ball_geometry_ver} and \ref{thm:hms_geometry_ver} both use the method of volume estimates: by picking a lower bound $D_0$ of the discriminant $D$, we will derive contradictions by bounding the volume of a curve. The strategy includes the following steps:
\begin{step}
  \textit{Develop an upper bound of $\vol_X(C)$ for nonspecial curve $C$.} This is done by combining Schwarz's lemma and an upper bound of elliptic points in terms of $\vol_X(C)$.
\end{step}
\begin{step}
  \textit{Develop a lower bound of $\vol_X(C)$ for nonspecial curve $C$ contradicting the previous upper bound.} This is done by showing the existence of a point on $C$ with large injectivity radius for large $D$ using relative versions of volume inequalities.
\end{step}
\begin{step}(optional)
  \textit{Deal with the case of special curves.}
\end{step}
Now we will see how this works in each case.

\subsection{Proof of Theorem~\ref{thm:2_ball_geometry_ver}}\label{sub:proof_of_theorem_ref_thm_2_ball_geometry_ver}

\setcounter{step}{0}

Follow the notation in \ref{notation_twoball}. Consider any irreducible curve $C\subset X$ of fixed genus $g$. Note that by Lemma~\ref{lem:ellclass}, for fixed $\ell$ there is a uniform bound $N_0$ on the order of stabilizer subgroups of any point. 

The idea of the proof is to derive two contradictory inequalities when the discriminant $D$ is large. We will first show the case where $C$ does not have generic stacky-ness, i.e. it is not a component of an elliptic complex line. The other case will follow using the same strategy. 

\begin{step}
	\textit{An upper bound of $\vol_X(C)$ when $C$ is not a component of an elliptic complex line.}

	Let $U\subset C$ be the open subset of $C$ with all elliptic points removed. Since $X$ is hyperbolic, $U$ can be uniformized by $\HH$. Apply Schwarz's lemma to the hyperbolic metric on $X$ and the uniformized metric on $\HH$ of constant Gaussian curvature $-1$, and then by Gauss-Bonnet we have
	\begin{equation}\label{eq:schwarzlem}
		\frac{1}{4\pi}\vol_X(C) \le -\chi_\text{orb}(C) \le 2g-2+\#\{\text{elliptic points}\}
	\end{equation}
	where $\chi_\text{orb}(C)$ is the orbifold Euler characteristic. This inequality will become an upper bound of $\vol_X(C)$ if we can show that $\#\{\text{elliptic points}\} < c\cdot \vol_X(C)$ for some small constant $c < \frac{1}{4\pi} $. Now we will first show a bound of $\#\{\text{elliptic points}\}$ for a fixed number $R>0$, and later we can see how to pick $R$ in terms of the genus $g$.

  By Lemma~\ref{lem:ellclass}, there is a maximal $|\Gamma_x|$ among $x\in X$. We call it $m$.

	\begin{substep}
		\textit{An upper bound of numbers of elliptic points on $C$ given by elliptic complex lines.}

		Consider all elliptic complex lines $\{L_j\}_j$ in $X$ that intersect $C$. Then by Theorem~\ref{thm:tubular_mult_bound} we can pick $D_0$ such that when $D>D_0$:
		\[
			\#\{\text{elliptic points given by elliptic complex lines}\} \le \sum_j(C\cdot L_j) \le \frac{m \vol_X(C)}{4\pi\sinh^2\left(\frac{r}{2}\right)}.
		\]
	\end{substep}
	\begin{substep}
		\textit{An upper bound of numbers of elliptic points on $C$ given by isolated elliptic points.}

		By Theorem~\ref{thm:point_mult_bound} and Proposition~\ref{prop:repulsion},  we can pick $D_0$ such that when $D>D_0$, the balls of radius $r$ centered at isolated elliptic points are disjoint, and thus
		\[
			\#\{\text{isolated elliptic points counted by multiplicity}\} \le \frac{m \vol_X(C)}{4\pi\sinh^2\left(\frac{r}{2}\right)}.
		\]
	\end{substep}

	In conclusion, we have 
	\[
		\frac{1}{4\pi}\vol_X(C) \le 2g-2+ \frac{2m \vol_X(C)}{4\pi\sinh^2\left(\frac{r}{2}\right)}
	\]
	and when $2m/\sinh^2\left(\frac{r}{2}\right) < 1$, we obtain an upper bound of $\vol_X(C)$. 
\end{step}

\begin{step}
	In the meantime, consider any point $p$ on $C$. By Theorem~\ref{thm:point_mult_bound}, we have a lower bound of $\vol_X(C)$ in terms of the injectivity radius of this point. We want this lower bound to be larger than the upper bound we just obtained when $D$ is large, which will imply that such $C$ cannot exist. 

  Suppose a point on $C$ with injectivity radius $R_0$ can lead to such a contradiction. Suppose we cannot find such a point. By Proposition~\ref{prop:inj_radius_ell} and \ref{prop:inj_radius_lox_twoball}, it means that $C$ is contained in a union of neighborhoods of elliptic points of radius $r$ for some $r$ depending on $R_0$. By repulsion of elliptic points Proposition~\ref{prop:repulsion}, when $D$ is large, $C$ is either inside the ball neighborhood of radius $r$ of an isolated elliptic point, or inside the neighborhood of a union of path connected elliptic complex lines. In either case, we pick $R$ such that $\displaystyle \frac{\cosh^2(R/2)}{m\cdot \cosh^2(r/2)}>1$. Then the relative versions of the volume inequalities (Proposition~\ref{prop:point_relative_vol} and \ref{prop:tubular_relative_vol}) give a contradiction when $D$ is large. Therefore, there is no curves of genus $g$ when $D$ is large if this curve is not an elliptic complex line. 
\end{step}

\begin{step}
	When $C$ is a component of an elliptic complex line $L$, this is similar: we still have (\ref{eq:schwarzlem}) except there is a factor of $1/|\Gamma_L|$ on the right hand side because of the generic quotient singularity of the line, and the order of the elliptic points might be different, but it is bounded by the original order. Since intersection points of elliptic complex lines are isolated elliptic points, we only have to consider isolated elliptic points on the curve. Consequently, (\ref{eq:schwarzlem}) becomes:
	\[
		\frac{1}{4\pi}\vol_X(C) \le \frac{1}{2}\Bigg[2g-2+  \frac{2m \vol_X(C)}{4\pi\sinh^2\left(\frac{r}{2}\right)} \Bigg].
	\]
	and gives an upper bound of $\vol_X(C)$ for large $D$. For the lower bound of $\vol_X(C)$, the same idea applies except that we do not need to consider tubular neighborhood by Proposition~\ref{prop:repulsion_lines}.  \qed
\end{step}

\subsection{Proof of Theorem~\ref{thm:hms_geometry_ver}}\label{sub:proof_of_theorem_ref_thm_hms_geometry_ver}

\setcounter{step}{0}
 
Follow the notation in \ref{notation_hms_smallsubfield}. Consider any irreducible curve $C\subset X$ of genus $g$. The main idea of the proof is the same as in \ref{notation_twoball}, and we will highlight the differences caused by the distinct repulsion phenomena. Again by Lemma~\ref{lem:ellclass}, there is a maximal $|\Gamma_x|$ among $x\in X$. We call it $m$. 

\begin{step}
	\textit{An upper bound of $\vol_X(C)$.}

	This step is the same as in the case of compact quotient of 2-balls. We again obtain
	\begin{equation}\label{eq:schwarzlem_hms}
		\frac{1}{4\pi}\vol_X(C) \le -\chi_\text{orb}(C) \le 2g-2+\#\{\text{elliptic points}\}.
	\end{equation}
	Then for a fixed number $R>0$ we will first show a bound of $\#\{\text{elliptic points}\}$ and then pick $R$ in terms of the genus $g$. 

	To do this, we will divide the elliptic points into two types using Proposition~\ref{prop:repulsion_hms}: for a given $R>0$, there exists $D_0(R)>0$ and $N_0(R)>0$ such that when $D>D_0(R)$, each pair of elliptic points is either far from each other with distance $>2R$, or they both lie on a special curve $F_B$ with $\norm_{K_0/\QQ}(\det B) < N_0$.  

	\begin{substep}
		\textit{An upper bound of numbers of elliptic points on $C$ given by elliptic points far from other elliptic points.}

		By the repulsion Proposition~\ref{prop:repulsion_hms}, we can add up the inequalities given by Theorem~\ref{thm:point_mult_bound_hms} for the disjoint balls of radius $R$ centered at those elliptic points when $D$ is large: 
		\[
			\#\{\text{elliptic points far from others on $C$ counted by multiplicity}\} \le \frac{m\vol_X(C)}{8\pi \sinh^2\left(\frac{R}{2}\right)} 
		\]
	\end{substep}
	\begin{substep}
		\textit{An upper bound of numbers of elliptic points on $C$ given by elliptic points on a special curve $F_B$ with $\norm_{K_0/\QQ}(\det B) < N_0(R)$.}

		By Proposition~\ref{prop:repulsion_curves_hms}, there is a $D_0(R,V_0)$ such that the tubular neighborhoods $W_{B,R}$ for all special curves $F_B$ with $\norm_{K_0/\QQ}(\det B) < N_0(R)$ are disjoint when $D > D_0(R,V_0)$ and
		\begin{align*}
			&\#\{\text{elliptic points at the intersection of $C$ and special curves with $\norm_{K_0/\QQ}(\det B) < N_0(R)$}\} \\ 
			\le& \sum_j(C\cdot L_j) \le \frac{\vol_X(C)}{8\pi\sinh^2\left(\frac{R}{2}\right)}.
		\end{align*}
	\end{substep}

	In conclusion, we have 
	\[
		\frac{1}{4\pi}\vol_X(C) \le 2g-2+  \frac{m+1}{8\pi\sinh^2\left(\frac{R}{2}\right)}\vol_X(C)
	\]
	and when $R$ is large enough, we obtain an upper bound of $\vol_X(C)$. 
\end{step}

\begin{step}
  Similar to the previous proof, suppose a point on $C$ with injectivity radius $R_0$ can lead to a contradiction. Suppose we cannot find such a point, so by Proposition~\ref{prop:inj_radius_ell} and \ref{prop:inj_radius_lox_oneball}, it means that $C$ is contained in a union of neighborhoods of elliptic points and axes of loxodromic elements of radius $r$ for some $r$ depending on $R_0$. 

  By the repulsions we listed in Section~\ref{ssub:repulsion_of_totally_geodesic_submanifolds_in_ref_notation_hms_smallsubfield}, if we fix $r,N_0,t_0$ and let $D$ be large enough, in terms of the overlaps of neighborhoods of radius $r$ of elliptic points, special curves $F_B$ with $\norm_{K_0/\QQ}(\det(B)) < N_0$, and axes of loxodromic elements $\axislox_M$ with $|\trace(M)|<t_0$ under all embeddings, we can only have the following cases (for brevity, below we omit the conditions we just mentioned):
  \begin{enumerate}
    \item $C$ is inside an isolated neighborhood of an elliptic point far from others.
    \item $C$ is inside an isolated tubular neighborhoods of axes of loxodromic elements far from others.
    \item $C$ is inside a neighborhood of a union of special curves and axes of loxodromic elements: special curves have to be far apart, but they can be connected by axes of loxodromic elements; elliptic points inside this neighborhood have to lie on some special curves; the loxodromic elements corresponding to one of these axes have to stabilize the special curves they intersect; the axes that show up are far apart unless they are in the neighborhood of a special curve. 
  \end{enumerate}

  To show a contradiction, note that the same strategy as before works for the first two cases: pick $R$ large enough and then apply Proposition~\ref{prop:point_relative_vol_hms} or \ref{prop:axis_lox_relative_vol_hms}. For the last case, it can be further divided into three cases:
  \begin{enumerate}
    \item $C$ is entirely inside the tubular neighborhood of a special curve.
    \item $C$ is entirely inside the tubular neighborhood of an axis of loxodromic element and not inside any neighborhood of special curves in the list. 
    \item $C$ has to be covered by at least the neighborhoods of a special curve and an axis of a loxodromic element. 
  \end{enumerate}
  The same proof applies to the first two cases: use Theorem~\ref{thm:special_curve_mult_bound_hms}(2) or Proposition~\ref{prop:axis_lox_relative_vol_hms}. For the last one, it means that it has a point near the boundary between the special curve $F_B$ and the axis of the loxodromic element $\axislox_M$. The injectivity radius at this point is dominated by the corresponding loxodromic element, so it has a ball neighborhood of radius $\epsilon$ bounded from below. This lower bound of radius depends on $\trace(M)$, but it also has a uniform lower bound that only depends on the $K_0$ in the setting. Now we apply theorem~\ref{thm:point_mult_bound_hms} to claim that $\vol(C\cap W_{B,r+\epsilon})$ is at least $\frac{8\pi}{m}\sinh^2\left(\frac{\epsilon}{2}\right)$. Now we apply the same argument to $F_B$ as before using Theorem~\ref{thm:special_curve_mult_bound_hms}(2).

	In conclusion, we can apply Theorem~\ref{prop:repulsion_hms} for some $R$ that is large enough to lead to contradiction in the two cases above, i.e. we can find some $D_0$ such that when $D>D_0$, $C$ cannot contain any elliptic point with large injectivity radius or any elliptic point contained in other special curve. In other words, $C$ itself must be a special curve if $X$ contains a curve $C$ of genus $g$. 

  Finiteness is a direct consequence of being special since there are only finitely many special curves with bounded degree.   
  \qed
\end{step}

\appendix

\bibliographystyle{alpha}  
\bibliography{bibl.bib}

\end{document}